\documentclass[12pt,oneside, reqno]{amsbook}
\usepackage[top=1in, left=1.1in, bottom=1.1in, right=1.1in]{geometry}
\usepackage{amsmath, amssymb, amsthm, cases,  mathtools, multicol, hyperref, subcaption}
\usepackage[usenames, dvipsnames]{xcolor}

\hypersetup{
	colorlinks=true,
	linkcolor=Cerulean,
	citecolor=Thistle
}
\usepackage{graphicx}
\usepackage{marginnote}
\usepackage{enumerate}
\usepackage{epstopdf}
\usepackage{caption}
\usepackage{dsfont}
\usepackage{cancel}
\usepackage{ragged2e}

\usepackage{fancyhdr,datetime} \usdate
\date{\today}
 \newif\ifdraft
\draftfalse
\newcommand{\mnote}[1]{{\ifdraft \marginnote{#1} \fi}}

\newtheorem{prop}{Proposition}[chapter]
\newtheorem{theorem}{Theorem}[chapter]
\newtheorem{corollary}{Corollary}[chapter]
\newtheorem{lemma}{Lemma}[chapter]
\theoremstyle{definition}
\newtheorem{definition}{Definition}[chapter]
\newtheorem{ass}{Assumption}[chapter]
\newtheorem{exmp}{Example}[chapter]
\newtheorem{remark}{Remark}[chapter]

\newcommand{\HIDE}[1]{} 

\newcommand{\R}{\mathbb{R}}

\newcommand{\N}{\mathbb{N}}

\newcommand{\f}{\ensuremath{f^{n+1}}}
\newcommand{\ft}{\ensuremath{\Tilde{f}}}
\newcommand{\p}{^{p+1}}
\newcommand{\pp}{^{p-1}}

\newcommand{\ut}{\ensuremath{\Tilde{u}}}
\newcommand{\un}{\ensuremath{u^{n+1}}}
\newcommand{\unn}{\ensuremath{u^{n-1}}}

\newcommand{\dg}[1]{\backslash #1\backslash} 
\newcommand{\kl}[2]{\int \log\frac{#1}{#2} {#1} dx}
\newcommand{\klr}{\rm KL(\rho|\pi)}

\DeclareMathOperator{\gd}{\rm grad\!}

\title
{Convergence problems in Nonlocal dynamics with nonlinearity}
\author{Won E Hong}

\begin{document}
\begin{center}
\thispagestyle{empty}

{\Large \mdseries 
CONVERGENCE PROBLEMS IN NONLOCAL DYNAMICS\\ \vspace{3mm}
WITH NONLINEARITY
}

\vspace{7mm}
by

\vspace{7mm}

\begingroup
{\large \sc Won Eui Hong}
\endgroup




\vfill

\textit{A dissertation submitted in partial fulfillment of the requirements for\\
the degree of Doctor of Philosophy in Mathematical Sciences}

\vfill

Last update: February 22, 2022

\vfill

\def\svgwidth{0.3\columnwidth}
\begingroup%
  \makeatletter%
  \providecommand\color[2][]{%
    \errmessage{(Inkscape) Color is used for the text in Inkscape, but the package 'color.sty' is not loaded}%
    \renewcommand\color[2][]{}%
  }%
  \providecommand\transparent[1]{%
    \errmessage{(Inkscape) Transparency is used (non-zero) for the text in Inkscape, but the package 'transparent.sty' is not loaded}%
    \renewcommand\transparent[1]{}%
  }%
  \providecommand\rotatebox[2]{#2}%
  \ifx\svgwidth\undefined%
    \setlength{\unitlength}{240.0000071bp}%
    \ifx\svgscale\undefined%
      \relax%
    \else%
      \setlength{\unitlength}{\unitlength * \real{\svgscale}}%
    \fi%
  \else%
    \setlength{\unitlength}{\svgwidth}%
  \fi%
  \global\let\svgwidth\undefined%
  \global\let\svgscale\undefined%
  \makeatother%
  \begin{picture}(1,1)%
    \put(0,0){\includegraphics[width=\unitlength,page=1]{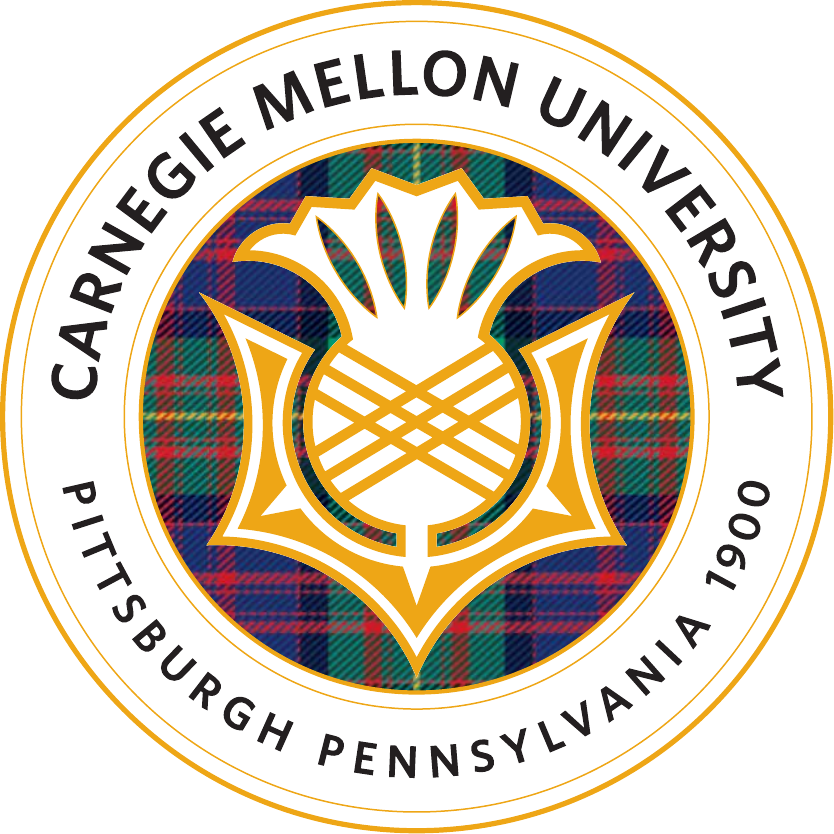}}%
  \end{picture}%
\endgroup%

\vfill

Department of Mathematical Sciences\\
Carnegie Mellon University\\
Pittsburgh, PA

\vfill

\textbf{Doctoral thesis committee}

\begin{tabular}{ccccc}
\textsc{Robert L. Pego (CMU, chair)} \\ 
\textsc{David Kinderlehrer (CMU)}  \\
\textsc{Dejan Slep\v cev (CMU)} \\
\textsc{Hailiang Liu (ISU)} 
\end{tabular}
\end{center}

\chapter*{Acknowledgments}
Looking back on my life as a Ph.D. student at Carnegie Mellon University, in Pittsburgh, I have received so much love and support from many people.

First and foremost, I appreciate my parents, Joseph Jichul Hong and Yu Soon Kim, who have been sacrificing themselves to raise me well. They have always been loving, encouraging, and supportive. As much as they are of me, I am immensely proud to be their son.

I cannot stress enough the support and care that my advisor Professor Robert Pego has provided me. Despite his extraordinary intelligence and academic success, he has always been humble and curious-minded. He taught me not only mathematical skills but also how to think intuitively and, at the same time, thoroughly. He has been my role model, and I am very much honored to have learned so closely from him. I also wish to express my sincere gratitude to the thesis committee, Professor David Kinderlehrer, Professor Dejan Slep\v cev, Professor Hailiang Liu. They have provided an incredible amount of support, kindness, and patience. They are such exceptional mathematicians, and I am very grateful to have had a chance to interact with each one of them.

I want to thank my family. My grandfather, Sung Yoon Hong, is one of the most diligent and energetic people I know. He has survived two wars and founded my family from the bare ground. I acknowledge my privilege built upon the trouble that he had undergone. He did not have a chance to go to school, and it is meaningful to hear acclaim from him that I am the first generation Ph.D.. I hope this will help my family survive as immigrants in the United States. My family moved to the United States in December 2014, and I cannot emphasize enough the love and support from my uncle, Kyu Hong Chang, and my aunt, Young Chang, my cousins Charles K. Chang and Brian K. Chang, who helped me settle down in the USA. I also thank my brothers, Taeho Hong and Jung Eui Hong, for not causing too much trouble and doing relatively well. I also want to express my appreciation and love towards my family members who live far apart, and I miss much: Jisoon Hong, Jisook Hong, Chanyoung Park, Haeun Park, Minkyoung Gam, Youngseo Kim, Youngwhan Kim, Wonhee Noh, Sangbum Noh, Hyunjeong Kim, Junyeob Kim.

Friendship is another trophy that I acquired at Carnegie Mellon University. I have been surrounded by a fun, loving, and supportive group of people. They have been helping me reach the finish line, and I hope we keep encouraging each other in the race of life. I thank Ananya Uppal (WH7110 forever), Senthil Purushwalkam Shiva Prakash, Giovanni Gravina, Lily ChinWen Yu, Marcos Mazari-Armida, Ana Paula Vizcaya, Truong-Son Van, Han Nguyen, Antoine Rémond-Tiedrez, Kayla Bollinger, Landon Settle, Kerrek Stinson, Mihir Hasabnis, Da Qi Chen, Junichi Koganemaru, Sittinon Jirattikansakul, David Itkin, Likhit Ganedi, Michael Anastos, Debsoumya Chakraborti, Ilqar Ramazanli, Charles Joseph Argue, Greg Kehne, Weicheng Wilson Ye, Mochong Duan, Xiaofei Shi, Linan Zhang, Forrest Shetley Miller, Andrew Warren, Oleksandr Rudenko, Tony Johansson, Sam Cohn, Adrian Hagerty, Clive Newstead, Joseph Briggs, Xiao Chang, Andy Zucker, Sangmin Park, Wesley Caldwell, Anish Sevekari, Aditya Raut, Jongwha Park, Seungjae Son, Zoe Wellner, Benjamin Weber, Pedro Marun, Yue Pu, Ryan Xu, Kevin Ou, David Huckleberry Gutman, Daniel Rodríguez, Slav Kirov, Matteo Rinaldi, Riccardo Cristoferi, Janusz McBroom, Jisu Kim, Dongnam Ko.

I also need to point out the generous support of the faculty and staff in the Department of Mathematical Sciences. I want to express my gratitude to William Hrusa, Jack Schaeffer, Noel Walkington, Hayden Schaeffer, Giovanni Leoni, Ian Tice, Martin Larsson, Franziska Weber, Clinton Conley, Gautam Iyer, Dmitry Kramkov, Deborah Brandon, Tom Bohman, Stella Andreoletti, Jeff Moreci, Charles Harper, Christine Gilchrist, Nuno Chagas. I also thank the visiting faculty members who has been encouraging me: KiHyun Yun, David Shirokoff.

I have been lucky enough to meet great people outside of the math realm. I want to shout out the name of people who have been keeping me sane and balanced. I sincerely thank Emanuel Rackard, Derek Terrell, Richard Teaster, Thomas Douglas, Charles Wu, Joshua Bow, Aria Yuan Wang.

Finally, I would like to thank the people who have wished me the best since I moved to the USA. Not in any particular but rather chronological order, here is the list of their name:
(Yonsei Univ:  Honggu Im, Woo Young Jung, Hyun Kwang Lim, Jungwoo Park, Seungwon Kim, Hyojin Jeong, Ikjoo Im, Seungchan Ko, Hyuk An, Yongnam Kim, Yongseok Choi, Haseo Ki, Joonil Kim, June Bok Lee),
(Sangsan high: Haewon Lee, Doeon Kim, Seungwoo Hong, Dongjae Lee, Jaegon Ye),
(Baejae mid: Eunjin Yoon, Sungho Park, Dong Gun Kim, Seungho Ha), Hosub James Jang.
\\

This material is based upon work supported by the National Science Foundation under Grants DMS-1812609 and DMS-2106534.
\\
\begin{FlushRight}
February, 2022
\\
Pittsburgh, Pennsylvania
\end{FlushRight}

\chapter*{Abstract}
We study nonlocal nonlinear dynamical systems and uncover the gradient structure to investigate the convergence of solutions. Mainly but not exclusively, we use the \L ojasiewicz inequality to prove convergence results in various spaces with continuous, or discrete temporal domain, and  finite, or infinite dimensional spatial domain. To be more specific, we analyze Lotka-Volterra type dynamics and concentration-dispersion dynamics.

Lotka-Volterra equations describe the population dynamics of a group of species, in which individuals interact either competitively or cooperatively with each other. It is well-known that Lotka-Volterra equations form a gradient system with respect to the Shahshahani metric. The Shahshahani metric, unfortunately, becomes singular in the scenario where some species become extinct. This singular nature of the Shahshahani metric is an obstacle to the usual convergence analysis. Under the assumption that the interaction between species is symmetric, we present two different methods to derive the convergence result. One, the entropy trapping method, is to adapt the idea of Akin and Hofbauer (Math. Biosci. 61 (1982) 51–62) of using monotonicity of the energy to bound the entropy, which provides the proximal distance of the solution from the desired equilibrium. Another method, inspired by Jabin and Liu's observation in (Nonlinearity 30 (2017) 4220) is to change variables to resolve the singular nature of gradient structure. We apply this idea to show the convergence result in generalized Lokta-Volterra systems, such as regularized Lotka-Volterra systems and nonlocal semi-linear heat equations, which can be seen as an infinite dimensional Lotka-Volterra equations with mutation.

Concentration-dispersion dynamics is a new type of equation that is inspired by fixed point formulations for solitary wave shapes. The equations are designed in a way that the solution evolves to match the shape of a concentrated and dispersed version of the solution, which is an outcome of power nonlinearity and convolution.
As a continuous time analogue of Petviashvili iteration, we aim to dynamically calculate the nontrivial solitary wave profile.
We show the well-posedness and compactness of the solutions. Moreover, we suggest specific initial data which stay away from the trivial equilibrium. Using the gradient structure we deduce the existence of nontrivial solitary and periodic wave profiles. Unfortunately, the convergence result remains open. Instead, we regularize the concentration dispersion equations by adding a nonlinear diffusive term, and prove the convergence of the solutions by using the \L ojasiewicz convergence framework. In this way, we can dynamically approximate nontrivial wave profiles with arbitrarily small error.
\tableofcontents
\addtocontents{toc}{\protect\setcounter{tocdepth}{1}}
\chapter*{Overview}
\paragraph{\textbf{Introduction}} A gradient flow is a dynamical system, where the particle follows the steepest direction to decrease a given objective function, which is often called energy. Once the trajectory of the particle is bounded, it is commonly expected that the particle will stop at a critical point of the energy. Contrary to this popular belief, however, a bounded solution of gradient flow can fail to converge, even if the associated energy is smooth. 

The idea for a such nonconvergence example was suggested by Haskell B. Curry (1944), who proposed an energy resembling an infinite spiral staircase with decreasing step size. In this way, the particle can move around a circle slowly and ceaselessly. A concrete example involving so called the ``Mexican hat" function, was constructed later by Palis and DeMelo (1982) ~\cite{palis2012geometric} and Absil, Mahony and Andrews (2005)~\cite{absil2005convergence}. This behavior is mainly enabled by lack of any relationship between the speed of the particle and the energy level where the particle is at. 

In 1965, \L ojasiewicz proved an inequality for analytic energy functions, that regulates the energy level with the speed of the particle, or the slope of the energy. Using the \L ojasiewicz inequality, energy dissipation can be transformed to bound the length of the trajectory, and consequently any precompact solution of a gradient system with analytic energy converges. This convergence analysis can be, in fact, applied to gradient-``like" systems where the velocity of the particle loosely follow the gradient of the associated energy, satisfying a certain angle condition.

This convergence property comes in handy in optimization, but practically what we need is appropriate discrete-time version of gradient descent in order to calculate the local minimizer.
In 2005, Absil, Mahony and Andrews~\cite{absil2005convergence}, as an extension of \L ojasiewicz convergence theorem, showed that a discrete iteration, which satisfies the angle condition with an analytic cost function, produces a convergent sequence as long as it is bounded. 
The flexibility of the angle condition presents innumerable optimization algorithms so one can investigate how to design the efficient iteration method for a given cost function. The most basic and simple method, Euler's method, performs well when the time step is small enough.

Generally in time-discrete schemes, extra treatment, such as choosing the optimal time step, is necessary to achieve the monotonicity of the energy. Considering it is the energy that drives the dynamics, it would be more reasonable to cook up a time-discrete equation so that the energy is monotonic along the sequence, i.e., energetically stable.
Pham Dinh Tao, in 1985, introduced an energetically stable optimizing algorithm and named it the Difference of Convex functions algorithm (DCA). It is an iterative algorithm for finding a fixed point of the gradient of two convex functions of which the difference is the objective function. In fact, the family of difference of convex functions is quite extensive that it has been applied to numerous nonconvex optimization problems. It was later in 2018~\cite{le2009convergence} when the convergence of the DCA was proven using \L ojasiewicz inequality. 
Due to the flexibility in convex splitting, DCA can in fact be comprehensive, including the classical methods in convex analysis and proximal methods~\cite{tao1997convex}.

The idea of using convex inequality to formulate energetically stable time-discrete equation was also developed independently in the area of partial differential equations.
In 1998, David J. Eyre came up with an idea of representing the energy as a difference of two convex functions, to form a time-discrete Cahn-Hilliard equation and showed the produced sequence well-approximates the continuous-time solution. By treating the convex part implicitly and the concave part explicitly, this semi-implicit Euler's method becomes energetically stable. In fact the semi-implicit Euler's method can be seen as the special case of DCA.

A few derivations of DCA can be discussed.
Since the implicit part of DCA is a convex optimization problem, the convex dual of DCA naturally arises. It is known that the dual DCA, or the DCA on the difference of dual of convex functions, is equivalent to the primal DCA. In order to boost the speed of convergence, we can group momentum methods into two types, whether the momentum is added externally or internally. The main and original examples of external, or internal addition of momentum on gradient descent are Polyak's momentum method, or Nesterov's acceleration method, respectively. As an extension of those, we present the DCA with momentum, which has slightly more flexibility on the definition of momentum. From this point of view, Polyak's momentum method and Nesterov's acceleration method can be regarded as the same algorithm on primal and dual energy, respectively.

\paragraph{\textbf{Subject of the thesis}}
Of course, not every dynamical system is a gradient system. However, it is known that if a system has a strict Lyapunov function, there exists a metric to represent the system as a gradient system~\cite{barta2012every}. So finding an appropriate Lyapunov function and a metric often provides break-through for analyzing a dynamical system. One of the revolutionary examples is Jordan-Kinderlehrer-Otto's characterization of the Fokker-Planck equation as Wasserstein gradient flow of KL-divergence. 

In this thesis we study convergence problems in nonlocal nonlinear dynamical systems of two main types: Lotka-Volterra type dynamics and concentration-dispersion dynamics.

\paragraph{\textbf{Lotka-Volterra type dynamics}}
Proposed, originally, as a model for population dynamics, Lotka-Volterra represents a fundamental model of nonlinear nonlocal system. It is well known that when the interaction matrix is symmetric, a Lotka-Voltera system is a gradient flow of a quadratic energy with respect to the so-called Shahshahani metric, which provides the inverse of population as weights on each species. Along the evolution it is possible that some species die out to make the Shahshahani metric to blow up. This singular structure is an obstacle for performing gradient analysis on Lotka-Volterra systems. 

When the quadratic energy is convex, it is known that the Kullback–Leibler divergence (KL divergence), relative to the unique equilibrium, is decreasing. Since KL divergence provides proximal distance from the equilibrium, the monotonicity of the KL-divergence has been used to deduce global convergence result in Lotka-Voltera type systems ~\cite{hofbauer2003evolutionary}~\cite{attouch2004regularized}~\cite{liu2015entropy}~\cite{jabin2017non}.

Our goal is, without assuming convexity of the energy, to prove convergence of bounded solutions. 
We provide two different ways to prove convergence of solutions as $t\to\infty$. First approach is to adapt the entropy trapping method of Akin and Hofbauer when they proved the convergence of solutions in replicator dynamics~\cite{akin1982recurrence}. Although the KL-divergence, or the entropy, is no longer monotonic without convexity assumption, it still plays a role of distance-like function from the equilibrium. The idea is to use the monotonicity of the energy to capture the possible oscillation of KL-divergence.

Another is to de-singularize the metric. This idea was suggested by Jabin and Liu for studying Lotka-Volterra equations with mutation~\cite{jabin2017non},
\begin{equation*}
    \partial_tu(t,x) =\Delta u(t,x) +\frac{1}{2}u(t,x)\left( a(x)-\int_\Omega b(x,y)u^2(t,y)dy\right).
\end{equation*}
They studied infinite dimensional Lotka-Volterra systems with a diffusion term, which represents mutation. Especially, the suggested equation has a quadratic competition term, rather than a linear, which makes the equation as a $L^2$ gradient flow of quartic energy in weak sense. Furthermore, they showed that one can change variables to recover the original form of Lotka-Voltera equations with linear competition.

The method of de-singularizing Shahshahani metric can be used not only in the original Lotka-Volterra dynamics, but also in regularized Lotka-Volterra type systems,
\begin{equation*}
     f'_i=-\frac{f_i}{\mu+\nu f_i}\nabla H(f)_i.
\end{equation*}
As a generalization of Lotka-Volterra equation, Attouch and Teboulle suggested regularized Lotka-Volterra dynamics, which is a gradient system of a smooth energy $H$ with respect to Shahshahani-type metric~\cite{attouch2004regularized}. 

In both variations of Lotka-Volterra~\cite{jabin2017non},~\cite{attouch2004regularized}, the authors proved convergence of the solution, provided the energy is convex, using suitable KL-divergences as Lyapunov functions. We assume only that the energy satisfies the \L ojasiewicz inequality to show convergence of bounded solutions by applying \L ojasiewicz convergence analysis on the de-singularized structure.

We also formulate energetically stable time-discrete Lotka-Volterra type equations using convex splitting and prove convergence of iterative algorithms by applying the entropy trapping method, and by de-singularizng the metric. This relates to practical uses of the Lotka-Volterra systems. For example, Hovsepian and Anslemo~\cite{hovsepian2011supervised} suggested supervised learning algorithms based on Lotka-Volterra systems, where they postulate the convergence of the solution, or make rather strong assumptions on the interaction matrix such as diagonal dominance.

\paragraph{\textbf{Concentration-dispersion dynamics}}
Solitary waves, physical phenomena of waves that travels at constant speed and shape, were inexplicable by the theories of hydrodynamics at the time when they were discovered by J. Scott Russell in 1834.
The first mathematical treatment of solitary waves, the KdV equation, was founded in 1877 by Joseph Valentin Boussinesq and actively rediscovered by Diederik Korteweg and Gustav de Vries in 1895. Since then the phenomena have been noticed in various nonlinear dispersive physical models, e.g., the sine-Gordon equation, the nonlinear Schr\"odinger equation etc.,~\cite{friesecke1994existence}, \cite{Herrmann2010UnimodalWA}, \cite{Pego2018ExistenceOS}. Now it is understood that solitary waves is created when the nonlinear and the dispersive effects are balanced, but exact analytic solutions of solitary wave are known to be challenging to compute.

In 1976, Petviashvili successfully created an algorithm that is empirically known to converge, but the mathematical verification on global convergence is still unknown.
Motivated by fixed point formulations of several kinds of solitary wave equations, we study integro-differential equations of $u(x,t):t\mapsto u(x)\in H^1(\R)$ as follows,
\begin{equation}\label{e: cdd0}
    \left\{
    \begin{aligned}
        \;\;\frac{d}{dt}u(t,x)&=K\ast u^p  - c(t) u\\
        c(t)&=\int u^p K\ast u^p dx.
    \end{aligned}
     \right.
    \end{equation}
where $*$ indicates a convolution operator and $p>1$.
The equation~\eqref{e: cdd0} is designed in a way for $u$ to chase after normalized $K\ast u^p$. The nonlinearity $u^p$ has an effect of concentration at peaks and the convolution by $K$ has an effect of dispersion. We can expect $u$ to stop evolving when the dispersive effect cancels the nonlinear concentration, just as how a solitary wave is created.

The equation~\eqref{e: cdd0}, in fact, can be seen as a gradient flow on an invariant manifold of $(p+1)$-th roots of probability densities with respect to the metric weighted by $u^{p-1}$, and the associated energy is a distance between $K\ast u^p$ and $u$ given by the metric.
On the other hand, if we initiate the evolution outside of the manifold, the associated energy can lose the monotonicity. It turns out we can stabilize the energy by multiplying it by an exponential factor related to the manifold. With the new energy, the dynamics is a gradient flow in the flat whole space, which makes analysis far easier.

Unlike usual dispersive effects in PDE does, the convolution does not provide smoothing in time. For example, if the initial data has a singularity or a spike, the solution at any time will have the same. However, we can show that the difference between the solution and scaled initial data is smoother than the initial data. So the solution is precompact in $H^1$ on compact domains. On the real line, we show that the solution initiated from ``bell-shaped" initial data is precompact in $L^2$.

Recall that we aim to approximate nontrivial wave profiles. Even if we prove the dynamics converges, it will be meaningless if every solution converges to trivial solutions, i.e., constant solutions. We show that the constant equilibria are unstable. Moreover, we can show that the peak of the solution stays strictly away from $0.$
Thus, by LaSalle's invariance principle type argument, we deduce the existence of nontrivial equilibria of ~\eqref{e: cdd0}.

Despite the $H^1$-analyticity of the energy, unfortunately, \L ojasiewicz convergence framework cannot be applied to ~\eqref{e: cdd0}. To resolve the issue, we regularize the concentration-dispersion dynamics as follows.
\begin{equation}\label{e: rcdd0}
    \left\{
    \begin{aligned}
        \;\;\frac{d}{dt}u(t,x)&=\epsilon (u^p)_{xx} +K\ast u^p  - c(t) u\\
        c(t)&=\int u^p K\ast u^p- \epsilon[(u^p)_x]^2 dx.
    \end{aligned}
     \right.
    \end{equation}
Indeed, the equation~\eqref{e: rcdd0} also has nice $L^2$-gradient structure, and furthermore, we prove that every solution of \eqref{e: rcdd0} on compact domains converges. Also by showing the instability of the constant solution, we can approximate nontrivial periodic waves with arbitrarily small error $\epsilon>0$.
\\
\paragraph{\textbf{Plan of the thesis}}
This thesis comprises three parts. Part 1 is an overview for the \L ojaisewicz convergence theorem on time-continuous (Chapter 1) and time-discrete (Chapter 2) gradient system. We discuss adaptation of DC algorithm, originated from optimization algorithm, to a time-discrete differential equation in Chapter 2.2. Additionally we provide novel perspectives on the \L ojaisiewicz framework on infinite dimensional space (Chapter 1.2) and DC algorithms with momentum (Chapter 2.3).

Part 2 is dedicated to convergence analysis on Lotka-Volterra type systems. We provide two different approaches to prove convergence for the continuous time equations (Chapter 3) as well as the discrete time equations (Chapter 4). Entropy trapping method can be found in Chapter 3.1 and Chapter 4.1. The convergence analysis using de-singularized gradient structure can be found in Chapter 3.2 and Chapter 4.2. As derivations of Lotka-Volterra systems, Lotka-Volterra systems with mutation are treated in Chapter 3.3 and time-discrete regularized Lotka-Volterra equations are treated in Chapter 4.3.

In Part 3, we introduce a new type of integro-differential equations, namely concentra\-tion-dispersion dynamics. We discuss the motivation and the gradient structure of the concentration-dispersion equations in Chapter 5. In Chapter 6, we perform analysis to show well-posedness, compactness of solutions, and existence of nontrivial wave profiles. In Chapter 7, we study the regularized concentration-dispersion dynamics and prove convergence of the solutions to nontrivial wave profiles.

\part{\L ojasiewicz inequality and gradient descent}
\chapter{Introduction to gradient flows}\label{ch: gf}
\mnote{need to mention nature of models}
In this thesis, we reveal gradient structures for various types of nonlinear and nonlocal dynamics to establish convergence of solutions in continuous-time and discrete-time, without assuming convexity of the associated energy. Our main tool will be the \L ojasiewicz inequality for analytic cost functions. In this chapter, together with a brief history, we will go over how the \L ojasiewicz inequality is used to prove convergence of solutions of gradient-like systems on finite and infinite dimensional domains.

    We say $u(t,x):\R_+\times\R^N\to\R^N$ is a gradient flow for an energy $E \in C^1(\mathbb{R}^N)$, if $u$ satisfies the following differential equation,
	\begin{equation}\label{e: gs}
		\frac{du}{dt}+\nabla E(u) = 0.
	\end{equation}
	For brevity, we let the time differentiation $d/dt$ be denoted by $'$.

	First of all, along the evolution of the solution $u(t)$ of \eqref{e: gs}, we notice the energy is monotonically nonincreasing.
	For $t \geq 0$, 
	\begin{equation}\label{e: e diss}
	\frac{d}{dt} E(u(t)) = \langle \nabla E(u), u'\rangle = -\Vert u' \Vert ^2 \leq 0,
	\end{equation}
	in other words, the energy $E$ is a Lyapunov function for the gradient system \eqref{e: gs}.
	
	Since $u$ moves to minimize $E,$ one can naturally hope for $u$ to approach to a critical point of $E$. Indeed, we can check this as follows. Let $\mathcal{E} := \{a \in \mathbb{R}^N: \nabla E(a) = 0\}$ be the set of critical points of the energy $E$ and assume that $u$ is bounded. Then, from the energy dissipation~\eqref{e: e diss}
	\begin{equation*}
	    \int_{0}^{\infty}\Vert u'(s)\Vert^2ds= E(u(0))-\lim_{t\to\infty} E(u(t))<\infty,
	\end{equation*}
	i.e.,  $u'\in L^2(\R;\R^N)$.
	Note that by continuity of $\nabla E$ and \eqref{e: gs}, every bounded trajectory has uniformly continuous derivative $u'$, which implies for any $\epsilon>0$, we can choose $\delta>0$ such that $\forall h\in(0,\delta),$
	\[
	\big\vert\Vert u'\Vert(t+h) - \Vert u'\Vert(t) \big\vert < \epsilon.
	\]
	By averaging on $(0,\delta)$,
	\begin{align*}
	    \Vert u'\Vert(t) - \epsilon &< \frac{1}{\delta}\int_{t}^{t+\delta} \Vert u'(s)\Vert ds\\ 
	    &\leq \frac{1}{\sqrt{\delta}}\int_{t}^{t+\delta}\Vert u'(s)\Vert^2 ds.
	\end{align*}
	Since $u'\in L^2(\R;\R^N)$,
	$$\limsup_{t \rightarrow \infty}\Vert u'(t)\Vert \leq \epsilon$$
	$$\lim_{t\rightarrow\infty}\Vert u'(t)\Vert = 0.$$ 
	Consequently,
	$$ \lim_{t\rightarrow\infty} \text{dist}(u(t),\mathcal{E})=0.$$ 
	The question naturally arises whether $u(t)$ attains a limit point. On $\R$ this holds true since all solutions are monotonic. However when the spatial dimension $N$ is larger, convergence does not hold in general, even if the energy is $C^\infty$. One can imagine a gradient solution that circles around the unit sphere, not allowing the solution to settle down to a limit point. This idea was suggested by Haskell B. Curry in 1944 and the following concrete example was constructed by Palis and de Melo much later in 1982.
	\begin{exmp}[$C^\infty$ non-convergence example~\cite{palis2012geometric}, p14]
	\begin{equation*}
    E(r\cos\theta,r\sin\theta) = 
    \left\{\begin{array}{ll}
        e^{1/(r^2-1)}, & \text{if} \;\; r<1;  \\
        0 & \text{if}\;\; r=1; \\
        e^{-1/(r^2-1)}\sin(1/(r-1)-\theta) & \text{if}\;\; r>1.
    \end{array}\right.
    \end{equation*}
	\end{exmp}
\section{\L ojasiewicz convergence theorem}
We review the \L ojasiewicz inequality, which regulates ``flatness" of analytic energies around critical points by the energy level. It provides quantitative convergence results for analytic gradient flows without any convexity assumption. In fact, there is a much larger class of functions that satisfies the \L ojasiewicz inequality, according to Bolte et al.~\cite{bolte2007lojasiewicz}.
Despite its impact in the study of dynamical systems, the importance of the \L ojasiewicz inequality was discovered much later, in the early nineties. It is known that the \L ojasiewicz inequality is a key to solve several conjectures in algebraic geometry~\cite{colding2014lojasiewicz},\cite{colding2017arnold},~\cite{lojasiewicz1963propriete},~\cite{law1965ensembles}.

 \begin{theorem}[\L ojasiewicz, 1963]\label{t: Loj in}
    Let E : $\mathbb{R}^N \longrightarrow \mathbb{R}$ be an analytic function. Then for all $a\in \mathcal{E}:=\{a\in \mathbb{R}^N: \nabla E(a) = 0\}$, there exists $c_a>0,\;\; r_a>0$ and $0<\theta_a \leq \frac{1}{2}$ such that,
    \begin{equation}\label{loja ineq}
        c_a|E(u)-E(a)|^{1-\theta_a}\leq \Vert\nabla E(u)\Vert, \qquad \forall u\in\mathbb{R}^N \;\; \Vert u-a\Vert<r_a.
    \end{equation}
\end{theorem}
It is worth mentioning that the \L ojasiewicz exponent $\theta_a$ is related to the order of vanishing of $(E(u)-E(a))$ at $a$. Roughly, if $E(u)-E(a)\approx O(\Vert u-a\Vert^n)$ for some $n>0$, then $\Vert\nabla E(u)\Vert\approx O(\Vert u-a\Vert^{n-1}) $, so that $(E(u)-E(a))^{1-1/n}\approx \Vert\nabla E(u)\Vert$. It is quite easy to see when $N=1$, since one can use power series expansion to prove the theorem. The general proof relies on theorems of analytic geometry~\cite{lojasiewicz1963propriete},~\cite{law1965ensembles},~\cite{kurdyka1994wf},~\cite{lojasiewicz1999gradient}.
 
  Note that the analyticity assumption is necessary. Let's consider
    \begin{equation*}
        f(x) = \left\{
        \begin{array}{ll} 
        e^{-\frac{1}{|x|}} & \;\; \text{for}\;\; x\neq0\\
        0& \;\; \text{for}\;\; x=0
        \end{array}\right.
    \end{equation*}
    Then for any $y$ in the zero neighborhood $|f'(y)|=\frac{1}{y^2}e^{-\frac{1}{|y|}}$ 
    \begin{equation*}
    \begin{aligned}
        \limsup_{y\rightarrow0}\frac{|f(y)|^\theta}{f'(y)} &= \limsup_{y\rightarrow0}y^2\frac{|f(y)|^\theta}{f(y)}\\ 
        &=\limsup_{y\rightarrow0}y^2e^{\frac{1-\theta}{y}}\longrightarrow\infty  
    \end{aligned}    
    \end{equation*}
    which fails to make (\ref{loja ineq}).
    
    Now we present the \L ojasiewicz convergence theorem. We will take a look at the proof, in the next section regarding the infinite dimensional case, to see how the \L ojasiewicz inequality is used.
    \begin{theorem}[\L ojasiewicz Theorem] Assume that $E$ satisfies (\ref{loja ineq}) at any equilibrium point and let $u\in L^\infty(\mathbb{R}^+,\mathbb{R}^N)$ be a solution of the gradient system of $E$~\eqref{e: gs}. Then there exists $\Tilde{u}\in\mathcal{E}$ such that
    $$\lim_{t\rightarrow\infty}\Vert u(t) - \Tilde{u} \Vert = 0$$
    Moreover, let $\theta$ be any \L ojasiewicz exponent of $E$ at point $\Tilde{u}$. Then we have
    \begin{equation}\label{e: L conv rate}
        \Vert u(t) - \Tilde{u}\Vert = 
        \left\{\begin{array}{ll}
             O(e^{-\delta t}) & \text{if}\;\; \theta=\frac{1}{2}, \;\;\;\;\text{for some}\;\; \delta>0 \\
             O(t^{-\theta/(1-2\theta)})& \text{if}\;\; 0<\theta<\frac{1}{2} 
        \end{array}\right.
    \end{equation}
    In particular if $E$ is analytic, all bounded solutions of gradient systems are convergent.
    \end{theorem}
\mnote{comment about angle cond}
    \section{Infinite-dimensional setting}\label{c: inf dim}
    In 1983, Leon Simon successfully extended the \L ojasiewicz inequality into Hilbert spaces in order to treat semi-linear parabolic equations~\cite{simon1983asymptotics}. The main idea is to represent the Hilbert space as a direct product of the kernel and the range of the Hessian operator of the analytic energy. Since the nullity of the Hessian operator represents the degree of degeneracy of the gradient flow, the energy needs further restriction to satisfy \L ojasiewicz inequality: namely, the Hessian is semi-Fredholm. 
    
    In this section, we look over Simon's work on the infinite dimensional \L ojasiewicz inequality.
    We will use the result in the later chapters to deduce convergence results.
    \mnote{describe more?}
    \paragraph{\textbf{Analytic functions on Banach spaces}} The following information about analytic functions on Banach space can be found on \cite{haraux2015convergence}.
\begin{definition}[Analyticity]
       Let $X,Y$ be two real Banach space and $a\in X$. Let $U$ be an open neighborhood of $a$ in $\mathcal{V}$. A map $f:U\to Y$ is called \emph{analytic} at $a$ if there exists $r>0$ and a sequence of $n$-linear, continuous, symmetric maps $(M_n)_{n\in\N}$ fulfilling the following conditions:
\begin{enumerate}
    \item $\sum_{n\in\N}\Vert M_n\Vert_{\mathcal{L}_n(X,Y)}r^n<\infty$ where 
    $$\Vert M_n\Vert_{\mathcal{L}_n(X,Y)}:=\sup\{\Vert M_n(x_1,x_2,\dots,x_n)\Vert_Y, \sup_i\Vert x_i\Vert_X\leq 1\},$$
    \item $\overline{B}(a,r)\subset U$,
    \item $\forall h\in \overline{B}(0,r)$, $f(a+h)=f(a) + \sum_{n\geq1}M_n(h^{(n)})$ where $h^{(n)}=\underbrace{(h,\dots,h)}_{n \text{\;times}}$.
\end{enumerate}
Moreover, f is \emph{analytic} on the open set $U$ if $f$ is analytic at every point of $U$.
\end{definition}

The following properties will be useful to check analyticity of functionals.
\begin{theorem}[Composition of Analytic Maps]
    Let $Z$ be a Banach space, $V$ be an open neighborhood of $f(a)$ and $g:V\to Z$ be analytic at $f(a)$. Then the map $g\circ f$ is analytic at $a$ with values in $Z$.
\end{theorem}

\begin{prop}\label{p: fdf anal}
    Let $f\in C^1(U,Y)$. The following properties are equivalent.
    \begin{enumerate}
        \item $f:U\to Y$ is analytic.
        \item $Df:U\to \mathcal{L}(X,Y)$ is analytic.
    \end{enumerate}
    Moreover if
        \[
        f(a+h)=f(a)+\sum_{n\ge 1} M_n(h^{(n)})
        \]
        is the expansion of $f(a+h)$ for all $h$ in the closed ball $\Bar{B}(0,r)\subset U-a$ then 
        \[
        Df(a+h) =M_1 +\sum_{n\ge 2}nM_n(h^{(n-1)},\cdot)
        \]
        is the expansion of $Df(a+h)$ for all $h$ in the open ball $B(0,r)$.
\end{prop}

\subsection{\L ojasiewicz inequality in infinite dimensional spaces}
We aim to study the idea how the \L ojasiewicz convergence theorem is introduced to infinite dimensional systems.
Throughout this chapter, we consider
two real Hilbert spaces $\mathcal{V}$, $\mathcal{H}$ with the following hierarchy
\begin{equation*}
    \mathcal{V}\subset \mathcal{H}=\mathcal{H}'\subset \mathcal{V}'
\end{equation*}
with continuous dense embedding and $\mathcal{H}'$, the topological dual of $\mathcal{H}$, is identified with $\mathcal{H}$.

    Chapter 11 of~\cite{haraux2015convergence} starts with two examples of analytic functions which do not admit a \L ojasiewicz inequality. Each case points out different issues on the Hessian operators of the energies: 1. the range of the Hessian is not closed, 2. the kernel of the Hessian is infinite dimensional. Thus it is natural to assume the Hessian to be \textit{semi-Fredholm}.
    \begin{definition}
       Let $W$, $Z$ be Banach spaces and $A\in \mathcal{L}(W;Z)$ a bounded linear operator. We say $A$ is \emph{semi-Fredholm} if 
       \begin{enumerate}
           \item $N(A)$ is finite dimensional, and
           \item $R(A)$ is closed.
       \end{enumerate}
    \end{definition}
    The following theorem will come in handy.
    \begin{theorem}[Theorem 2.2.5 of~\cite{haraux2015convergence}]\label{t: fredcom}
        Let $A:W\to Z$ be a semi-Fredholm and $G:W\to Z$ be a compact operator. Then $A+G:W\to Z$ is a semi-Fredholm operator.
    \end{theorem}
    Conceptually, the kernel of the Hessian contains degenerate directions of the gradient flow. With the semi-Fredholm condition, the Hilbert space can be split into the direct product of two spaces: one where the Hessian is invertible and another which contains finitely many independent degenerate directions so that we can invoke the \L ojaisewicz inequality.
    
    The simplest case is when the Hessian is an isomorphism, in which case we have a linear convergence rate. Without loss of generality, assume that $0$ is a critical point of $E.$
     \begin{prop}[Proposition 11.2.1 of~\cite{haraux2015convergence}]
        Assume that $D^2E(0)\in L(\mathcal{V},\mathcal{V}')$ is an isomorphism. Then the \L ojasiewicz inequality is satisfied with the exponent $\theta=\frac{1}{2}.$ In other words, there exists $c>0$ and $r >0$ such that 
        \begin{equation*}
            \Vert u\Vert_V<r \Longrightarrow \Vert \nabla E(u)\Vert_{V'}\ge c|E(u)|^{1/2}.
        \end{equation*}
    \end{prop}
    
    Now we discuss the case when the Hessian of the energy is semi-Fredholm. The following theorem can be found on Chapter 11 of Haraux and Jendoubi's book~\cite{haraux2015convergence}, but the original idea is attributed to Simon~\cite{simon1983asymptotics}. 
    
    Often, it can be tricky to check analyticity of energy in infinite dimensional spaces. What we actually need is analyticity of the finite dimensional copy of the energy on the kernel of the Hessian of the energy. 
    So Haraux and Jendoubi constructed an abstract framework using a Banach space $Z$ such that $\ker A\subset Z\subset \mathcal{H}$ to achieve this.
    So it can be helpful to restrict the space further, possibly the one with Banach algebra, to make it simpler. The following proposition gives us an analogue of orthogonal decomposition of $H=\ker A\otimes  R(A)$, with respect to the Hilbert structure of $H$, also known as Lyapunov–Schmidt reduction.
    \begin{prop}[Proposition 11.2.6 of~\cite{haraux2015convergence}]
        Assume that $A:=D^2E(0):\mathcal{V}\to\mathcal{V}'$ is a semi-Fredholm operator. Let $\Pi$ be the orthogonal projection mapping to $\ker A$ in $H$.
        Let $Z$ be a Banach space such that $\ker A\subset Z\subset H$ with continuous and dense imbedding. Then $W:=(\Pi+A)^{-1}Z$ is a Banach space isomorphic to $Z$ with respect to $\Vert w\Vert_W=\Vert (\Pi+A)w\Vert_Z.$
    \end{prop}
    This structure provides the key to transfer the \L ojasiewicz inequality from the finite dimensional space to the infinite dimensional space. In this setting, we can invoke the \L ojasiewicz inequality in the finite dimensional kernel of the Hessian, which, in fact, contains all the degeneracy. For the details of the proof, we refer Chapter 11 of Haraux and Jendoubi's book~\cite{haraux2015convergence}. Without loss of generality, we let $0$ be a critical point of $E.$
    \begin{theorem}[Theorem 11.2.7 of~\cite{haraux2015convergence}]\label{t: zloj}
        Assume $A:=D^2E(0)$ is a semi-Fredholm operator and that $N:=\ker A\subset Z$. Assume moreover that: $E:U\to\R$ is analytic where $U\subset W$ is an open neighborhood of $0$, in addition $\nabla E(U)\subset Z$ and $\nabla E: U\to Z$ is analytic. Then there exists $\theta\in(0,1/2]$, $r>0$ and $c>0$ such that
        \begin{equation*}
            \Vert u\Vert_V<r \Longrightarrow \Vert \nabla E(u)\Vert_{V'}\ge c|E(u)|^{1-\theta}.
        \end{equation*}
    \end{theorem}
    
\paragraph{\textbf{Simplified framework}}
Luckily, in the models we study, the energy is analytic in $\mathcal{V}$. Even though this is a simpler case, it is not covered by Theorem~\ref{t: zloj}. We briefly go over how this works.
Assume that $E:\mathcal{V}\to\R$ is analytic.
\begin{prop}
    Assume that $A=D^2E(0)$ is a semi-Fredholm operator and let
    \begin{align*}
        \mathcal{N}\;:\;&\mathcal{V}\longrightarrow \mathcal{V}'\\
        &u\longmapsto \Pi u+\nabla E(u).
    \end{align*}
    Then there exist a neighborhood of $0$, $W_1(0)$ in $\mathcal{V}$, a neighborhood of $0$, $W_2(0)$ in $\mathcal{V}'$ and an analytic map $\Psi:W_2(0)\to W_1(0)$ which satisfies
    \begin{align*}
        \mathcal{N}(\Psi(f))=f &\quad\forall f\in W_2(0),\\
        \Psi(\mathcal{N}(u))=u &\quad\forall u\in W_1(0),\\
        \Vert \Psi(f)-\Psi(g)\Vert_V \leq C_1\Vert f-g\Vert_{V'}& \quad\forall f,g\in W_2(0), \;\; C_1>0.
    \end{align*}
\end{prop}
\begin{proof}
     $\mathcal{N}$ is analytic since by Proposition~\ref{p: fdf anal}, $\nabla E:\mathcal{V}\to\mathcal{V}'$ is analytic. $D\mathcal{N}(0)=\Pi+\nabla^2 E(0)$ is an isomorphism from $\mathcal{V}$ to $\mathcal{V}'$. So by the inverse function theorem we have the result.
\end{proof}
[Finite dimensional copy of the energy]
We identify $\ker(A)$ as a subspace in $\mathcal{V}'$ by the structure of $\mathcal{H}$. Let $(\varphi_1,\varphi_2,\cdots,\varphi_d)$ denote an orthonormal basis of $\ker(A)$ relatively to the inner product of $H.$ Let's define the coordinate chart $\varphi(\xi):= \sum_{j=1}^d\xi_j\varphi_j$. Then in a small enough neighborhood of $0$ we achieve $\varphi(\xi)\in W_2(0)$. Now we define the map $\Gamma$ by
\begin{equation*}
    \Gamma(\xi)=E(\Psi(\varphi(\xi))).
\end{equation*}

The following proposition suggests that the finite dimensional copy of the energy $\Gamma$ is a good enough approximation of the energy $E$. We refer~\cite{haraux2015convergence} for the proof.
\begin{prop}[Proposition 11.2.4 of~\cite{haraux2015convergence}]\label{p: gam}
    Let $u\in W_1(0)$ be such that $\Pi(u)=\sum_{j=1}^d\xi_j\varphi_j\in W_2(0)$. Then for some constants $C,K>0$ independent of $u$, we have
    \begin{align*}
        \Vert \nabla \Gamma(\xi)\Vert_{\R^d}&\leq C\Vert \nabla E(u)\Vert_{\mathcal{V}'}\\
        |E(u)-\Gamma(\xi)|&\leq K\Vert \nabla E(u)\Vert_{\mathcal{V}'}^2.
    \end{align*}
\end{prop}

With a slight modification of the proof of Theorem 11.2.7 from~\cite{haraux2015convergence}, we have the \L ojasiewicz inequality when $E$ is analytic in the sense of $\mathcal{V}$.
\begin{theorem}\label{t: vloja}
        Assume $A:=D^2E(0):\mathcal{V}\to\mathcal{V}'$ is a semi-Fredholm operator and  $E:U\to\R$ is analytic where $U\subset\mathcal{V}$ is an open neighborhood of $0$. Then there exists $\theta\in(0,1/2]$, $r>0$ and $c>0$ such that
        \begin{equation*}
            \Vert u\Vert_V<r \Longrightarrow \Vert \nabla E(u)\Vert_{V'}\ge c|E(u)|^{1-\theta}.
        \end{equation*}
    \end{theorem}
    \mnote{explain more}
\begin{proof}
        Since $\Gamma=E\circ\Psi\circ\varphi$ is analytic in a neighborhood of $0\in\R^d$ as compositions of analytic functions, we can apply the classical \L ojasiewicz inequality.
        \begin{equation*}
            |E(u)|^{1-\theta}\leq |\Gamma(\xi)|^{1-\theta}+|\Gamma(\xi)-E(u)|^{1-\theta} \leq \frac{1}{C_0}\Vert \nabla\Gamma(\xi)\Vert_{\R^d}+|\Gamma(\xi)-E(u)|^{1-\theta}.
        \end{equation*}
        By Proposition~\ref{p: gam}, we have
        \begin{equation*}
            |E(u)|^{1-\theta}\leq \frac{C}{C_0}\Vert \nabla{E}(u)\Vert_{\mathcal{V}'}+K^{1-\theta}\Vert\nabla{E}(u)\Vert_{\mathcal{V}'}^{2(1-\theta)}.
        \end{equation*}
Since $2(1-\theta)\ge 1$, there exists $r,c>0$ such that
\begin{equation*}
    \Vert \nabla E(u)\Vert_{\mathcal{V}'}\ge c|E(u)|^{1-\theta},
\end{equation*}
for all $u\in\mathcal{V}$ such that $\Vert u\Vert_\mathcal{V}<r$.
\end{proof}
    
\subsection{\L ojasiewicz convergence scheme for gradient-like systems}
    We present the \L ojasiewicz convergence theorem in a generalized framework. More specifically, let $\mathcal{V}$, $\mathcal{H}$ be two real Hilbert spaces such that
\begin{equation*}
    \mathcal{V}\subset \mathcal{H}=\mathcal{H}'\subset \mathcal{V}'
\end{equation*}
with continuous dense embedding and $\mathcal{H}'$, the topological dual of $\mathcal{H}$ is identified with $\mathcal{H}$.  
\newline
\mnote{We need u' to 0, gradient like system p68 ch7}
\begin{definition}\label{d: gsi}
Let $E\in C^1(\mathcal{V};\R)$. We say that $u\in C^1(\R_+;\mathcal{V})$ satisfies the \textit{angle condition} with the energy $E$ if there exists $\sigma>0$ such that
    \begin{equation}\label{e: rcon}
        \langle -\nabla{E}(u), u'\rangle_{\mathcal{V}'\times\mathcal{V}}\geq \sigma\Vert\nabla{E}(u)\Vert_{\mathcal{V}'} \Vert u'\Vert_{\mathcal{H}}.
    \end{equation}
\end{definition}
Note that the angle condition~\eqref{e: rcon} is, in fact, purely geometric, which is invariant under a change of positive time scale. So the angle condition is not sufficient to determine an explicit rate of convergence. 

\begin{definition}
Let $E\in C^1(\mathcal{V};\R)$. We say that $u\in C^1(\R_+;\mathcal{V})$ satisfies the \textit{rate condition} with the energy $E$ if there exists $\gamma>0$ such that
    \begin{equation}\label{e: irate con}
      \Vert u'\Vert_\mathcal{H}\ge \gamma\Vert \nabla{E}(u)\Vert_{\mathcal{V}'}.
    \end{equation}
\end{definition}
In usual case of gradient-like systems, we have the \textit{dissipation inequality},
\begin{equation}\label{e: ediss}
    -\frac{d}{dt}E(u(t))=\langle -\nabla{E}(u), u'\rangle_{\mathcal{V}'\times\mathcal{V}} \ge c\Vert u'\Vert_\mathcal{H}^2.
\end{equation}
Evidently, the rate condition~\eqref{e: irate con} together with the dissipation inequality~\eqref{e: ediss} provides the angle condition.
\begin{remark}
    Note that Definition~\ref{d: gsi} is a generalization of the gradient system of $E\in C^1(\mathcal{V},\R)$
    \begin{equation*}
         \frac{d}{dt}u(t)=-\nabla{E}(u(t)).
    \end{equation*}
    Assume that $u\in C^1(\R_+;\mathcal{V})$ is a solution. This means that for any $\phi\in\mathcal{V}\subset\mathcal{H}$, the solution satisfies the following,
    \begin{equation*}
        \langle u',\phi\rangle_\mathcal{H}=-\langle \nabla{E}(u), \phi\rangle_{\mathcal{V}'\times\mathcal{V}}.
    \end{equation*}
    The rate condition is naturally satisfied and from that we can check the angle condition as well.\newline
    [Rate condition]
    \begin{equation*}
        \Vert \nabla{E}(u)\Vert_{\mathcal{V}'}=\sup_{\Vert \phi\Vert_\mathcal{V}=1} \langle \nabla{E}(u), \phi\rangle_{\mathcal{V}'\times\mathcal{V}}
        =\sup_{\Vert \phi\Vert_\mathcal{V}=1} -\langle u',\phi\rangle_\mathcal{H} \le \Vert u'\Vert_\mathcal{H}.
    \end{equation*}
    [Angle condition]
    \begin{equation*}
        \langle -\nabla{E}(u),u'\rangle_{\mathcal{V}'\times\mathcal{V}}=\Vert u'\Vert_\mathcal{H}^2 \ge \Vert \nabla{E}(u)\Vert_{\mathcal{V}'}\Vert u'\Vert_\mathcal{H}.
    \end{equation*}
\end{remark}
\mnote{dissipative+rate=angle}
\mnote{more description}
We present the \L ojasiewicz convergence theorem in infinite dimensional gradient-like systems. This can be seen as an extension of Theorem 10.3.1 and Theorem 11.3.1 in ~\cite{haraux2015convergence} combined.
    \begin{theorem}\label{t: conv} 
    Let $u\in C^1(\R_+;\mathcal{V})$ satisfy the angle condition~\eqref{e: rcon} with the energy $E\in C^1(\mathcal{V};\R)$, and assume that
    \begin{enumerate}
        \item $\cup_{t\ge0}\{u(t)\}$ is precompact in $\mathcal{V},$ and
        \item $E$ satisfies \L ojasiewicz inequality near every point in $\omega(u)$, the $\omega$-limit set of $u$.
    \end{enumerate}
    Then there exists $\ut\in\mathcal{V}$ such that 
    $$\lim_{t\rightarrow\infty}\Vert u(t) - \ut\Vert_\mathcal{V} = 0.$$
    Moreover, if $u$ satisfies the rate condition~\eqref{e: irate con}, then $\ut\in\mathcal{E}:=\{u\in \mathcal{V}: \nabla E(u) = 0\}$ and
    \begin{equation*}
        \Vert u(t) - \ut\Vert_\mathcal{H} = 
        \left\{\begin{array}{ll}
             O(e^{-\delta t}) & \text{if}\;\; \theta=\frac{1}{2}, \;\;\;\;\text{for some}\;\; \delta>0 \\
             O(t^{-\theta/(1-2\theta)})& \text{if}\;\; 0<\theta<\frac{1}{2},
        \end{array}\right.
    \end{equation*}
    where $\theta$ is the \L ojasiewicz exponent of $E$ at $\ut$.
    \end{theorem}
\begin{proof}
Let's fix $\ut\in\omega(u)$ and without loss of generality let's assume $E(\ut)=0$.\newline
Step 1.\textbf{ Unifying constants.} [Lemma 2.1.6 of~\cite{haraux2015convergence}] 
Since $\Gamma:=\overline{\{u(t)\}_{t\ge0}}$ is compact in $\mathcal{V}$, we can show that $\omega(u)$ is compact and connected subset of $\mathcal{V}$ (Theorem 5.1.8 of~\cite{haraux2015convergence}). Therefore by finite covering argument, we can find positive constants $C$, $\theta$ and $r$ such that,
\[
 C|E(u)|^{1-\theta}\leq \Vert \nabla E(u)\Vert_{\mathcal{V}'}, \quad \forall u,\; \text{dist}(u,\omega(u))<r,
\]
where the distance is defined with respect to $\Vert\cdot\Vert_{\mathcal{V}}$.
\newline
Step 2. \textbf{Convergence and Cauchyness.}
\newline
Since the trajectory of $u$ is precompact, by Theorem 5.1.8 of~\cite{haraux2015convergence}, we have 
$$\lim_{t\to\infty} \text{dist}(u(t),\omega(u))=0.$$
Therefore we can find $t_0>0$ such that for all $t\ge t_0$, $\text{dist}(u(t),\omega(u))<r$. Now we can achieve $u
\in L^1([t_0,\infty);\mathcal{H)}$ by the following calculation,
\begin{align}
    -\frac{d}{dt}E(u)^\theta &= -E(u)^{\theta-1}\frac{d}{dt}E(u)\notag\\
    &=E(u)^{\theta-1}\langle\nabla{E}(u),u'\rangle\notag\\
    &\ge \sigma E(u)^{\theta-1}\Vert \nabla E(u)\Vert_{\mathcal{V}'}\Vert u'\Vert_{\mathcal{H}}\label{e: acc}\\
    &\ge \sigma C\Vert u'\Vert_{\mathcal{H}}\notag
\end{align}
Therefore, by
\begin{align}\label{e: ul1}
    \int_{t_0}^\infty \Vert u'(t)\Vert_{\mathcal{H}} dt \le CE(u(t_0))^\theta<\infty,
\end{align}
we can see that the trajectory has finite length in $\mathcal{H}$ which implies that $\{u(t)\}_{t\ge0}$ is Cauchy in $\mathcal{H}$. Thus
\begin{equation*}
    \lim_{t\to\infty}\Vert u(t)-\ut\Vert_{\mathcal{H}}=0.
\end{equation*}
Moreover, since the trajectory is precompact in $\mathcal{V}$, we have 
\begin{equation*}
    \lim_{t\to\infty}\Vert u(t)-\ut\Vert_{\mathcal{V}}=0.
\end{equation*}
Step 3. \textbf{Convergence rate.}
From \eqref{e: ul1}, we notice that the convergence rate of $u$ is determined by the decay rate of the energy $E$.
By applying the rate condition~\eqref{e: irate con}, on \eqref{e: acc} and the \L ojasiewicz inequality, we can derive the following inequality of $E$,
\begin{align*}
    -\frac{d}{dt}E(u)^\theta &\ge \sigma\gamma E(u)^{\theta-1}\Vert \nabla E(u)\Vert_{\mathcal{V}'}^2\\
     &\ge \sigma\gamma C^2 E(u)^{1-\theta}.
\end{align*}
So we achieve the rate of convergence by using the Gr\"onwall type estimate.

Furthermore, due to~\eqref{e: ul1}, we know that $\liminf_{t\to\infty}\Vert u'(t)\Vert=0$. Again by the rate condition~\eqref{e: irate con},
\begin{equation}
      0=\liminf_{t\to\infty}\Vert u'\Vert_\mathcal{H}\ge \gamma\lim_{t\to\infty}\Vert \nabla{E}(u)\Vert_{\mathcal{V}'},
    \end{equation}
which implies $\Vert \nabla{E}(\ut)\Vert_{\mathcal{V}'}=0$.
\end{proof}

\chapter{Discrete-time gradient descent by convex splitting}\label{c: disc time cs}
Gradient descent is a widely used technique in optimization, since it steers the solution to optimize the objective function, or the energy. 

In~\cite{absil2005convergence}, Absil, Mahony and Andrews successfully extended the \L ojasewicz convergence theorem to sequences that satisfies \textit{the strong descent condition}, which assures the produced sequence to behave as similar to a solution of gradient flow. Needless to say, the energy must be monotone along the sequence, however, in general, time discretization scheme does not provide the monotonicity of the energy for free without proper choice of time step, or localization, which might require expensive calculation. 

Mainly, we discuss an optimization algorithm using the class of functions that have a representation as a \textit{difference of convex functions} (DC). Widely used in non-convex optimization, DC algorithms, suggested by Tao~\cite{tao1986algorithms} in 1986, treat the gradient of concave part explicitly and the gradient of convex part implicitly. Combining basic convex inequalities on each convex and concave part provides monotonicity of the energy. Independently, in the area of PDE, Eyre~\cite{eyre1998unconditionally} used the same idea to formalize energetically stable time-discretization scheme for partial differential equations, such as Allen-Cahn equations, Cahn-Hilliard equations, etc. Indeed, Eyre's semi-implicit Euler's method can be regarded as a special case of DCA.
When the methods are applied to analytic cost functions, by Absil \textit{et al}.'s convergence theorem, every bounded sequence converges~\cite{dinh2014recent}.

Lastly, we provide a simple, but novel, observation relating to methods of acceleration for DC algorithms. We describe a class of DC algorithms with momentum, which adds a (possibly non-convex) Bregman divergence centered at the previous point to the concave part of the DC algorithm. In this framework, Polyak's heavy ball method and Nesterov's acceleration method can be regarded as the same momentum method, but based on objective functions that are ``dual" to each other.

\section{Convergence analysis of discrete schemes}
In \cite{absil2005convergence}, Absil \textit{et al}. introduced the  \L ojasiewicz inequality to prove convergence of any bounded sequence driven by an analytic energy under the following \textit{strong descent condition}.
\begin{definition}
A sequence $\{u^n\}_{n\in\N_0}\subset\R^N$ satisfies \textit{strong descent condition} if the following conditions are satisfied for $k>K$ with some $K>0$ and $\sigma>0$.
\paragraph{[Primary descent condition]}
\begin{equation*}
    H(u^k)-H(u^{k+1})\ge \sigma\Vert \nabla H(u^k)\Vert\Vert u^{k+1}-u^k\Vert.
\end{equation*}
\paragraph{[Complementary descent condition]}
\begin{equation*}
    [H(u^{k+1})=H(u^k)] \Rightarrow [u^{k+1}=u^k].
\end{equation*}
\end{definition}
The result is abstract in order to be applicable to various iterations, without specifying the connection between the sequence and the gradient. As a quantitative result, they showed the $l^1$-tail of the sequence is bounded by the decay of the energy.
\begin{theorem}[Absil \textit{et al}. \cite{absil2005convergence}]\label{t: absil}
    Let $H:\R^N\to \R$ be analytic. Then any bounded sequence $\{u_n\}_{n\in\N_0}$ satisfying the strong descent condition converges to an equilibrium $\Tilde{u}$.
    Moreover, $l^1$- tail of $\{u^n\}_{n\in\N_0}$ can be estimated by the decay of $H_n:=H(u^n)-H(\Tilde{u})$,
    \begin{equation*}
    \sum_{k\ge n}\Vert u^{k+1}-u^k\Vert\leq \frac{1}{c\sigma\theta}H^\theta_n.
    \end{equation*}
where $c,\theta>0$ are \L ojasiewicz constant and exponent at $\Tilde{u}$, respectively. 
\end{theorem}
\begin{remark}
Theorem~\ref{t: absil} does not tell the explicit rate of convergence. Note that primary descent condition is a discrete-time analogue of the angle condition between the gradient and the velocity, i.e.,
\begin{equation*}
    \frac{d}{dt}H(u(t))=\langle \nabla H(u(t)), u'(t)\rangle \ge \sigma \Vert\nabla H(u(t))\Vert\Vert u'(t)\Vert.
\end{equation*}
The angle condition is, in fact, purely geometric, which is invariant under the choice of time scale. Therefore, the strong descent condition is not sufficient to produce explicit rate of convergence. For any $\{u^n\}_{n\in\N_0}$ satisfying strong descent condition, one can cook up super-, or subsequences that satisfy the same strong descent condition but have different rate of convergence.
\end{remark}
In order to calculate quantitatively the rate of convergence, we add a condition which an ordinary ``gradient-motivated" scheme likely satisfies.
\begin{definition}
A sequence $\{u^n\}_{n\in\N_0}\subset\R^N$ satisfies the \textit{rate condition} if for $k>K$ with some $K>0$ and $\gamma>0$.
\begin{equation*}
    \Vert u^{k+1}-u^k\Vert\ge \gamma\Vert \nabla H(u^k)\Vert.
\end{equation*}
\end{definition}
Note that the rate condition prevents the repetition of a sequence except at a critical point.
\begin{corollary}
In addition to the assumption in Theorem~\ref{t: absil}, suppose $\{u^n\}_{n\in\N_0}$ satisfies the {rate condition}. Then the decay of  $H_n$ is given by $\theta$ and $a:=\sigma\gamma c^2$,
   \begin{equation}\label{e: h conv rate}
        H_n \leq
        \left\{\begin{array}{ll}
             H_0e^{-a n}, & \text{if}\;\; \theta=\frac{1}{2},\\
             (H_0^{2\theta-1}+(1-2\theta)a n)^\frac{1}{2\theta -1},& \text{if}\;\; 0<\theta<\frac{1}{2}.
        \end{array}\right.
\end{equation}
\end{corollary}
\begin{proof}
Note that the convergence rate of $\{u^n\}_{n\in\N_0}$ depends on the decay rate of $H_n$. Combining primary descent condition and the ordinary descent condition gives estimate on the energy gap,
\begin{align*}
     H_{n+1}-H_n &\leq -{\sigma}{\gamma}\Vert \nabla H_n\Vert^2\\
    &\leq -{\sigma}{\gamma}c^2 H_n^{2-2\theta}.
\end{align*}
Let $k\leq n$, since $H_k$ is positive and monotonically decreasing and $\theta<1$, for any $h\in[H_{k+1},H_k]$ 
$$H_k^{2\theta-2}\leq h^{2\theta-2}.$$
So,
\begin{align*}
    {\sigma}{\gamma}c^2&\leq H^{2\theta-2}_k(H_k-H_{k+1})\\
    &=\int_{H_{k+1}}^{H_k} H^{2\theta-2}_kdh\\
    &\leq\int_{H_{k+1}}^{H_k}h^{2\theta-2}dh.
\end{align*}
If $\theta=\frac{1}{2}$,
\begin{align*}
    {\sigma}{\gamma}c^2&\leq \log H_k-\log H_{k+1}\\
    H_n&\leq H_0e^{-{\sigma}{\gamma}c^2n}.
\end{align*}
If $\theta\in (0,\frac{1}{2})$,
\begin{align*}
    {\sigma}{\gamma}c^2 &\leq \frac{1}{2\theta-1}(H^{2\theta-1}_{k}-H^{2\theta-1}_{k+1})\\
    H_n&\leq (H_0^{2\theta-1}+{\sigma}{\gamma}c^2(1-2\theta)n)^\frac{1}{2\theta -1}.
\end{align*}
\end{proof}
\section{Convergence of difference of convex functions algorithm}\label{s: dca}
Our focus is to show convergence of gradient flows without assuming convexity of the energy. As the backbone of nonconvex programming and global optimization, the difference of convex algorithm, or DCA, uses the following class of functions.
\begin{definition}
    We say $f$ is a \textit{difference of convex functions} (DC) if there are convex functions $g$ and $h$ such that $f=g-h$, and the pair $(g,h)$, is a \textit{convex splitting} of $f$. 
\end{definition}
The family of DC functions is quite general. It is known that the set of DC functions is dense in the set of continuous functions on a compact convex domain in $\R^N.$ 
Moreover, for a fixed $f$, choice of a convex splitting pair is flexible. The set of convex splitting is convex, since a convex combination of any two $(g_1,h_1)$ and $(g_2,h_2)$ convex splitting is again convex splitting for $f$. Also for any convex function $g'$, $(g+g',h+g')$ is again a convex splitting of $f$.

In this context, finding a critical point of $f$ can be written as solving the following fixed point problem using convex splitting of $f$,
\begin{equation*}
    \nabla g(x) =\nabla h(x).
\end{equation*}

The analogous optimization scheme was introduced by Pham Dinh Tao~\cite{tao1986algorithms} in 1986, and the iteration is named \textit{DC algorithm}. The history of DCA is well reviewed in~\cite{le2018dc}.
Let $H=H_+-H_-$ be a cost function with the convex splitting $H_+$ and $H_-.$
\begin{equation}\label{e: dc al}
    \nabla H_+(\un)=\nabla H_-(u^n)
\end{equation}
\eqref{e: dc al} is called DC Algorithm (Difference of Convex functions). Due to the implicit part of the algorithm, we need to make sure $\nabla H_+$ is bijective.
 \begin{definition}
 Let $\kappa\in\R$, we say $f:\R^N\to\R$ is $\kappa$-convex if $f(x)-\frac{\kappa}{2}\Vert x\Vert^2$ is convex. Moreover, if $\kappa>0$ we say $f$ is strongly convex.
 \end{definition}
 \begin{definition}
We say $f:\R^N\to\R$  is $L$-smooth if $f$ has Lipschitz derivative with constant $L>0.$
 \end{definition}
Throughout the thesis, we assume $H_+$ is strongly convex, so that the iteration~\eqref{e: dc al} is well-defined. Indeed, if $H_+$ is strongly convex, there exists unique $u^*$ for any $v$ such that
\begin{equation*}
    \nabla H_+(u^*)=v,
\end{equation*}
since $u\mapsto H_+(u)-v^Tu$ has a unique minimizer due to the strong convexity.

The main benefit of DCA is that DCA is energetically stable, in other words, the energy $H$ is monotonic along the sequence. It is directly deduced from the following basic convex inequality.
\begin{lemma}[Convex inequality]\label{l: conv ineq}
        For any convex $g\in C^1(\R)$ and for any $a,b\in \mathbb{R}$,
        \begin{equation*}
            g'(a)(b-a) \leq g(b) - g(a) \leq g'(b)(b-a).
        \end{equation*}
\end{lemma}

\begin{prop}
    Let $H=H_+-H_-$ such that $H_+$ is $\kappa$-convex and $H_-$ is $\mu$-convex, and $\{u^n\}_{n\in\N_0}$ be a sequence from DCA~\eqref{e: dc al}. If $\kappa+\mu>0$. then for any $n\in\N_0$
    \begin{equation*}
        H(\un)\leq H(u^n).
    \end{equation*}
\end{prop}
\begin{proof}
We apply Lemma~\ref{l: conv ineq} to
$H_+(u)-\frac{\kappa}{2}\Vert u\Vert^2$ and $H_-(u)-\frac{\mu}{2}\Vert u\Vert^2$ so that
\begin{align*}
    H_+(\un)-H_+(u^n) &\leq \langle \nabla H_+(\un),\un-u^n\rangle -\frac{\kappa}{2}\Vert \un-u^n\Vert^2,\\
    H_-(\un)-H_-(u^n) &\geq \langle \nabla H_-(u^n),\un-u^n\rangle +\frac{\mu}{2}\Vert \un-u^n\Vert^2.
\end{align*}
By subtracting those we check the monotonicity of $H$ along the sequence.
\begin{equation}\label{e: h mon}
     H(\un)-H(u^n)\leq -\frac{\kappa+\mu}{2}\Vert \un-u^n\Vert^2.
\end{equation} 
\end{proof}

Although it had been widely used in optimization, the convergence of DCA was proved only rather recently in 2014 by Dinh \textit{et al}. \cite{dinh2014recent}. They applied the \L ojasiewicz convergence theorem constructed by Absil \textit{et al}.~\cite{absil2005convergence} to DCA with subanalytic data.
\begin{theorem}\label{t: nlmg conv}
        Let $H:\R^N\to\R$ satisfies \L ojasiewicz inequality and $H=H_+-H_-$, where $H_+$ is $\kappa$-convex and $L$-smooth and $H_-$ is $\mu$-convex with with $\kappa+\mu>0$ and $L>0$ . Then every bounded solution $\{u^n\}_{n\in\N_0}$ of \textit{DC Algorithm} \eqref{e: dc al}
    converges.
    
    Moreover, the $l^1$- tail of $\{u^n\}_{n\in\N_0}$ can be estimated by the decay of $H_n:=H(u^n)-H(\Tilde{u})$,
    \begin{equation*}
    \sum_{k\ge n}\Vert u^{k+1}-u^k\Vert\leq \frac{2L}{c\theta(\kappa+\mu)}H^\theta_n,
    \end{equation*}
where $L$ is Lipschitz constant of $\nabla H_+$ and $c,\theta>0$ are the \L ojasiewicz constant and exponent at $\Tilde{u}$, respectively. 
Let $\beta:=(\kappa+\mu)c^2/{2L^2}$ then the decay of $H_n$ is as follow,
   \begin{equation*}
        H_n \leq
        \left\{\begin{array}{ll}
             H_0e^{-\beta n}, & \text{if}\;\; \theta=\frac{1}{2},\\
             (H_0^{2\theta-1}+(1-2\theta)\beta n)^\frac{1}{2\theta -1},& \text{if}\;\; 0<\theta<\frac{1}{2}.
        \end{array}\right.
    \end{equation*}
\end{theorem}

\begin{proof}[Proof of Theorem~\ref{t: nlmg conv}]
We check the strong descent condition and the rate condition. The complementary descent condition follows directly from the monotonicity of $H$~\eqref{e: h mon}. Now from~\eqref{e: dc al} we split $H_-=H_+-H$ to see
\begin{align}\label{e: hp p}
    \nabla H_+ (\un)-\nabla H_+ (u^n)&= -\nabla H(u^n).
\end{align}
So we have the rate condition,
\begin{align*}
    \Vert \nabla H(u^n)\Vert = \Vert\nabla H_+ (\un)-\nabla H_+ (u^n)\Vert\leq L \Vert \un -u^n\Vert,
\end{align*}
and the primary descent condition follows from~\eqref{e: h mon}
\begin{align*}
     H(\un)-H(u^n)&\leq -\frac{\kappa+\mu}{2L}\Vert \nabla H(u^n)\Vert \Vert \un-u^n\Vert.
\end{align*} 
\end{proof}

\begin{remark}
As in \eqref{e: hp p}, DCA algorithm can be seen as a gradient descent with $H_+$ preconditioning. Considering the flexibility of convex splitting, it is important to choose proper convex splitting of the energy in order to make the algorithm efficient. 
A question arises: how can we split $H$ in order to achieve optimal convergence rate? For given convex splitting $H_\pm$, let's perturb then with quadratic function $H_\pm+t\Vert\cdot\Vert^2/2$ and optimize $\beta$ from~Theorem~\ref{t: nlmg conv}. Adding quadratic function makes $(\kappa,\mu,L)\mapsto (\kappa+t,\mu+t,L+t)$, so
\begin{align*}
   \beta(t)=\frac{\kappa+\mu+2t}{2(L+t)^2}c^2,
\end{align*}
which is maximized when $t=L-\kappa-\mu$. This implies when $H_+$ is well rounded in a sense that $L-\kappa \ll 1$, making $H_-$ flatter might provide faster convergence.
\end{remark}

Due to the flexibility of the convex splitting, \eqref{e: dc al} is in fact a large class of optimization algorithms. 
\begin{exmp}
The simplest example is Euler's method, 
\begin{equation*}
    \un-u^n=\tau\nabla H(u^n).
\end{equation*}
Here, the convex splitting of the energy $H$ is given as follows,
\begin{equation*}
    H(u)=\frac{1}{2\tau}\Vert u\Vert^2-\big(\frac{1}{2\tau}\Vert u\Vert^2-H(u)\big),
\end{equation*}
where $\tau>0$ is small enough to assure the convexity of $H_-$.
\end{exmp}

It is worth mentioning that the idea of convex splitting to stabilize the energy also arises independently in the area of partial differential equations.
\begin{exmp}
[Eyre's semi-implicit Euler's method]
Adapting Elliott and Stuart's idea ~\cite{elliott1993global} to stabilize the energy using convex splitting, Eyre~\cite{eyre1998unconditionally}, in 1997, formalized a semi-implicit time discretization of gradient flows,
\begin{equation}\label{e: disc cs}
        \un-u^n=-\tau(\nabla H_+(\un)-\nabla H_-(u^n)),
\end{equation}
where $(H_+,H_-)$ is a convex splitting of the energy $H$, in order to formulate time-discretized partial differential equations such as the Cahn-Hilliard equations.
Followed from the difference of convex inequalities, Lemma~\ref{l: conv ineq}, Eyre's semi-implicit Euler's method provides unconditionally stable energy. The scheme was proposed and has been applied to solve discretized partial differential equations, such as Allen-Cahn or bistable reaction diffusion equations~\cite{allen1979microscopic}~\cite{christlieb2014high}, the Cahn-Hilliard equations~\cite{cahn1961spinodal}~\cite{glasner2016improving}, the Ginzburg-Landau equations~\cite{chapman1992macroscopic}, the Runge–Kutta scheme~\cite{shin2017unconditionally}, and the phase field crystal (PFC) model~\cite{wise2009energy}.

To the best of our knowledge convergence of the semi-implicit Euler's method, or the connection to DCA has not been discussed. Note that we recover the semi-implicit Euler's method from DCA by adding extra quadratic energy to the given convex splitting $(H_+,H_-)$ as follows,
\begin{equation*}
    H(u)=\big(\frac{1}{2\tau}\Vert u\Vert^2+H_+(u)\big)-\big(\frac{1}{2\tau}\Vert u\Vert^2+H_-(u)\big).
\end{equation*}
Thus, by Theorem~\ref{t: nlmg conv}, as long as $H$ satisfies \L ojasiewicz inequality, any precompact sequence produced from the semi implicit Euler's method~\eqref{e: disc cs} converges with the explicit convergence rate.
\end{exmp}
\section{Duality of Polyak's and Nesterov's momentum method}
In this section, we discuss a few variants of DC algorithms. Firstly,
due to the implicit part of the algorithm, each step of DC algorithm requires to solve a convex problem, inheriting the dual DC algorithm. In this section we will go over the equivalence of primal and dual DC algorithms. 

In order to accelerate the rate of convergence, Polyak firstly developed the Heavy ball method (1964) which boosts iterates in the direction of the previous increment, like adding a ``momentum". Another successful and celebrated example, Nesterov's acceleration method (1983), rather takes different approach. Comparing to external addition of momentum in the Heavy ball method, Nesterov's acceleration applies momentum internally. Since then, numerous acceleration methods have been adapting either Polyak's approach or Nesterov's approach.
With the perspective of the primal and dual DC in mind, however, we can interpret Polyak's Heavy ball method and Nesterov's acceleration method as two examples of a general class of DC algorithms with momentum, applied to energies that are dual to each other. In this way, we provide a convergence proof that covers both types of method.

\paragraph{\textbf{Dual DC algorithms}}
Inverting a gradient, when enough regularity is presumed, can be interpreted as a maximization problem,
\begin{equation*}
    \nabla f(x)=p \Longleftrightarrow \arg\!\max_x\{ p^Tx -f(x)\} \Longleftrightarrow \nabla f^*(p)=x,
\end{equation*}
where $f^*(p):=\sup_{x} \{p^Tx-f(x)\}$ is the \textit{Legendre transform} of $f(x)$.
So by solving the implicit part of the DC algorithm, the dual DC algorithm naturally arise. Let $p^n:=\nabla H_+(u^n)=\nabla H_-(u^{n-1})$, then we can write down the iteration of $p^n$
\begin{align*}
    p^{n+1}&=\nabla H_-(u^n)\\
    &=\nabla H_-(\nabla H_+^*(p^n)).
\end{align*}
This leads us to the dual DC algorithm.

[Dual DC Algorithm]
\begin{equation*}
    \nabla H^*_- (p^{n+1})=\nabla H^*_+(p^n).
\end{equation*}
$\nabla H_+$ has to be easily invertible.

The dual DC algorithm aims to optimize the dual energy, $H_-^*-H_+^*$, not to be confused with the convex dual of the energy.
Indeed, we can check the equivalence between the primal and dual DC algorithms~\cite{tao1997convex},
\begin{align*}
  \inf_x\{ H_+(x)-H_-(x)\}&= \inf_x\{ H_+(x) -\sup_p \{p^Tx -H^*_-(p)\}\}\\
  &=\inf_{x,p}\{ H_+(x) -p^Tx+H^*_-(p)\}\\
  &=\inf_p\{ H^*_-(p)-H^*_+(p)\}.
\end{align*}
It is known that assuming $f$ is strongly convex and analytic is enough to assure the \L ojasiewicz inequality of $f^*$ (Proposition 2.1~\cite{le2018convergence}). So when $H_+$ and $H_-$ are strongly convex and analytic, we can apply \L ojasiewicz convergence theorem to the dual DC. Later, we will use this to prove convergence of DC algorithm with momentum which  includes Polyak's momentum method and Nesterov's acceleration algorithm.

\begin{remark}
A funny thing happens when $H_+=H_-^*$, i.e., if there exists $\Phi$ such that $H=\Phi ^*-\Phi$. Then the DC algorithm is given as follows,
\begin{align*}
   &\nabla \Phi^*(\un)=\nabla \Phi(u^n),\\
    &\un=\nabla \Phi\circ \nabla \Phi(u^n).
\end{align*}
In this particular case, minimizing $H$ is equivalent to finding a fixed point of $\nabla \Phi^{(2)}$, in other word finding $x,y$ such that $(\nabla\Phi(x)=y) \wedge (\nabla\Phi(y)=x)$.
Suppose that $H\in C^2$, then the Hessian of $\Phi$ satisfies the following,
\begin{align*}
    \nabla^2 H&=\nabla^2 \Phi^* -\nabla^2 \Phi\\
    &=(\nabla^2 \Phi)^{-1} -\nabla^2 \Phi.
\end{align*}
The solution is uniquely determined by the Hessian of $H$, i.e., $\nabla^2\Phi=\frac{1}{2}(S-\nabla^2 H\big)>0$, where $S$ is a positive definite matrix satisfying $S^2=(\nabla^2 H)^2 +4I$. 
\end{remark}

\paragraph{\textbf{DC algorithms with momentum}}
Now consider DC algorithms for the energy $H$ modified by $\beta$-smooth ``momentum" $b:\R^N\to \R$ centered at a point $v\in \R^N$,
\begin{equation}~\label{e: dcm pot}
   \Tilde{H}(u):= H(u)-b(u)+b(v)+\langle \nabla b(v),u\rangle.
\end{equation}
 Here we do not necessarily assume $b$ is convex. When $b$ is convex, however, one can notice that \eqref{e: dcm pot} is exactly
\begin{align*}
    \Tilde{H}(u)=H(u)-D_b(u,v),
\end{align*}
where $D_b(u,v)$ is so-called Bregman divergence, which can be considered as adding repulsive potential pushing away from $v$. Treating the added momentum as a concave part of $H$, we have a convex splitting of $\Tilde{H}$ as follows,
\begin{align*}
    \Tilde{H}(u)&=H_+(u)-(H_-(u)+D_b(u,v))\\
    &=\Tilde{H}_+(u)-\Tilde{H}_-(u).
\end{align*}
At each step, let's fix $v$ as the previous point, $u^{n-1}$. Then the corresponding DC algorithm on $\Tilde{H}$ motivates the following class of DC algorithms with momentum on $H$,

[DC algorithms with momentum]
\begin{equation}\label{e: dc mom}
   \nabla H_+(\un) =\nabla H_-(u^n) + \nabla b(u^n)-\nabla b(u^{n-1}).
\end{equation}
Due to the flexibility of convex splitting and choosing the momentum $b$, the DC algorithm with momentum can be quite comprehensive.
\begin{exmp}
[Polyak's heavy ball method]
\begin{equation*}
    \un= u^n -\tau\nabla H(u^n) +\beta(u^n-u^{n-1}),
\end{equation*}
which is choosing $H_+(u)= \frac{1}{2}\Vert u\Vert^2$, $H_-(u)=\frac{1}{2}\Vert u\Vert^2-\tau H(u)$ and $b(u)=\beta\Vert u\Vert^2/2$ with $\tau <1/L$ where $L$ is Lipschtiz constant of $\nabla H$.
\begin{equation*}
    \nabla H_+(\un) =\nabla H_-(u^n)+\beta(u^n-u^{n-1}).
\end{equation*}
\end{exmp}
\begin{exmp}
[Adaptive momentum algorithm from Bianchi \& Barakat~\cite{barakat2020convergence}]
\begin{align*}
    x^{n+1}&=x^n - a_{n+1}p^{n+1}\\
    p^{n+1}&=p^n+b(\nabla f(x^n)-p^n).
\end{align*}
Rearranging the equation with respect to $x$ provides the following algorithm,
\begin{equation*}
    x^{n+1}=x^n -a_{n+1}b\nabla f(x^n) +\frac{a_{n+1}(1-b)}{a_n}(x^n-x^{n-1}).
\end{equation*}
This is equivalent to letting $\tau$ and $\beta$ be adaptive in Polyak's heavy ball method.
\end{exmp}
\begin{exmp}
[Nesterov's acceleration method]
Note that for given convex splitting $(H_+,H_-)$ of $H$, the dual energy of $H$ is $H_-^*-H_+^*$. Let $H_+(u)= \frac{1}{2}\Vert u\Vert^2$, $H_-(u)=\frac{1}{2}\Vert u\Vert^2-\tau H(u)$ with small enough $\tau <1/L$ and consider adding $b(u)=\beta\Vert u\Vert^2/2$ at the dual dc problem as follows,
\begin{align*}
     \nabla H^*_-(p^{n+1}) &=\nabla H^*_+(p^n)+\beta(p^n-p^{n-1}).
\end{align*}
Since $p^n=\nabla H_+(u^n)=u^n$,
\begin{equation*}
    p^{n+1}=\nabla H_-\big(u^n+\beta(u^n-u^{n-1})\big).
\end{equation*}
This corresponds to the Nesterov's acceleration method,
\begin{equation*}
    \un= u^n+\beta(u^n-u^{n-1}) -\tau\nabla H(u^n+\beta(u^n-u^{n-1})).
\end{equation*}
Thus, we can consider the Nesterov's acceleration method, as the DC algorithm with momentum on the dual energy of $H$.
\end{exmp}
\begin{exmp}
[Boosted DC algorithm]
In~\cite{artacho2018accelerating}~\cite{artacho2019boosted}, Artacho \textit{et al}. came up with the following acceleration algorithm,
\begin{equation*}
    \nabla H_+(\un)=\nabla H_-(u^n+\beta_n(u^n-\unn)),
\end{equation*}
where $\beta_n$ is chosen through a line search procedure, so that we have the monotonicity of energy, uniformly for some $\lambda>0.$
\begin{equation*}
    H(u^n+\beta_n(u^n-\unn))\leq H(u^n)-\lambda \Vert u^n-\unn\Vert^2.
\end{equation*}
They applied Absil's \L ojasiewicz convergence theorem~\cite{absil2005convergence} to prove convergence.
\end{exmp}
\begin{exmp}
[iPiano]
In~\cite{ochs2014ipiano}~\cite{ochs2018local}, Ochs \textit{et al}. studies the following non-convex optimization, 
\begin{equation*}
    \min_{x\in\R^N} f(x)+g(x),
\end{equation*}
where basically assumes $g$ is convex and $f$ is $L$-smooth (those conditions can be relaxed).
They came up with inertial proximal algorithm for non-convex optimization (iPiano) of the following form,
\begin{equation*}
    \un+\alpha\nabla g(\un)=u^n-\alpha\nabla f(u^n) +\beta (u^n-\unn).
\end{equation*}
This can be considered as the DC algorithm with momentum with the convex splitting,
\begin{align*}
    f(x)+g(x)&=(g(x)+\frac{1}{2\alpha}\Vert x\Vert^2)-(\frac{1}{2\alpha}\Vert x\Vert^2-f(x))\\
    &=H_+(x)-H_-(x).
\end{align*}
Note that $H_-$ is $(\frac{1}{\alpha}-L)$-convex, so the condition for $\alpha$ is given to make sure the monotonicity of energy, constructed by the convex inequality~\ref{l: conv ineq}. The convergence result by \L ojasiewicz scheme can be found in~\cite{ochs2018local}.

Moreover, in 2019~\cite{wu2019general}, Wu and Li proved the convergence of generalization of the iPiano,
\begin{equation*}
    \un+\alpha\nabla g(\un)=u^n-\alpha\nabla f\big(u^n-\gamma (u^n-\unn)\big) +\beta (u^n-\unn).
\end{equation*}
This can be seen as adding momentum on both primal and dual step.
\end{exmp}

\mnote{strong convexity is only for inverting the implicit part}
\begin{theorem}[Convergence of DCA with momentum]\label{t: dc mom}
Let $H:\R^N\to\R$ satisfy \L ojasiewicz inequality and $H=H_+-H_-$, where $H_+$ is $\kappa$-convex, $L$-smooth and $H_-$ is $\mu$-convex. If $\beta\in(0,\frac{\kappa+\mu}{2})$, then every bounded solution $\{u^n\}_{n\in\N_0}$ of \textit{DC algorithm with momentum} \eqref{e: dc mom}
    converges.
    
\end{theorem}
\begin{remark}[Nesterov-type algorithm]
Let $G=G_+-G_-$ with $G_\pm$ being strongly convex and (sub)analytic. Then it is known that $G_-^*-G_+^*$ satisfies \L ojasiewicz inequality~(Proposition 2.1~\cite{le2018convergence}), so that we can apply Theorem~\ref{t: dc mom} with $H_\pm=G_\mp^*$.

For example, the Nesterov's algorithm for an (sub)analytic $G:\R^N\to\R$ starts with the following convex splitting 
\begin{align*}
    G_+(u)&=\frac{1}{2}\Vert u\Vert^2\\
    G_-(u)&=\frac{1}{2}\Vert u\Vert^2-\tau G(u).
\end{align*}
Since $G_-$ is strongly convex for small enough $\tau>0$, not only $H_-$ but also $H_+$ satisfies \L ojasiewicz inequality and has Lipschitz derivative. Thus Theorem~\ref{t: dc mom} is applied to the Nesterov's acceleration method.
\end{remark}

The idea of doubling variables in the proof can also be found in~\cite{diakonikolas2021generalized}~\cite{ochs2018local}~\cite{wu2019general}.
\begin{proof}
Let $u,v\in\R^N$ and define $M:\R^{2N}\to \R$ as follows,
\begin{equation*}
    M(u,v):=H(u)+\frac{\kappa+\mu}{4}\Vert u-v\Vert^2.
\end{equation*}
Note that,
\begin{equation}\label{e: nab f}
    \begin{aligned}
     \nabla_u M(u,v) &=\nabla H(u) +\frac{\kappa+\mu}{2}(u-v)\\
    \nabla_v M(u,v) &=\frac{\kappa+\mu}{2}(v-u), 
    \end{aligned}
\end{equation}
$\nabla M(u,v)=0$ if and only if $\nabla H(u)=0$ and $u=v$.

We will use \L ojasiewicz inequality of $M.$ According to Theorem 3.6 of~\cite{li2018calculus} $M$ satisfies the \L ojasiewicz inequality with the same exponent of $H$.

Let $w^n:=(u^n,u^{n-1})^T$ and $u^{-1}:=u^0$.

[Gradient estimate]
First, by changing $H_-=H_+-H$ from~\eqref{e: dc mom}, we have
\begin{equation*}
   \nabla H(u^n)=-\nabla H_+(\un) +\nabla H_+(u^n)+\nabla b(u^n)-\nabla b(u^{n-1}).
\end{equation*}
We use this to estimate $\nabla F$ as follows,
\begin{align*}
    \Vert \nabla M(w^n)\Vert &\leq \Vert \nabla H(u^n)+\frac{\kappa+\mu}{2}(u^n-\unn)\Vert+\Vert \frac{\kappa+\mu}{2}(u^n-\unn)\Vert\\
    &\leq \Vert \nabla H(u^n)\Vert +(\kappa+\mu)\Vert u^n-\unn\Vert\\
    &\leq \Vert \nabla H_+(\un)-\nabla H_+(u^n)-\nabla b(u^n)+\nabla b(\unn)\Vert\\
    & \;\; +(\kappa+\mu)\Vert u^n-\unn\Vert\\
    &\leq L\Vert \un-u^n\Vert +(\kappa+\mu+\beta) \Vert u^n-\unn\Vert\\
    &\leq \sqrt{2} \max\{L,\kappa+\mu+\beta\}\Vert w^{n+1}-w^n\Vert
\end{align*}
\begin{align*}
     H(\un)-H(u^n)&\leq \langle \nabla H_+(\un)-\nabla H_-(u^n),\un-u^n\rangle-\frac{\kappa+\mu}{2}\Vert \un-u^n\Vert^2\\
     &\leq \langle \nabla b(u^n)-\nabla b(\unn),\un-u^n\rangle-\frac{\kappa+\mu}{2}\Vert \un-u^n\Vert^2
\end{align*}

[Primary descent condition]
    \begin{align*}
     M(w^{n+1})-M(w^n)&\leq \langle \nabla b(u^n)-\nabla b(\unn),\un-u^n\rangle\\
     &\;\;-\frac{\kappa+\mu}{4}(\Vert \un-u^n\Vert^2+\Vert u^n -\unn\Vert^2)\\
     &\leq -\frac{\kappa+\mu-2\beta}{4}(\Vert \un-u^n\Vert^2+\Vert u^n -\unn\Vert^2)\\
     &= -\frac{\kappa+\mu-2\beta}{4}\Vert w^{n+1}-w^n\Vert^2\\
     &\leq -\frac{\kappa+\mu-2\beta}{4\sqrt{2}\max\{L,\kappa+\mu+\beta\}}\Vert \nabla M(w^n)\Vert\Vert w^{n+1}-w^n\Vert
    \end{align*}
    
[Complementary descent condition]
Due to the primary descent condition, it suffices to check if $\nabla M(u,v)=0$ implies $u=v$. This follows from the observation on $\nabla M$,~\eqref{e: nab f}.

[Rate of convergence]
Using $(a+b)^2\leq 2a^2+2b^2,$
\begin{equation*}
    \sum_{k\ge n}\Vert u^{k+1}-u^k\Vert\leq \frac{1}{\sqrt{2}}\sum_{k\ge n}\Vert w^{k+1}-w^k\Vert\leq \frac{c(\kappa+\mu-2\beta)}{4\max\{L,\kappa+\mu+\beta\}^2}M^\theta_n,
    \end{equation*}
where $L$ is Lipschitz constant of $\nabla H_+$ and $c,\theta>0$ are \L ojasiewicz constant and exponent at $\Tilde{u}$, respectively. 
Let $\beta:=\frac{c^2(\kappa +\mu-2\beta)}{8\max\{L,\kappa+\mu+\beta\}^2}$ then the decay of $H_n$ is as follow,
   \begin{equation}\label{e: conv rate m}
        M_n \leq
        \left\{\begin{array}{ll}
             M_0e^{-\beta n}, & \text{if}\;\; \theta=\frac{1}{2},\\
             (M_0^{2\theta-1}+(1-2\theta)\beta n)^\frac{1}{2\theta -1},& \text{if}\;\; 0<\theta<\frac{1}{2}.
        \end{array}\right.
\end{equation}
\end{proof}
\begin{remark}
   The convergence rate~\eqref{e: conv rate m} depends only on the geometry of $H$, or the \L ojasiewicz of $H$, since $M$ satisfies the \L ojasiewicz inequality with the same exponent to the $H$'s (Theorem 3.6 of~\cite{li2018calculus}). So, without local analysis, the boosting effect of the momentum method is not reflected in the similar \L ojasiewicz convergence approach,~\cite{diakonikolas2021generalized}~\cite{ochs2018local}~\cite{wu2019general}. 
   
   Note that in this framework, we did not impose any geometric assumption on the momentum function $b$. 
   Heuristically, the effect of the momentum method depends on the local geometry of the momentum function $b$.  If $b$ is convex, the added momentum adds a kick away from the previous point, and if $b$ is concave, it provides the opposite affect. It is possible that the geometry of $b$ needs to be used to deduce the better quantitative convergence result.
   
   Additionally, it is not always that the momentum method is faster than the mere gradient descent. Generally the momentum $b$ is chosen to be a strongly convex function, $\frac{\beta}{2}\Vert\cdot\Vert^2$, and this might be too strong for the energies that have flatter geometry, i.e., when $\theta<1/2$.
   So it seems reasonable to try add momentum that has the local geometry similar to the $H$'s.
   \end{remark}
\mnote{Comment about convergence rate}
\part{Lotka-Volterra type dynamics}
\chapter{Convergence of Lotka-Volterra dynamics with symmetric interaction matrix}\label{ch: lv cont}
Comprising a family of classic and prototypical models
in population ecology, Lotka-Volterra systems have nonlocal nonlinear structure in a fairly amenable way to mathematical analysis. 
Let $\mathbb{R}^N_+:=\{f\in \mathbb{R}^N : f_i>0,\; \forall i\in [N]\}$, where $N\in\mathbb{N}$ and $[N]:=\{1,2,...,N\}$. For $a, d\in\mathbb{R}^N_+$ and $B\in \R^{N\times N}$ the differential equations of the \textit{Lotka-Volterra system} are given by
    \begin{equation}\label{eq1}
        f'_{i} = d_if_{i}(a_i - \sum_{j\in[N]} B_{ij}f_j), \quad \forall i\in [N].
    \end{equation}
In the ecosystem of $N$-many species, the population of the $i$-th species, $f_i$, is determined by its own  growth rate $a_i$ and the influence, whether inhibiting or cooperating, of the $j$-th species over the $i$-th species, $B_{ij}$.
    
There are two important functions, we name them energy and entropy, that are widely used to analyze Lotka-Voltera system.
\paragraph{\textbf{Monotonic energy}}
When the interaction $B$ is symmetric, the energy $E:\mathbb{R}^N \rightarrow \mathbb{R}$,
    \begin{equation*}
        E(f) := -H := \sum_{i\in[N]}f_i a_i - \frac{1}{2}\sum_{i,j\in[N]}B_{ij}f_i f_j,
    \end{equation*}
is a Lyapunov functional of \eqref{eq1},
\begin{equation*}
        \frac{dE}{dt} = \sum_{i\in[N]} d_if_i(a_i-\sum_{j\in[N]} B_{ij}f_j)^2
                    = \sum_{i\in[N]} \frac{f'_i(t)^2}{d_if_i} \ge0\,.
\end{equation*}
Indeed, the Lotka-Volterra system is a gradient flow of $E$ with respect to the metric which is named after Shahshahani,
\begin{equation*}
    \langle g, h\rangle_f:=\sum_{i\in[N]} g_ih_if_i^{-1}.
\end{equation*}
To see this let's take the variation of $E,$
\begin{align*}
    DE(f)(g)&= \sum_{i\in[N]} g_ia_i -g_i(Bf)_i\\
    &=\sum_{i\in[N]} f_i(a_i -(Bf)_i)g_if^{-1}_i.\\
    &=:\langle \text{grad} E(f),g\rangle_f.
\end{align*}
As we have seen in the previous chapter, \L ojasiewicz convergence analysis requires an ``angle condition" between $\nabla E(f)$ and $f'$, i.e., there is $\sigma>0$ such that
\begin{equation*}
    \langle \nabla E(f), f'\rangle \geq \sigma \Vert \nabla E(f)\Vert \Vert f'\Vert.
\end{equation*}
In general, since there can be a species that is vanishing along the evolution, the Shahshahani metric can blow up. Due to this singular structure, it is difficult use the Shahshahani gradient structure to prove convergence of the solutions.
\paragraph{\textbf{Proximal entropy}}
The entropy $F:\mathbb{R}^N_+ \rightarrow \mathbb{R}$ relative to a given state $\ft \in \overline{\mathbb{R}^N_+}$ with a weight $w_i>0$, given by
    \begin{equation*}
        F(f) := \sum_{i\in[N]} w_i(\Tilde{f}_i\log\frac{\Tilde{f_i}}{f_i} + f_i - \Tilde{f}_i) \,.
    \end{equation*}
    The relative entropy $F$ is also known as the Kullback–Leibler divergence, providing a proximal distance from $\ft$.
    We point out that the relative entropy $F$ is well-defined on $\overline{\R^N_+}$, where it has the unique minimizer $\ft$:
    \begin{lemma} Let $\ft\in\overline{\R^N_+}$. Then
        $F:\overline{\mathbb{R}_+^{N}} \rightarrow [0,\infty]$ is convex, and  $F(f)=0$ if and only if $f=\Tilde{f}$.
    \end{lemma}
    
    \begin{proof}
    Define $J = \{i\in[N] : \Tilde{f}_i>0\}$. Then for any $f\in\R^N_+$,
            \begin{align*}
            \text{for all} \;\; j\in J, \qquad &
            \frac{\partial F}{\partial f_j} = w_j(1-\frac{\Tilde{f}_j}{f_j}), \qquad
            \frac{\partial^2 F}{\partial f_j^2} = w_j\frac{\Tilde{f}_j}{f_j^2} > 0, \\
            \text{for all} \;\; j\notin J, \qquad &
            \frac{\partial F}{\partial f_j} = w_j.
        \end{align*}
        Each term in $F$ is non-negative and vanishes only when $f_j = \Tilde{f}_j$. Hence we have the result.
    \end{proof}
    

   It is known that when $B$ is VL-stable, i.e., $x^T DBx>0$ for some positive diagonal matrix $D$ for every $x\in\R^N$, there exists $\ft$ such that ${dF}/{dt} < 0$, resulting a global attraction to $\ft$ (Chapter 15 of~\cite{hofbauer2003evolutionary}). Likewise, in Lokta-Volterra type systems, convexity of $-E$ is assumed to make the entropy, or KL-divergence, be a Lyapunov functional to show the existence and the uniqueness of globally attracting equilibrium state~\cite{attouch2004regularized}~\cite{liu2015entropy}~\cite{jabin2017non}.

The goal in this chapter is to prove convergence of every bounded solution, without assuming convexity of $-E$, where the relative entropy is no longer monotone. We provide two strategies. One is to control possible oscillation of entropy using the monotonicity of $E$, or the entropy traping method. Another is to de-singularize the Shahshahani metric by changing the variables.
\section{Entropy trapping method}
Our goal here is establishing the convergence of every trajectory of \eqref{eq1} without assuming positive definiteness of $(B_{ij})$. In this case the entropy $F$ is no longer necessarily monotone in $t$. The key  is to adapt an idea from work of Akin \& Hofbauer on replicator equations \cite{akin1982recurrence}, and use the monotonicity of $E$ to control the oscillation of $F$. 
In particular, the key is to show that for some $C >0$,
    \begin{equation*}
        \frac{dF}{dt} < C\frac{dE}{dt}.
    \end{equation*}
Our main result on convergence is as follows. 
Recall that the $\omega$-limit set of a trajectory $f$ is the set
\[\omega(f):=\{\ft\in\mathbb{R}^N : \exists \text{ a sequence }  t_k \rightarrow\infty \text{ with } f(t_k) \rightarrow \ft\}\,.
\]
    \begin{theorem}\label{thm1}
        Assume $(B_{ij})$ is a symmetric matrix. If $f(t)$ is a bounded solution for (\ref{eq1}), then the $\omega$-limit set $\omega(f)$ consists of a single point, i.e., $\lim_{t\to\infty} f(t)$ exists. 
    \end{theorem}
    \begin{proof}
        Fix $\Tilde{f} \in \omega(f)$ and the weight on the entropy $w_i=1/d_i$. Define, 
        
        \begin{flalign*}
        K &:= \{i\in[N] \; :\; a_i - \sum_j B_{ij}\Tilde{f}_j\neq 0 \},\\
            Q(f) &:= \sum_{i\in K} f_i \\
            Z(f) &:= \min_{i\in K}\{ d_i(a_i - \sum_j B_{ij}f_j)^2\}.
        \end{flalign*}
        If $K$ is empty, define $Q=0$ and $Z=1$. Then
        
        \begin{equation}\label{energy}
            \frac{dE}{dt} = \sum_i d_if_i(a_i - \sum_j B_{ij}f_j)^2 \geq Z(f)Q(f).
        \end{equation}
Because $Z(\Tilde{f}) = z >0$, $\{f \;:\; Z(f) > \frac{z}{2}\}$ is a neighborhood of $\Tilde{f} = \bigcap_{\epsilon>0}\{f \;:\; F(f) \leq \epsilon \}$. So we can find $\epsilon^* > 0$ such that 
        
        \begin{equation}\label{Z}
            F(f) \leq \epsilon^* \;\; \Longrightarrow \;\;Z(f) > \frac{z}{2}.
        \end{equation}
        
        On the other hand,
        
        \begin{align*}
            \frac{dF}{dt} &= \sum_i (f_i - \Tilde{f}_i)(a_i - \sum_j B_{ij}f_j) \\
            &=2(\sum_i f_i a_i - \frac{1}{2}\sum_{i,j}f_iB_{ij}f_j) - \sum_i f_i a_i - \sum_i \Tilde{f}_ia_i + \sum_{i,j}\Tilde{f}_i B_{ij} f_j\\
            &\leq 2(\sum_i \Tilde{f}_i a_i - \frac{1}{2}\sum_{i,j}\Tilde{f}_iB_{ij}\Tilde{f}_j) - \sum_i f_i a_i - \sum_i \Tilde{f}_ia_i + \sum_{i,j}\Tilde{f}_i B_{ij} f_j \\
            &= \sum_i (\Tilde{f}_i - f_i)(a_i - \sum_j B_{ij}\Tilde{f}_j)\\
            &= -\sum_{i\in K}f_i(a_i - \sum_j B_{ij}\Tilde{f}_j)\\
            &\leq Q(f) \max_{i\in K}|a_i - \sum_j B_{ij}\Tilde{f}_j|
        \end{align*}
        Here we used monotonicity of E and that $\Tilde{f}_i = 0$ when $i\in K$. Let $M = \max_{i\in K}|a_i - \sum_j B_{ij}\Tilde{f}_j|$. Then
        \begin{equation}\label{entrop}
            \frac{dF}{dt} \leq MQ(f).
        \end{equation}
        Now choose a sequence of times $\{t_n\}$ approaching infinity so that $f(t_n)$ approaches $\Tilde{f}$ and 
        \begin{equation*}
            \lim_{n\to\infty}F(f(t_n)) = 0.
        \end{equation*}
        For $\epsilon < \epsilon^*$ pick $N=N(\epsilon)$ such that
        \begin{equation}
            \begin{aligned}\label{cauchy}
             F(f(t_N)) < \frac{\epsilon}{2} \\
            \frac{M}{z}[E(\Tilde{f})-E(f(t_N))] < \frac{\epsilon}{4}.
            \end{aligned}
        \end{equation}
        
        Note that as long as a solution path stays inside the $\epsilon^*$-neighborhood defined by F, we can compare the rate of entropy with that of energy by combining (\ref{energy}), (\ref{Z}) and (\ref{entrop}), and obtain
        
        \begin{equation*}
            \frac{dF}{dt} < \frac{2M}{z}\frac{dE}{dt}.
        \end{equation*}
        So let $t>t_N$ so that $\forall s\in[t_N,t]$
        \begin{equation*}
            F(f(s)) \leq \epsilon^*.
        \end{equation*}
        Then by integration and (\ref{cauchy})
        \begin{equation}\label{boots}
            0 \leq F(f(t)) < F(f(t_N)) + \frac{2M}{z}[E(f(t)) - E(f(t_N))]<\epsilon.
        \end{equation}
        Because $\epsilon<\epsilon^*$, $f(t)$ remains in the $\epsilon^*$-neighborhood for all $t>t_N$ so (\ref{boots}) holds for all such t. As we picked $\epsilon$ arbitrarily, (\ref{boots}) implies
        \begin{equation*}
            F(f(t)) \longrightarrow 0 \;\;\; \text{as} \;\;\; t\longrightarrow \infty,
        \end{equation*}
        i.e., $f(t)$ approaches $\Tilde{f}$ as $t\to\infty$, and so $\Tilde{f}$ is the unique limit point.\qedhere
    \end{proof}
\section{De-singularized gradient structure of Lokta-Volterra systems}\label{s: desing}
It is quite intuitive to find the energy $E$ that drives Lotka-Volterra system. For convergence analysis, however, the underlying metric for the gradient structure, the Shahshahani metric, is difficult to work with, due to the singular structure. In~\cite{jabin2017non} Jabin and Liu studies slightly different form of infinite dimensional Lotka-Volterra equations, which is a gradient flow of the energy with respect to $L^2$ inner product. In fact, after the change of variable, the equation is recovered to the original Lotka-Volterra equation. We adapt this change of variables to de-singularize the Shahshahani metric to prove convergence of the bounded solutions of Lokta-Volterra systems, and, furthermore, the regularized Lotka-Volterra systems~\cite{attouch2004regularized}. 

We will study convergence problems of the infinite dimensional Lotka-Volterra system~\cite{jabin2017non} that Jabin and Liu studied in the next section.

\paragraph{\textbf{De-singularized gradient structure}}
Since every solution $f$ of~\eqref{eq1} will stay positive, the system of its square root will correspond to that of $f$ \cite{jabin2017non}. Let $u_i:=\sqrt{f_i}$ then \eqref{eq1} is equivalent to the following equation,
\begin{align}\label{e: lot u}
    u_i'&=\frac{1}{2}d_iu_i(a_i-\sum_{j\in[N]}B_{ij}u_j^2),
\end{align}
 We can see the equation is the gradient flow with respect to the Euclidean inner product of the following energy,
\begin{align*}
    H(u)&:=-\frac{1}{4} \sum_{i\in[N]}u_i^2 a_i + \frac{1}{8}\sum_{i,j\in[N]}B_{ij}u^2_i u^2_j\\
   DH(u)(v) &=-\sum_{i\in[N]}\frac{1}{2}u_i(a_i-\sum_{j\in[N]}B_{ij}u^2_j)v_i.
\end{align*}
Therefore,
$$u'=-\dg{d}\nabla H(u),$$
where $\dg{d}$ is the diagonal matrix constructed by $d$. Simply, one can check the angle condition. So we can apply \L ojasiewicz convergence theorem.
\begin{theorem}\label{t: lv conv}
    Every bounded solution $u(t)$ of \eqref{e: lot u} converges and the convergence rate depends on the \L ojasiewicz exponent of $E$ at the limit point $\Tilde{u}$,
    \begin{equation}
        \Vert u(t) - \Tilde{u}\Vert = 
        \left\{\begin{array}{ll}
             O(e^{-\delta t}) & \text{if}\;\; \theta=\frac{1}{2}, \;\;\;\;\text{for some}\;\; \delta>0 \\
             O(t^{-\theta/(1-2\theta)})& \text{if}\;\; 0<\theta<\frac{1}{2} 
        \end{array}\right.
    \end{equation}
\end{theorem}
\paragraph{\textbf{Regularized Lokta-Volterra systems}}
In fact the idea of scaling can be applied to a general class of ODE, suggested by Attouch and Teboulle~\cite{attouch2004regularized}. For a given energy $H:\R^N\to\R$ and $\mu,\nu>0$, the solution of regularized Lotka-Volterra equations $f\in\R^N$ obeys the following differential equations,

[Regularized Lotka-Volterra equations]
\begin{equation}\label{e: reg lv}
    f'_i=-\frac{f_i}{\mu+\nu f_i}\nabla H(f)_i.
\end{equation}
The regularized Lotka-Volterra equations~\eqref{e: reg lv} arise from the proximal-like iteration scheme as follows,
\begin{equation*}
    x_k=\arg\min_{x\in\R^N}\{H(x)+\tau_k F(x|x_{k-1})\},
\end{equation*}
where $F(x|y)$ is a distance-like function, or a relative entropy that gives proximal distance to x from y. In this case, $F(x|y)$ is a logarithmic-quadratic function, or the regularized Kullback-Leibler divergence,
\begin{equation*}
    F(x|y):=\frac{\nu}{2} \Vert x-y\Vert^2+\mu\sum_i y_i\log\frac{y_i}{x_i}+x_i-y_i.
\end{equation*}
Similar to the original Lokta-Volterra system, when $x\mapsto H(x)$ proper convex and $y=\arg\min_x H(x)$, $F(x|y)$ becomes a Lyapunov function of the system. Attouch and Teboulle used this idea, under the convexity assumption, to prove convergence of a bounded solution.

\paragraph{\textbf{De-singularization}}
According to Theorem 3.1~\cite{attouch2004regularized}, if $H\in C^2(\R^N)$ then the equation~\eqref{e: reg lv} is well-posed and any solution initiated in $\R^N_+$ stays in $\R^N_+$. 
By letting $u_i:=\sqrt{f_i}$ for all $i$, $(\hat{\mu},\hat{\nu}):=(4\mu,4\nu)$ and $\hat{H}(u):=H(u^2)$, we have
\begin{align}\label{e: reg lv u}
    u_i'&=-\frac{1}{\hat{\mu}+\hat{\nu} u^2_i}\nabla \hat{H}(u)_i
\end{align}
Thus, the equation~\eqref{e: reg lv u} is a gradient flow of $\hat{H}$ with respect to a metric $\langle v,w\rangle_u=\sum_iv_iw_i(\hat{\mu}+\hat{\nu} u_i^2)$.
Since the metric is stable when the solution $u(t)$ is bounded, we can perform \L ojasiewicz convergence theorem.
\begin{theorem}
    Assume that $\hat{H}:\R^N_+\to\R$ satisfies \L ojasiewciz inequality. Then every bounded solution $u(t)$ of \eqref{e: reg lv u}, or equivalently every bounded solution $f(t)$ of~\eqref{e: reg lv}, converges as $t\to\infty$ and the convergence rate depends on the \L ojasiewicz exponent of $\hat{H}$ at the limit point $\Tilde{u}$,
    \begin{equation}
        \Vert u(t) - \Tilde{u}\Vert = 
        \left\{\begin{array}{ll}
             O(e^{-\delta t}) & \text{if}\;\; \theta=\frac{1}{2}, \;\;\;\;\text{for some}\;\; \delta>0 \\
             O(t^{-\theta/(1-2\theta)})& \text{if}\;\; 0<\theta<\frac{1}{2} 
        \end{array}\right.
    \end{equation}
\end{theorem}

\section{Nonlocal semi-linear heat equation}\label{s: slheat}
We aim to extend the convergence result of Jabin and Liu on the infinite dimensional Lotka-Volterra equations~\cite{jabin2017non}, where they use the monotonicity of KL-divergence under the convexity assumption on the energy. Let the ecosystem, $\Omega$ be a bounded domain with smooth boundary in $\R^N$ and assume that the population of $x$-species at time $t>0$, $u(t,x):\R_+\times\Omega\to\R$, satisfies the following equations.
\begin{equation}
    \begin{aligned}
     \partial_tu(t,x) &=\Delta u(t,x) +\frac{1}{2}u(t,x)\left( a(x)-\int_\Omega b(x,y)u^2(t,y)dy\right) \label{e: slh} \\
u(0,x) &=u_0(x) \ge0\quad x\in\Omega\\
\frac{\partial u}{\partial \nu} &=0 \quad x\in \partial\Omega
    \end{aligned}
\end{equation}
In the model, coefficient $a(x)$ is the intrinsic growth rate of $x$-species, and
$b(x,y)=b(y,x)$ represents the interaction between $x$-species and $y$-species, while the diffusion
term plays certain role of mutations in the population dynamics. We assume the following,
\begin{ass}\label{a: ab}
\begin{align*}
    a\in L^\infty(\Omega), \quad b(x,y)\in L^\infty(\Omega\times\Omega). 
\end{align*}
\end{ass}

As mentioned in~\cite{jabin2017non} and Section~\ref{s: desing}, \eqref{e: slh} is equivalent to the following with the change of variables $u^2=f$,
\begin{equation}\label{e: inf lv f}
    \partial_t f(t,x)=\Delta f -\frac{|\nabla f|^2}{2f} +f(t,x)\left( a(x)-\int_\Omega b(x,y)f(t,y)dy\right),
\end{equation}
which, without the mutation term, corresponds to the Lotka-Volterra equations. 
\eqref{e: inf lv f} is a gradient flow with respect to the Shahshahani metric $\langle g, h\rangle_f:=\int gh/fdx$ of the energy,
\begin{equation*}
    E(u) :=\frac{1}{2} \iint b(x,y) f(t,x)f(t,y)dxdy -\int a(x) f(t,x)dx +\frac{1}{2}\int \frac{|\nabla f|^2}{f}dx.
\end{equation*}
\paragraph{\textbf{Gradient structure}}
We consider Hilbert spaces $\mathcal{V}:=H^1(\Omega)$ and $\mathcal{H}:=L^2(\Omega)$ with continuous dense imbedding
\begin{equation*}
    \mathcal{V}\subset \mathcal{H}\subset \mathcal{V}',
\end{equation*}
where $\mathcal{V}'$ is the topological dual of $\mathcal{V}$, using the $\mathcal{G}$ inner product, $\langle \cdot, \cdot\rangle$. Note that $\mathcal{H}$ is dense in $\mathcal{V}'$, for any $v'\in \mathcal{V}'$, there is a sequence $\{v'_n\}\subset H$ so that for any $v\in \mathcal{V}$,
\[
\lim_{n\to\infty} \langle v_n',v\rangle_H =v'(v).
\]
Comparing to \eqref{e: inf lv f}, \eqref{e: slh} has a de-singularized gradient structure of the energy $H$,
\[
H(u) :=\frac{1}{8} \iint b(x,y) u^2(t,x)u^2(t,y)dxdy -\frac{1}{4}\int a(x) u^2(t,x)dx +\frac{1}{2}\int |\nabla u(t,x)|^2dx.
\]
For simplicity, we write 
\begin{equation*}
    B(u)(x):=\int b(x,y)u(y)dy.
\end{equation*}
We can check the $L^2$ gradient of $H$ corresponds to the right hand side of~\eqref{e: slh},
\begin{align*}
    DH(u)(v)&= \frac{1}{2} \int B(u^2)u(x)v(x)dx -\frac{1}{2}\int a(x) u(x)v(x)dx + \int \nabla u(x) \nabla v(x) dx\\
    &=:\langle \nabla H(u), v\rangle_{L^2}.
\end{align*}
Therefore,
\begin{equation*}
    \partial_t u= -\nabla H(u).
\end{equation*}
This gives the energy dissipation.
\begin{align*}
    \frac{d}{dt}H(u) = -\int |\partial_t u|^2 dx \leq 0
\end{align*}
\mnote{the condition on $b$ allows us to perform Young's convolution inequality}
\paragraph{\textbf{Analytic energy $H$}}
\begin{lemma}
Under Assumption~\ref{a: ab}, $H$ is analytic in $\mathcal{V}=H^1(\Omega).$
\end{lemma}
\begin{proof}
        It is enough to show the following,
        \begin{align*}
            \int a(x) u(x) v(x)dx&\leq \Vert a\Vert_\infty \Vert u\Vert_{L^2} \Vert v\Vert_{L^2}\\
            \int |\nabla u(x)\nabla v(x)| dx&\leq \Vert u\Vert_{H^1} \Vert v\Vert_{H^1}\\
            \iint b(x,y)u(x)v(x)w(y)z(y)dxdy&\leq  \sup_y \int |b(x,y)| dx \Vert u\Vert_{L^2} \Vert v\Vert_{L^2} \Vert w\Vert_{L^2} \Vert z\Vert_{L^2}.
        \end{align*}
\end{proof}
\paragraph{\textbf{\L ojaisewicz inequality}}
\begin{ass}\label{a: bm}
\begin{equation*}
    b_m:=\inf_{x,y}b(x,y)>0
\end{equation*}
\end{ass}
Under Assumption~\ref{a: ab} and Assumption~\ref{a: bm}, Jabin and Liu showed that there exists a nonnegative weak solution $\ut\in H^1$ for $DH(\ut)=0$ using Dirichlet's principle (Energy method) and calculus of variation (Theorem 1.2 of~\cite{jabin2017non}). With this knowledge, we aim to show that the \L ojasiewicz inequality holds for such equilibrium $\ut$. We use the framework presented in Section~\ref{c: inf dim} of Chapter~\ref{ch: gf}. Firstly, we check the Hessian of $H$. 
\begin{align*}
    DH(u)(v) &= \frac{1}{2} \int B(u^2)u(x)v(x)dx -\frac{1}{2}\int a(x) u(x)v(x)dx + \int \nabla u(x) \nabla v(x) dx\\
    D^2H(u)(v,w) &= \frac{1}{2} \int B(u^2)w(x)v(x)dxdy + \int B(uw)u(x)v(x)dx\\ &-\frac{1}{2}\int a(x) w(x)v(x)dx + \int \nabla w(x) \nabla v(x) dx.
\end{align*}

Now for a fixed equilibrium $\ut\in H^1$, define, $A:\mathcal{V}\to\mathcal{V}'$ so for any $v\in\mathcal{V},$
 \begin{align*}
     \langle Av,w\rangle :=D^2H(\ut)(v,w)
 \end{align*}
 So in the weak sense,
 \begin{align*}
     Av&=\frac{1}{2}(B(\ut^2)-a)v+ \ut B(\ut v)-\Delta v\\
     &=: G(v)-\Delta v.
 \end{align*}
 \begin{prop}\label{p: asfn}
     $A:\mathcal{V}\to\mathcal{V}'$ is a semi-Fredholm operator.
 \end{prop}
 \begin{proof}
         Firstly, $G: V\to V'$ compact since $G\in \mathcal{L}(V,H)$ and the imbedding $\text{Id}: H\to V' ; v\mapsto v$ is compact. So by Theorem
         ~\ref{t: fredcom} it is enough to show that $-\Delta:V\to V'$ is a semi-Fredholm operator.
 
 [Finite dimensional kernel]
 \begin{align*}
     -\Delta u&=0 \quad \text{in } \Omega\\
     \frac{\partial u}{\partial \nu}&=0 \quad \text{on } \partial \Omega
 \end{align*}
 if and only if $u$ is a constant. Thus $N(-\Delta)$ is one-dimensional.
 
 [Closed range] By Fredholm alternative,
 \begin{align*}
     -\Delta u&=f \quad \text{in } \Omega\\
     \frac{\partial u}{\partial \nu}&=0 \quad \text{on } \partial \Omega
 \end{align*}
 has a solution if and only if $\int_\Omega fdx =0$.
 If we choose $g\in \overline{R(-\Delta)}$, there is a sequence $\{u_n\}$ such that $-\Delta u_n\to g$ in $\mathcal{V}'$, which implies $0=\lim_{n\to\infty}\int 
 -\Delta u_ndx =\int gdx$. Therefore, the range is closed.
 \end{proof}
\paragraph{\textbf{Convergence result}}
Note that in the competitive Lotka-Volterra systems, every solution is bounded and converges (Theorem~\ref{t: lv conv}). Thanks to the smoothing effect of parabolic equations, every solution is precompact in $\mathcal{V}$. So we can prove the convergence of every solutions of the infinite dimensional Lokta-Volterra equations with mutation.
\begin{theorem}[Convergence of competitive Lokta-Volterra systems with mutation]
    Under Assumption~\ref{a: ab} and Assumption~\ref{a: bm}, any solution $u(x,t)$ of \eqref{e: slh} converges as $t\to\infty$.
\end{theorem}
\begin{proof}

[Well-posedness]
We use Theorem 3.3.3 and Theorem 3.5.2 in~\cite{henry2006geometric}, which shows, along with the well-posedness, $u(t)\in C^1(\R_+,\mathcal{V)}$, i.e., smoothing action of semilinear parabolic equations. According to the theorem, in order to achieve parabolic regularity of the solution of $\eqref{e: slh}$, we need to check if $f: H^1 \to L^2$ in \eqref{e: slh} is locally Lipschitz.
\begin{align*}
    \Vert ub*u^2 -vb*v^2\Vert_{L^2} &\leq \Vert (u-v)b*u^2\Vert_{L^2} +\Vert vb*(u^2 -v^2)\Vert_{L^2} \\
    &\leq \sup_x |\int b(x,y)u^2(y)dy| \Vert u-v\Vert_{L^2}\\
    &+ \sup_x |\int b(x,y)(u(y)-v(y))(u(y)+v(y))dy| \Vert v\Vert_{L^2} \\
    &\leq \Vert b\Vert_\infty \Vert u \Vert^2_{L^2} \Vert u-v \Vert_{L^2} +\Vert b\Vert_\infty \Vert v \Vert_{L^2}\Vert u+v \Vert_{L^2} \Vert u-v \Vert_{L^2} 
\end{align*}
[Precompactness]
Due to the smoothing effect of the semilinear parabolic equations, it is enough to check $L^2$ bound to achieve higher order bound. We use Theorem 3.3.6 in~\cite{henry2006geometric} and $L^2$ bound achieved in Theorem 1.1 in~\cite{jabin2017non}.
\newline
[\L ojasiewicz convergence]
As previously discussed, $H$ is analytic in $\mathcal{V}$ and has semi-Fredholm Hessian operator from $\mathcal{V}$ to $\mathcal{V}'$.
We use Theorem~\ref{t: vloja} and Theorem~\ref{t: conv} to conclude.
\end{proof}

\chapter{Convergence of discrete-time Lotka-Volterra dynamics}
  In this chapter we form the time-discrete Lotka-Volterra equations inspired by DC algorithms, or convex splitting (Chapter~\ref{c: disc time cs}). We provide two different analysis, one follows the Shahshahani gradient and another follows regular Euclidean gradient constructed by change of variables. Accordingly, we prove convergence with different strategies for each method, without convexity of energy: the entropy trapping method and the \L ojasiewicz theorem, respectively.
  
  As we discussed in the continuous-time case, the singular nature of the Shahshahani metric makes convergence analysis challenging. So we prove the convergence by relating the energy with the entropy to show that the solution is trapped in the neighborhood of the equilibrium. Due to the time discrepancy caused by the discrete time steps, however, the KL-divergence that we used as the entropy, does not work in the same way. We add a quadratic term to resolve the issue. With the new entropy, we proceed with the entropy trapping method to prove convergence of every solution of competitive Lotka-Volterra systems. In fact, the time-discrete Lotka-Volterra equation with the Shahshahani metric can be seen as Euler's method for regularized Lotka-Volterra equation of the form~\eqref{e: reg lv}. This can be considered as an extension of the convergence theorem that is presented in~\cite{attouch2004regularized}, where Attouch and Teboulle assume the convexity of the energy to prove convergence. 
  
  On the other hand, adapting the observation on changing variables presented by Jabin and Liu~\cite{jabin2017non}, we, again, resolve the singularities of Shahshahani metric and apply the \L ojasiewicz convergence theorem by Absil \textit{et al}.~\cite{absil2005convergence} to the Lotka-Volterra system. We suggest a certain way to convex-split the energy, with which the implicit part of the algorithm is not too complex to invert, yet the solution stays in the positive orthant $\R^N_+$ regardless of the time step $\tau>0.$ Thus the convergence and the positivity of the solutions are independent of the time step $\tau$.
  
    \section{Application of entropy trapping method}
   \subsection{Time-discrete competitive Lotka-Volterra}
   We call a Lotka-Volterra system \textit{competitive} when the associated interaction matrix $B$ has only positive entries, in other words $B_{ij}>0$ for all $i,j\in [N]$. It is well known that every solution of a competitive Lotka-Volterra system is bounded (Chapter 15 of~\cite{hofbauer2003evolutionary}). Additionally when $B$ is $D$-symmetrizable, every solution converges. Thus it is natural to formulate the time-discrete Lotka-Volterra system with the same behavior.
   
   In this section we will formulate a discrete-time Lotka-Volterra with competitive interaction and apply entropy trapping method to analyze the behavior of its solution, without de-singularizing Shahshahani metric structure. 
   \paragraph{\textbf{Set-up}}
     We formulate a discrete-time Lotka-Volterra system inspired by Eyre's semi-implicit Euler's method~\eqref{e: disc cs}, but with the gradient with respect to Shahshahani metric. Let's take $H(f)=-E(f) = \frac{1}{2}f^TBf -a^Tf$ and let $\lambda \in \mathbb{R}$ be bigger than the Perron root of $B$, so that $B = \lambda - A$, where $A = \lambda - B$ is a positive definite matrix. We split $H(f)$ as follows,
    \begin{align*}
        H(f) &= \frac{1}{2}f^TBf - a^Tf \\
            &= \frac{1}{2}f^T\lambda f - (\frac{1}{2}f^TAf + a^Tf) \\
            &=: H_+(f) - H_-(f).
    \end{align*}
    Now we evaluate the gradient of $H_\pm$ with respect to Shahshahani metric at $f^n$, and apply semi-implicit Euler's method as follows,
    \begin{align}
        (f^{n+1} - f^n)_i &= -\tau d_i(\text{grad}_{f^n}H_+(\f)-\text{grad}_{f^n}H_-(f^n))\notag\\
        &= -\tau d_if^n_i(\nabla H_+(f^{n+1}) - \nabla H_-(f^n))_i\label{disc time dyn H} \\
            &= \tau d_if^n_i(a- Bf^n -\lambda(f^{n+1} - f^n))_i\label{disc time dyn}
    \end{align}
    Our goal here is to ensure the dynamics of \eqref{disc time dyn} to be similar with its continuous version as much as possible. Firstly, we need to make sure the sequence $\{f^n\}_{n\in\N}\subset \R^N_+$ with appropriate choice of a time step $\tau$. It turns out when the interaction is competitive, i.e., $B_{ij}> 0$,
    we will have the desired properties as follows.
     \begin{prop}\label{p: properties}
     Assume that $\underline{B}:=\min_{i,j}B_{ij}>0$. For any $\epsilon>0$, define $$\Omega_\epsilon:=\{f\in\R^N_+ : \Vert f\Vert_\infty \leq (\lambda^{-1}+\underline{B}^{-1})\Vert a\Vert_\infty +\epsilon\},$$ then there exists $M_\epsilon>0$, such that for $0< \tau < \frac{1}{\Vert d\Vert_\infty M_\epsilon}$ the sequence $\{f^n\}$ generated by the discrete scheme (\ref{disc time dyn}) with initial condition $f^0\in \Omega_\epsilon$ has properties as follows:
     \begin{itemize}
         \item (Feasibility) $f^n\in\Omega_\epsilon$ for all $n\in\N$,
         \item (Ratio bound)
         \begin{equation*}
             1-\tau \Vert d\Vert_\infty M_\epsilon \leq \frac{\f}{f^n}\leq 1+\tau\Vert d\Vert_\infty M_\epsilon
         \end{equation*}
         \item (Monotonicity) E is non-decreasing along  (\ref{disc time dyn}), i.e., for all $n\in \N_0$
         \begin{equation*}
             E(f^{n+1}) - E(f^n) \geq 0, 
         \end{equation*}
    and the equality holds if and only if $f^n_i = 0$ or $a_i = (Bf^n)_i$ for all $i$,
        \item ($l^2$-bound)
        \begin{equation*}
            \sum_{n=k}^\infty \sum_{i=1}^N \frac{1}{\tau d_i f^n_i}(\f-f^n)^2_i \leq \lim_{n\to\infty} E(f^n)-E(f^k).
        \end{equation*}
     \end{itemize}
    \end{prop}
    Note that the restriction on $\tau$ is to satisfy the natural constraint on Lotka-Volterra system to make the solution $\{f^n\}_{n\in\N_0}$ stay in the positive orthant $\R^N_+$. Meanwhile, the monotonicity and $l^2$-bound holds regardless of $\tau$, due to the convex splitting.
    
    We use these properties to prove convergence of bounded solutions using the entropy trapping method. The proof follows by lemmas with more detail.
    \begin{lemma}[Feasibility]\label{l: feasible}
        For any $\epsilon>0$, define 
        $$M_\epsilon:=\max_{f\in\Omega_\epsilon}\Vert a-Bf\Vert_\infty.$$ 
        
        Then $\Omega_\epsilon$ is an invariant set of \eqref{disc time dyn} when $0<\tau<\frac{1}{\Vert d\Vert_\infty M_\epsilon}$.
    \end{lemma}
    \begin{proof}
    Essentially, we need to show that the sequence $\{f^n\}_{n\in\N}$ stays positive and bounded. 
   Let's rewrite \eqref{disc time dyn} as follows
    \begin{equation}\label{sq bd}
        (1+ \tau\lambda d_i f^n_i) (\f-f^n)_i = \tau d_i f^n_i(a - Bf^n)_i,
    \end{equation}
     \begin{equation}\label{sq pos}
        \f_i = f^n_i \Big(\frac{1+\tau\lambda d_i f^n_i + \tau d_i(a - Bf^n)_i}{1+\tau\lambda d_i f^n_i}\Big).
    \end{equation}
   
    [Positivity]
     According to \eqref{sq pos}, given that $f^n>0$, $1+\tau d_i(a-Bf^n)_i>0$ will make $\f>0$. This will allow us to control the numerator when $f^n_i$ is negligible but $(a-Bf^n)_i$ is possibly negatively large. 
    Since we chose $\tau<\frac{1}{\Vert d\Vert_\infty M_\epsilon}$, for all $i$,
    \begin{align}\label{e: 1+tau}
    1+\tau d_i(a-Bf^n)_i&\geq 1-\tau \Vert d\Vert_\infty M_\epsilon>0.
    \end{align}
    
    [Boundedness]
    For any $f\in\Upsilon:=\{f\in \R^N_+:\Vert f\Vert_\infty> \frac{\Vert a\Vert_\infty}{\underline{B}} \}$ and for all $i\in [N]$
    \begin{align*}
        (a-Bf)_i &\leq a_i -\underline{B}\sum_j f^n_j \leq a_i -\underline{B}\Vert f^n\Vert_\infty<0.
    \end{align*}
    This estimate works because every component of $B$ and $f$ is positive and additionally $\sum_j f_j>\Vert f\Vert_\infty$. So from \eqref{sq bd}, we can tell that if $f^n\in\Omega_\epsilon\cap \Upsilon$, $\f\in\Omega_\epsilon$.
    
    On the other hand, for $f^n\in \Omega_\epsilon\setminus\Upsilon$ we want to show that the jump is not too big, so that $\f\in\Omega_\epsilon$. Indeed, from \eqref{sq bd} and positivity
    \begin{align*}
        0<\f_i&=f^n_i+\frac{\tau d_if^n_i}{1+\tau\lambda d_if^n_i}(a-Bf^n)_i\\
        &< f^n_i +\frac{1}{\lambda}\Vert a\Vert_\infty\\
        &\leq (\frac{1}{\underline{B}}+\frac{1}{\lambda})\Vert a\Vert_\infty.
    \end{align*}
    Therefore, $\f\in\Omega_\epsilon$ and we have the desired result.
    \end{proof}
    Unlike the continuous-time version, it is tricky to deal with the discrepancy between current and the next step. It turns out a small enough time step $\tau$ provides a restriction on the ratio between two consecutive steps.
    \begin{lemma}[Ratio bound]
    Under the same assumption from Lemma~\ref{l: feasible}, 
    $$ 0<1-\tau \Vert d\Vert_\infty M_\epsilon \leq \frac{\f}{f^n}\leq 1+\tau\Vert d\Vert_\infty M_\epsilon.$$
    \end{lemma}
    \begin{proof}
    From \eqref{sq bd} we have,
    \begin{equation*}
        \Big|\frac{\f}{f^n}- 1\Big|= \Big|\frac{\tau d_i(a-Bf^n)_i}{1+\tau\lambda d_i f^n_i}\Big|\leq \tau\Vert d\Vert_\infty M_\epsilon.
    \end{equation*}
Therefore,
    \begin{align*}
        0<1-\tau \Vert d\Vert_\infty M_\epsilon \leq \frac{\f}{f^n}
        \leq 1+\tau\Vert d\Vert_\infty M_\epsilon,
    \end{align*}     
    where $M_\epsilon$ is from Lemma~\ref{l: feasible}.
    \end{proof}
    \begin{lemma}[Monotonicity]\label{eq cond} E is non-decreasing along the evolution of equation (\ref{disc time dyn}). In other words,
    \begin{equation*}
        E(f^{n+1}) - E(f^n) \geq 0 \qquad \forall n\in\N,
    \end{equation*}
    and equality holds iff for all $i$, $f^n_i = 0$ or $a_i = (Bf^n)_i.$
    \end{lemma}
    
    \begin{proof}
        From Lemma~\ref{l: conv ineq}
        \begin{align*}
            E(f^{n+1}) - E(f^n) \;\; &=-(H(\f)- H(f^n))\\
            &\geq \;\;-\langle \nabla H_+(f^{n+1}) - \nabla H_-(f^n), f^{n+1}-f^n\rangle \\
                &= \;\; \tau\sum_i d_i f^n_i \big(\nabla H_+(\f) - \nabla H_-(f^n)\big)^2_i\geq 0.
        \end{align*}
        Suppose the equality holds. Then the middle term of the inequality vanishes. So for all $i$, $f^n_i = 0$ or $\big(\nabla H_+(\f) - \nabla H_-(f^n)\big)_i = 0$, which implies $\f_i -f^n_i = 0$ by (\ref{disc time dyn H}). Thus, by (\ref{disc time dyn}), for all $i$, $f^n_i = 0$ or $a_i = (Bf^n)_i$.
        
        Conversely, if for all $i$, $f^n_i = 0$ or $a_i = (Bf^n)_i$, enough to show $\f_i=f^n_i$ for the latter case. For such i,
        \begin{align*}
            (f^{n+1} - f^n)_i &= \tau d_i f^n_i(a- Bf^n -\lambda(f^{n+1} - f^n))_i\\
                &= -\tau d_i f^n_i(\lambda(f^{n+1} - f^n))_i
        \end{align*}
        \begin{equation*}
            (1 + \tau \lambda d_i f^n_i)(f^{n+1} - f^n)_i = 0
        \end{equation*}
        implies $(f^{n+1} - f^n)_i = 0$. Since $\f=f^n$, the equality holds.
    \end{proof}
    \begin{lemma}[$l^2$ bound]\label{l: l2 bound}
    $\sum_{n=k}^\infty \sum_{i=1}^N \frac{1}{\tau d_i f^n_i}(\f-f^n)^2_i \leq \liminf_n E(f^n)-E(f^k).$
    \end{lemma}
    \begin{proof}
     From (\ref{disc time dyn H}),
    \begin{align*}
        E(\f)-E(f^n) &\geq -\langle \nabla H_+(\f) -\nabla H_-(f^n), \f - f^n\rangle \\
            &= \sum_{i=1}^N  \frac{1}{\tau d_i f^n_i}(\f-f^n)^2_i.
    \end{align*}
    For $m>k$ using telescoping sum,
    \begin{align*}
        \sum_{n=k}^m \sum_{i=1}^N \frac{1}{\tau d_i f^n_i}(\f-f^n)^2_i &\leq E(f^k) -E(f^{m+1}).
    \end{align*}
    Thus
    \begin{equation*}
         \sum_{n=k}^\infty \sum_{i=1}^N \frac{1}{\tau d_i f^n_i}(\f-f^n)^2_i \leq E(f^k) - \lim_{n\to\infty} E(f^n).\qedhere
    \end{equation*}
    \end{proof}
   
    \subsection{Convergence result}
    We tailored the discrete-time equation to make $E$ increasing. In order to prove convergence, we need a suitable entropy function too. We modify the previous entropy $F$ by adding quadratic terms.
    
    \begin{definition} For fixed $\Tilde{f}$ we define
        \[ 
        F(f^n) := \sum_i w_i(\ft_i \log\frac{{\ft_i}}{f^n_i} +f^n_i - \ft_i) + \sum_i \frac{\lambda\tau}{2}(f^n_i - \ft_i)^2
        \]
    \end{definition}
    
    The modified entropy still has nice properties, convexity and having a unique minimizer. In fact the entropy is a sum of two positive convex functions because the first term can be represented as
    $$ \sum_i w_i\ft_i(\frac{f^n_i}{\ft_i} - \log\frac{f^n_i}{\ft_i}-1)$$
    with knowing that $x -\log x -1$ is a convex function having minimum value 0 at 1. So we have $F \geq 0$, moreover $F(f) = 0$ if and only if $f=\ft$.
    \begin{theorem} 
     Let $\{f^n\}$ be a sequence generated from  \eqref{disc time dyn}. For any $\epsilon>0$, there exists $M_\epsilon$ such that if $0<\tau<\frac{1}{\Vert d\Vert_\infty M_\epsilon}$ then $\lim_{n\to\infty}f^n$ exists.
    \end{theorem}
    
    \begin{proof}
    The argument is similar to the proof of Theorem \ref{thm1} but a tad more technical. Let $\{f^n\}_{n=1}^\infty$ be a sequence generated from the difference equation \eqref{disc time dyn}. Thanks to Lemma~\ref{l: feasible}, $\{f^n\}_{n=1}^\infty$ is bounded, so there is a convergent subsequence $\{f^{n_k}\}_{k=1}^\infty$ and $\ft \in \overline{\mathbb{R}_+^N}$ such that $f^{n_k} \rightarrow \ft$ as $k \rightarrow \infty$.
    
    Let $K := \{i\in[N] \; :\; a_i - \sum_j B_{ij}\Tilde{f}_j\neq 0 \}$ and
     \begin{align*}
            Q(f) &:= \sum_{i\in K} f_i, \\
            Z(f) &:= \min_{i\in K}\{ (a_i - \sum_j B_{ij}f_j)^2\}.
    \end{align*}
    If $K$ is empty, define $Q=0$ and $Z=1$. Then, as we have seen in the proof of Lemma~\ref{l: l2 bound}
    \begin{align*}
        E(\f)-E(f^n) &\geq \langle \frac{1}{\tau d f^n}(\f-f^n) , \f -f^n \rangle \\
        &= \sum_{i\in[N]} \frac{\tau d_i f^n_i}{(1+\lambda\tau d_i f^n_i)^2}(a-Bf^n)^2_i \\
        & \geq \sum_{i\in K} \frac{\tau d_i f^n_i}{(1+\lambda\tau d_if^n_i)^2}(a-Bf^n)^2_i\\
        &\geq \frac{1}{C_0}Z(f^n)Q(f^n).
    \end{align*}
    Also note that $Z(\Tilde{f}) = z >0$, so $\{f \;:\; Z(f) > \frac{z}{2}\}$ is a neighborhood of $\ft$. Since $\Tilde{f} = \bigcap_{\epsilon>0}\{f \;:\; F(f) \leq \epsilon \}$, we can find $\epsilon^* > 0$ such that,
    \begin{align*}
        F(f) \leq \epsilon^* \;\; &\Longrightarrow \;\;Z(f) > \frac{z}{2} \text{  and} \max_{i\in \textbf{supp}\ft} \frac{1}{f^n_i} <L,
    \end{align*}
    for some $L>0.$
    
    Let's analyze the difference of the entropy with $w_i=1/d_i$. By Lemma~\ref{l: conv ineq}
    \begin{align*}
        F(\f) -F(f^n)
            &\leq \sum_i \frac{1}{d_i\f_i}(\f_i-\ft_i)(\f_i -f^n_i) + \sum_i \lambda\tau(\f_i - \ft_i)(\f_i -f^n_i) \\
            &= \sum_i \frac{1}{d_i\f_i}(\f_i-f^n_i+f^n_i-\ft_i)(\f_i -f^n_i) \\
            &\;\; + \sum_i \lambda\tau(\f_i -f^n_i+f^n_i - \ft_i)(\f_i -f^n_i) \\
            &= \sum_i (\frac{1}{d_i\f_i}+\lambda\tau)(\f_i -f^n_i)^2 +\sum_i \frac{1}{d_i\f_i}(f^n_i-\ft_i)(\f_i-f^n_i)\\ 
            &\;\; +\sum_i\lambda\tau(f^n_i -\ft_i)(\f_i -f^n)\\
            &=\sum_i (\frac{f^n_i}{\f_i}+\lambda\tau  d_i f^n_i)\frac{1}{d_i f^n_i}(\f_i -f^n_i)^2\\
            &\;\;+\sum_i d_i^{-1}(\frac{1}{\f_i}-\frac{1}{f^n_i})(f^n_i-\ft_i)(\f_i-f^n_i)\\
            &\;\; +\sum_i\frac{1+\lambda\tau d_i f^n_i}{d_i f^n_i}(f^n_i -\ft_i)(\f_i -f^n_i)\\
            &= \text{I} + \text{II} + \text{III}
    \end{align*}
    Let's analyze each term seperately.
    \begin{align*}
    \text{I} &\leq C_1 \sum_i \frac{1}{d_i f^n_i}(\f_i - f^n_i)^2.\\
        \text{II} &= \sum_i \frac{f^n_i - \f_i}{d_if^n_i\f_i}(f^n_i - \ft_i)(\f_i-f^n_i) \\
            &= -\sum_{i\in \{\ft_i = 0\}} \frac{1}{d_i\f_i}(\f_i-f^n_i)^2 - \sum_{i\in\{\ft_i\neq0\}}\frac{f^n_i}{d_i\f_i}(\frac{1}{f^n_i})^2 (f^n_i - \ft_i)(\f_i -f^n_i)^2 \\
            &\leq C_2 \sum_{i\in\{\ft_i\neq0\}} d_i^{-1}(\frac{1}{f^n_i})^2 (\f_i -f^n_i)^2.
    \end{align*}
    $C_1$ and $C_2$ are due to the bound on the domain and the ratio bound.
    To analyze III recall the equation~\eqref{sq bd}
    $$(1+ \tau\lambda d_i f^n_i)(\f_i -f^n_i) = \tau d_i f^n_i(a -Bf^n)_i$$
    \begin{align*}
        \text{III} &= \tau \sum_i (f^n_i -\ft_i)(a-Bf^n)_i\\
            &=\tau( a\cdot f^n - f^n\cdot Bf^n - a\cdot\ft + \ft\cdot Bf^n) \\
            &\leq \tau(a\cdot f^n -2a\cdot f^n +2a\cdot\ft -\ft\cdot B\ft - a\cdot\ft + \ft\cdot Bf^n) \\
            &= \tau( -a\cdot f^n +\ft\cdot B f^n)
    \end{align*}
    because $E(f^n) = a\cdot f^n - \frac{1}{2}f^n Bf^n$ is increasing and $\ft\cdot(a -B\ft) = \sum_i \ft_i(a_i -B\ft_i) =0$. Thus
    \begin{equation*}
        \text{III} \leq \tau \sum_{i\in K} f^n_i(B\ft - a)_i.
    \end{equation*}
    All together, we have
    \begin{equation}\label{disc entro diff}
        F(\f)-F(f^n) \leq (C_1+C_2\max_{i\in\textbf{supp}\ft}\frac{1}{f^n_i}) \sum_i \frac{1}{d_if^n_i}(\f_i - f^n_i)^2
        +\tau Q(f^n)\max_{i\in K}|a_i-B\ft_i|.
    \end{equation}
    Now let $M=\max_{i\in K}|a_i-B\ft_i|$ and for any $\epsilon \in (0,\epsilon^*)$, choose $m>0$ such that $\forall n_k \geq m$
    \begin{align*}
        &F(f^{n_k}) <\frac{\epsilon}{2},\\
        &\sum_{n=m}^\infty\sum_i\frac{1}{d_i f^n_i}(f^{n+1}_i-f^n_i)^2 <\frac{\epsilon}{4C(1+L)},\\
        &\tau\frac{C_0M}{z}[E(\ft)-E(f^{n_k})] <\frac{\epsilon}{8},
    \end{align*}
    where $C=\max\{C_1,C_2\}$.
    \\
    By combining (\ref{disc entro diff}) and the conditions above,
    \begin{align*}
        F(f^{n_k+1}) < \epsilon < \epsilon^*.
    \end{align*}
    Note that having predecessor in $\epsilon^*$ neighborhood of $\ft$ allows us to have an entropy difference inequality of form (\ref{disc entro diff}). We continue trapping the sequence by induction. Assume, for any fixed $l\in \mathbb{N}$, every predecessor of $f^{n_k+l}$, originated from $f^{n_k}$ is in $\epsilon^*$ neighborhood of $\ft$. Then we will have series of the entropy difference inequalities, (\ref{disc entro diff}), and the telescopic sum of those gives us
    \begin{align*}
         F(f^{n_k+l}) &\leq F(f^{n_k}) + C(1+L)\sum_{n=n_k}^\infty\sum_i\frac{1}{d_if^n_i}(f^{n+1}_i-f^n_i)^2
        +\tau M\sum_{n=n_k}^\infty Q(f^n)\\
        & < \epsilon
    \end{align*}
   Thus we show $F(f_n)<\epsilon \quad \forall n>m.$
    \end{proof}
    \begin{remark}
    The main difficulty of this proof comes from the discreteness of the time step, $\f-f^n$, which vanishes in continuous case, coarsening the logarithmic entropy difference with some redundancy. The key to balancing the error from the logarithmic term is adding extra quadratic terms.
    
    Interestingly, the modified entropy $F$ agrees with the relative entropy, $\mathcal{E}$, suggested by Attouch and Teboulle~\cite{attouch2004regularized},
    \begin{equation*}
        \mathcal{E}(x,y):=\frac{\nu}{2}\Vert x-y\Vert^2 +\mu \sum_j x_j \log(x_j/y_j)+y_j-x_j,
    \end{equation*}
    As we see in \eqref{sq bd}
     \begin{equation*}
       \frac{1}{\tau}(\f-f^n)_i = \frac{ f^n_i}{1/d_i+ \tau\lambda f^n_i}(a - Bf^n)_i.
    \end{equation*}
    The equation~\eqref{disc time dyn} can be regarded as an application of Euler's method to the regularized Lotka-Volterra equation with $\mu=1/d_i$ and $\nu=\tau\lambda$, in which case $\mathcal{E}(\ft, f)=F(f).$
    In the next section, we discuss another way to time-discretize regularized Lotka-Volterra type system with more general energies.
    \end{remark}
 
 \section{Convex splitting for de-singularized systems}
 Technically, implicit iterations provided in DCA \eqref{e: dc al} demand to solve a convex optimization problem given by $H_+$ in each steps. Therefore, in practice we need to find a nice convex splitting representation of $H=H_+-H_-$ so that $H_+$ is easy to invert. 
 In this section, we will apply the implicit discrete iterations to Lotka-Volterra equations. The first task is to to find a convex splitting representation of $H=H_+-H_-$ in a way that $\nabla H_+$ is simple enough to invert. It turns out we can not only makes the calculation simpler but also make the sequence stay positive.
 
 Let $[u]^p_i:=u_i^p$ denote the component-wise power on a vector $u\in\R^N$ and $B=B^+-B^-$ so that
\begin{align*}
    H(u):=-\frac{1}{4}E([u]^2)&=-\frac{1}{4} \sum_{i\in[N]}u_i^2 a_i + \frac{1}{8}\sum_{i,j\in[N]}B_{ij}u^2_i u^2_j\\
   &=-(\frac{1}{4} \sum_{i\in[N]}u_i^2 a_i +\frac{1}{8}\sum_{i,j\in[N]}B^-_{ij}u^2_i u^2_j) +\frac{1}{8}\sum_{i\in[N]} B^+_{ij} u^2_iu^2_j\\
   &=:-H_-(u)+H_+(u)
\end{align*}

\begin{lemma}
Let $B_+$ and $B_-$ be positive definite with nonnegative entries. Then $H_+$ and $H_-$ are convex. Therefore we have convex splitting of $H$, i.e., $H=H_+-H_-.$
\end{lemma}
\begin{proof}
It is enough to check, for $A$ a positive definite matrix with nonnegative entries, that $h(u)=\frac{1}{4}\sum_{i,j\in[N]}A_{ij}u^2_i u^2_j$ is convex.
\begin{align*}
    Dh(u)(v)&=\sum_{i,j\in[N]}A_{ij}v_iu_i u^2_j\\
    D^2h(u)(v,w)&=\sum_{i,j\in[N]}A_{ij}v_iw_i u^2_j+2v_iu_iA_{ij}u_jw_j,\\
    &=\langle v, \nabla^2h(u)w\rangle.
\end{align*}
By denoting $\dg{\cdot}:\R^N\to\R^{N\times N}$ a map that sends a vector to the corresponding diagonal matrix, we can represent the Hessian of $h$, 
\[
\nabla^2h(u)=\dg{A[u]^2}+2\dg{u}A\dg{u}.
\]
Since every component of $u$ and $A$ are positive $\nabla^2h(u)$ is positive definite.
\end{proof}
The previous lemma leaves infinite possibility of DCA algorithms. We suggest the following splitting of $B$, which can further simplify the iteration.
\begin{lemma}\label{l: bpm}
Given a matrix $B\in \R^N\times\R^N$,  let $B^+:=\lambda I+\beta \mathds{1}\mathds{1}^T$ where $\lambda>0$ is the spectral radius of $B$, $\beta$ is a maximum nonnegative off-diagonal element, or zero, and $\mathds{1}_i=1$ for all $i\in[N]$.  Then $B^+$ and $B^-:=B^+-B$ are positive definite matrices with nonnegative entries.
\end{lemma}
\begin{proof}
    It is clear that all entries of $B^+$ and $B^-$ are nonnegative. $B^+$ is positive definite since $\mathds{1}\mathds{1}^T$ is positive semidefinite. Note that for any $v\in\R^N$,
    \begin{equation*}
        v^T B^-v=v^T(\lambda I-B)v+\beta v^T \mathds{1}\mathds{1}^T v \geq 0.\qedhere
    \end{equation*}
\end{proof}
\paragraph{\textbf{Iteration scheme}}
With the suggested convex splitting, we have the following difference equations.

[DC algorithm]
\begin{align}
     \nabla H_{+}(\un) &=\nabla H_-(u^n),\notag\\ 
     \un_i(B^+[\un]^2)_i &=u^n_i, (a+B^-[u^n]^2)_i,\label{e: plv}
\end{align}

[Semi-implicit Euler's]
\begin{align}
    u^{n+1}_i-u^n_i&=-\tau(\nabla H_+(\un)-\nabla H_-(u^n))_i,\notag\\
    &=-\frac{\tau}{2}\un_i(B^+[\un]^2_i)+\frac{\tau}{2}u^n_i(a+B_-[u^n]^2)_i.\label{e: tlv} 
\end{align}
Recall that the semi-implicit Euler's method is the DCA with extra quadratic potential. However, for the semi-implict Euler's method, instead of calculating the spectral radius $\lambda$ in advance, we can change $\tau$ smaller and smaller until the algorithm works empirically.
The suggested convex splitting of $H$ from Lemma~\ref{l: bpm} reduces the complexity. Let $S:=\sum_i u_i^2$, then
\begin{align*}
    2\nabla H_+(u)_i&=u_i\sum_j B^+_{ij}u_j^2\\
    &=(\lambda+\beta)u_i^3 + \beta u_i\sum_{j\not =i}u_j^2\\
    &=(\lambda+\beta) u_i^3 + \beta (S-u_i^2)u_i\\
    &=\lambda u_i^3 +\beta Su_i.
\end{align*}
So we can simplify \eqref{e: plv} and \eqref{e: tlv} to the following $(N+1)$ many $1$-dimensional problems respectively,
\begin{align*}
 \text{DC algorithm}&\left\{\begin{array}{l}
     \lambda u_i^3+\beta Su_i=v_i\\
    \sum_i u_i^2 =S,
        \end{array}\right.\\
   \text{Semi-implicit Euler} &\left\{\begin{array}{l}
             \tau\lambda u_i^3+(2+\tau\beta S)u_i =v_i \\
            \sum_i u_i^2 =S.
    \end{array}\right.
\end{align*}

\begin{remark}[Lotka-Volterra system with cooperative interaction]
If $B$ does not have positive off diagonal element, i.e., $B$ is cooperative, we can choose $\beta=0$. In this case $S$ does not play any role so the calculation will be much simpler. In other words, \eqref{e: plv} and \eqref{e: tlv}, become $N$-many 1-dimensional cubic equations. 
\end{remark}
\paragraph{\textbf{Positivity of the sequence}}
Prior to proving convergence, we would like to point out that solutions of Lotka-Volterra stay in $\R^N_+$. In general, this provides another restriction on choosing time steps. Due to those specific splitting, however, the sequence stays in $\R^N_+$ independent of the choice of the time step.
\mnote{$a$ need to be positive}
\begin{prop}
$\{u^n\}_{n\in\N}$ from \eqref{e: plv} or ~\eqref{e: tlv} with initial $u_0\in\R^N_+$ stays in $\R^N_+.$
\end{prop}
\begin{proof}
Since \eqref{e: tlv} is a special case of \eqref{e: plv}, it suffices show the proof for~\eqref{e: plv}. Let's take a look at the $i$-th component in the explicit formula of~\eqref{e: plv},
\begin{align*}
   \un_i\sum_{j\not=i}B^+_{ij}[\un_j]^2+ B^+_{ii}[\un_i]^3&=u^n_i(a+B^-[u^n]^2)_i.
\end{align*}
Note that the strong convexity of $B^+$ assures $B^+_{ii}>0$.
Since the iteration is well defined, we can consider the left hand side as a function of $\un_i$ which is strictly increasing and vanishes at $0$. On the other hand, the right hand side is nonnegative because $u^n \ge 0$. Therefore $\un\ge 0.$
\end{proof}
Therefore by Theorem~\ref{t: nlmg conv} we can conclude the following.
\begin{theorem}
    Any bounded sequence $\{u^n\}_{n\in\N_0}$ of discrete-time Lotka-Volterra equation stays in $\R^N_+$ and if bounded converges.
\end{theorem}
\section{Time-discretization of regularized Lotka-Volterra systems}
\paragraph{\textbf{Preconditioning and Riemannian metric}}
Let $g$ be a strongly convex function on $\R^N$ and consider the Riemannian metric induced by $g$, $\langle v,w\rangle_{g(u)}:=v^T\nabla^2 g(u) w$.
\begin{align*}
    DH(u)(v)&=\nabla H(u)^T v\\
    &=(\nabla^2g(u)^{-1}\nabla H(u))^T\nabla^2 g(u) v\\
    &=:\langle \text{grad} H(u), v\rangle_{g(u)}.
\end{align*}
So the gradient flow of $H$ with respect to the $g$-metric is as follows,
\begin{align*}
    u'(t)\;\;&=-\text{grad}_g H(u)\\
    &=-\nabla^2 g(u)^{-1}\nabla H(u).
\end{align*}
Or equivalently,
\begin{equation*}
	 \frac{d}{dt}\nabla g(u(t)) =-\nabla H(u),
\end{equation*}
in which $u$ is following the gradient of $H$ in the scope of $\nabla g$. The analogous time-discretization method is called the mirror descent or the preconditioning method.
\mnote{lens??}
\begin{equation*}
    \nabla g(\un)-\nabla g(u^n)=-\tau \nabla H(u^n).
\end{equation*}
Note that we can use this idea to Euler's semi-implicit method,
\begin{equation}
    \nabla g(\un)-\nabla g(u^n) = -\tau(\nabla\Tilde{H}_+(\un)-\nabla \Tilde{H}_-(u^n)),
\end{equation}
which is merely adding additional convexity, $g$, to the convex splitting $(\Tilde{H}_+,\Tilde{H}_-)$ of $H$,
\begin{equation*}
    H(u)=\big(\Tilde{H}_+(u)+\frac{1}{\tau}g(u)\big) -\big(\Tilde{H}_-(u)+\frac{1}{\tau}g(u)\big).
\end{equation*}

As we discussed, the Regularized Lotka-Volterra equation can be seen as a convex preconditioning. So we suggest the following iteration schemes.

\begin{equation}\label{e: reg lv eu}
    \nabla g(\un)-\nabla g(u^n)=-\gamma \nabla H(u^n),
\end{equation}

\begin{equation}\label{e: reg lv semi}
     \nabla g(\un)-\nabla g(u^n)=-\tau (\nabla H_+(\un)-\nabla H_-(u^n)),
\end{equation}
where 
$$g(u)= \frac{\mu}{2}\mu\Vert u\Vert^2+ \frac{\nu}{12}\sum_i u_i^4.$$

Note that both schemes fall into the category of DC algorithms.
The first scheme is preconditioned Euler's method. So we need to choose $\tau>0$ small enough to make $g-\tau H$ convex. On the other hand, the second scheme does not have a restriction on $\tau$.

To conclude, by Theorem~\ref{t: nlmg conv} we have the following.
\begin{theorem}
    Assume that $\tau>0$ and that $\gamma>0$ is sufficiently small so that $\frac{1}{\gamma}g-H$ is convex. Then the sequence $\{u^n\}_{n\in\N_0}$ generated from the time-discretized regularized Lotka-Volterra systems~\eqref{e: reg lv eu}, or \eqref{e: reg lv semi} converges.
\end{theorem}
\part{Concentration-dispersion dynamics}
\chapter{Gradient structure of concentration-dispersion models}
Despite its gaining in importance in nonlinear analysis, exact solutions of solitary waves are challenging to compute. In 1976, Petviashvili proposed a powerful numerical method for computing solitary
wave solutions without analysis or proof. Due to the efficacy of the algorithm, Petviashili’s method was applied to numerous nonlinear problems in modern mathematical physics (\cite{petviashvili1976equation}, \cite{pelinovsky2004convergence}, \cite{lakoba2007generalized}). 

We study a new type of nonlinear dispersive integro-differential equations, inspired by the one dimensional nonlinear wave equation with power nonlinearity. Using the gradient structure and \L ojasiewicz convergence theorem, we seek to prove global convergence of the solutions to the equilibrium which can be considered as a solitary wave.

\paragraph{\textbf{Motivation}}
Consider a nonlinear scalar wave equation with power nonlinearity in one dimension:
     \begin{equation*}
        u_t-(\mathcal{L}u)_x+pu^{p-1}u_x=0
     \end{equation*}
where $u(t,x):\R_+\times\R\to\R$, $p>1$, and $\mathcal{L}$ is a linear self-adjoint nonnegative pseudo-differential operator of order $m$, in $x$ with constant coefficients in $L^2(\R)$.
    The ansatz $u(t,x)=\Phi(x-ct)$
    \[
    c\Phi + \mathcal{L}\Phi =\Phi^p
    \]
    and Fourier transform,
    \[
    (c+\widehat{\mathcal{L}}(k))\widehat{\Phi}(k) = \widehat{\Phi^p}(k),
    \]
    inspires the following fixed point formula for a solitary wave profile:
    \[
    \Phi=\int_\R K(x-y)\Phi^p(y)dy,
    \]
where the Fourier transform of $K$ satisfies the following,
\begin{equation*}
    \widehat{K}(k)=\frac{1}{c+\widehat{\mathcal{L}}(k)}
\end{equation*}
\paragraph{\textbf{Petviashvili iteration}}
In order to calculate nontrivial solitary wave solutions in a space of two dimensions, Petviashvili~\cite{petviashvili1976equation} devised the fixed point algorithm with a scaling factor $M_n$ with power $\gamma>0$, when it is applied to a space of one dimension, as follows
\begin{equation}\label{e: petit}
\left\{
    \begin{aligned}
         \widehat{u}_{n+1}(k)&=M_n^\gamma\widehat{K}(k){\widehat{u^p_n}}(k)\\
    M_n(\widehat{u}_n)&=\frac{\int_\R [c+\widehat{\mathcal{L}}(k)][\widehat{u}_n(k)]^2dk}{\int_\R \widehat{u_n}(k)\widehat{u_n^p}(k)dk}.
    \end{aligned}
\right.
\end{equation}
It is known that without the scaling factor $M_n$, or when $\gamma=0$, the iteration~\eqref{e: petit} usually diverges.
When $p=2$ Petviashvili empirically found that the iteration~\eqref{e: petit} method converges when $1<\gamma<3$, and at $\gamma=2$ the fastest rate of convergence occurs. Indeed, Pelinovsky and Stepanyants in 2004~\cite{pelinovsky2004convergence} proved that, under spectral stability assumptions on the linearized operator at an unknown equilibrium (Assumption 2.1~\cite{pelinovsky2004convergence}), given $p>1$, the iteration~\eqref{e: petit} converges when $1<\gamma<\frac{p+1}{p-1}$, and at $\gamma=\frac{p}{p-1}$ the solution converges at the fastest rate.
\\
\paragraph{\textbf{Concentration-dispersion equation}}
For $L\in(0,\infty]$, let $\mathds{T}_L:=\R/L= [-\frac{L}{2},\frac{L}{2}]$ denote the torus of length $L$.
 We consider a time continuous evolution of $u(t,x):\R_+\times\mathds{T}_L\to\R$, which is analogous to Petviashvili iteration. We define \textit{concentration}-\textit{dispersion} \textit{equation} as follows,
    \begin{equation}\label{e: solwav}
    \left\{
    \begin{aligned}
        \;\;\frac{d}{dt}u(t,x)&=K\ast u^p  - c(t) u,\\
        c(t)&=\int u^p K\ast u^p dx.
    \end{aligned}
     \right.
    \end{equation}
In what follows, we assume that $K(x)>0$ for all $x\in \mathds{T}_L$ and $K(x)\in W^{1,1}\cap W^{1,p+1}(\mathds{T}_L)$.

Power nonlinearity creates concentration effect and convolution gives dispersion effect. The equation is designed for $u$ to chase after a scalar multiple of $K\ast u^p$.
Just for the comparison between~\eqref{e: petit} and \eqref{e: solwav}, the scaling factors can be related as follows,
\begin{align*}
    M_n(\widehat{u}_n)&=\frac{\int_\R [c+\widehat{\mathcal{L}}(k)][\widehat{u}_n(k)]^2dk}{\int_\R \widehat{u_n}(k)\widehat{u_n^p}(k)dk}\\
    &=M_{n-1}^{2\gamma}\frac{\int_\R \widehat{u_{n-1}^p}(k)\widehat{K}(k)\widehat{u_{n-1}^p}(k)}{\int_\R \widehat{u_n}(k)\widehat{u_n^p}(k)dk}\\
    &=M_{n-1}^{2\gamma}\frac{\int_\R u^p_{n-1} K*u^p_{n-1}dx}{\int_\R u\p_n dx}
\end{align*}

Prior to discussing the gradient structure of the concentration-dispersion equation~\eqref{e: solwav}, we study the asymptotic properties of a few important functionals.
\paragraph{\textbf{Important functionals}}
Here we aim to study a dynamical way to calculate the solution, asymptotic behavior, and stability of non-local and nonlinear differential equations.
We introduce two functionals, $E$ and $F$, which provides different gradient structure to describe concentration-dispersion dynamics~\eqref{e: solwav}. Let
\begin{equation*}
    \begin{aligned}
        E(u)&:= \frac{1}{2p}\int u^p K\ast u^p dx,\\
        F(u)&:= e^{-\frac{2p}{p+1}\int u\p dx}E(u).
    \end{aligned}
\end{equation*}
Note that $c(t)=2pE(u)$.
\paragraph{\textbf{Gradient structure on the sphere in $L\p$}}
    Let 
    $$\mathcal{M}:=\{u>0 : \int u(x)\p  dx=1\}\subset H^1(\mathds{T}_L) =:\mathcal{V}.$$
    $\mathcal{M}$ is an invariant manifold of \eqref{e: solwav}, and, moreover, $\mathcal{M}$ attracts every solution: for any $L\in(0,\infty]$,
    \[
    \frac{1}{p+1}(\int u\p dx)'= \int u^p K\ast u^p dx \big(1-\int u\p dx\big).
    \]
    We will discuss the detailed proof later in Proposition~\ref{p: conv func}.
    In this case, \eqref{e: solwav} is a gradient flow of $E$ with respect to the Riemannian structure given as follows,
\[
    \langle v, w \rangle_u := \int vwu\pp dx, \quad \forall u\in \mathcal{M}, \quad v,w\in T_u\mathcal{M},
\]
where the tangent manifold at $u\in\mathcal{M}$ is given as follows,
$$T_u\mathcal{M}:=\{v\in \mathcal{V} : \int vu^p dx=0\}.$$ 

It is interesting to see that with respect to the given metric $T_u\mathcal{M}$ is ``orthogonal" to $u$, i.e., $\langle u,v\rangle_u=0$, as if $\mathcal{M}$ is spherical. It is easy to check that \eqref{e: solwav} on $\mathcal{M}$ is equivalent to the Gram-Schmidt projection from $X$ onto the manifold, i.e.,
\begin{equation}
    u'=K*u^p-\frac{\langle K*u^p,u\rangle_u}{\langle u,u\rangle_u}u.
\end{equation}
Indeed, \eqref{e: solwav} is a projected gradient flow onto $\mathcal{M}$. If we perturb $u$ with $v\in T_u\mathcal{M}$
\begin{align}
    DE(u)(v) &= \int u\pp (K\ast u^p) vdx\label{e: nabE} \\
    &=: \int \nabla E(u)vdx\notag\\
    &=\langle \nabla E(u),v\rangle_{L^2}\notag\\
    &=\int (K\ast u^p -c(t)u)vu\pp dx\notag\\
    &=\langle \text{grad}E(u), v\rangle_u\notag
\end{align}
so that
\begin{align*}
    u'&=\text{grad}E(u),\\
    \frac{d}{dt}E(u(t))&=\Vert \text{grad}E(u)\Vert^2_u.
\end{align*}
\\
\paragraph{\textbf{$L^2$-gradient structure}}
For a fixed $L\in(0,\infty]$, let $\mathcal{H}:=L^2(\mathds{T}_L)$ and  $\mathcal{V}:=H^1(\mathds{T}_L)$, so that
\begin{equation*}
    \mathcal{V}\subset\mathcal{H}\subset\mathcal{V}'.
\end{equation*}
As long as $E$ is well defined, we can extend the gradient structure on $\mathcal{M}$ to $L^2$ by the proper scaling of the energy $E$.
Let the scaling factor be
\begin{equation*}
    \alpha(t):=e^{\frac{2p}{p+1}\int u\p dx},
\end{equation*}
so that
\begin{equation*}
    F(u):= E(u)/\alpha(t).
\end{equation*}
Then we can check the $L^2$ gradient-like structure given by $F$,
\begin{align*}
     \alpha(t)DF(u)(v)&= DE(u)(v)-E(u)\frac{2p}{p+1}D(\int u\p dx)(v)\\
    &=\int u\pp K*u^p vdx-2pE(u)\int u^pv dx\\
    &=\int u\pp(K*u^p - 2pE(u)u)vdx.\\
    &=\int u\pp(K*u^p - c(t)u)vdx.
\end{align*}
\begin{align*}
  DF(u)(v) &= e^{-\frac{2p}{p+1}\int u\p dx}\int u\pp(K*u^p - c(t)u)vdx,\\
  &=:\langle \nabla F(u), v\rangle_{L^2}.
\end{align*}
Therefore, \eqref{e: solwav} takes the form
\begin{equation*}
	u'= e^{\frac{2p}{p+1}\int u\p dx}u^{1-p}\nabla F(u).
\end{equation*}

We finish the chapter by proving some convergence properties of the functionals. It is easy to see that $u(t,\cdot)$ of~\eqref{e: solwav} stays positive when $u(0,\cdot)>0$. The concrete proof along with well-posedness can be found in Theorem~\ref{t: wpbs}.
\begin{prop}\label{p: conv func}
    Let $u(x,t)\in L\p(\mathds{T}_L)$ be a solution of \eqref{e: solwav} such that $u(x,0)>0$ for all $x\in\mathds{T}_L$, then the following functionals have convergent properties as $t\to\infty:$
    \begin{enumerate}
        \item $F(u)= E(u)/\alpha(t)$ is nondecreasing,
        \item $\Vert u\Vert_{L\p}$ converges monotonically to 1,
        \item $E(u)= \frac{1}{2p}\int u^p K\ast u^p dx$ converges to a positive constant.
    \end{enumerate}
    Additionally, if $\Vert u(0)\Vert_{L\p}\leq 1$ then $E(u)$ is nondecreasing as well.
\end{prop}
\begin{proof}
    $F$ is nondecreasing because the time derivative of $F$ is nonnegative, i.e.,
    \begin{align*}
        F(u(t))'e^{\frac{2p}{p+1}\int u\p dx}&=E(u)'-E(u)\frac{2p}{p+1}(\int u\p dx)'\\
        &=\int u\pp K*u^p u'dx-2pE(u)\int u^pu' dx\\
        &=\int u\pp(K*u^p - 2pE(u)u)u'dx.\\
        &=\int u\pp(K*u^p - c(t)u)^2dx \ge 0.
    \end{align*}
    Since
    \begin{equation*}
        F(u)=E(u)/\alpha(t)\le E(u)=2pc(t),
    \end{equation*}
    we have $\inf_{t\ge0}c(t)>0$. This yields the monotonic convergence of $L\p$ norm together with taking time-differentiation,
    \begin{equation}\label{e: ulpcon}
        \frac{1}{p+1}(\int u\p dx)'= c(t) \big(1-\int u\p dx\big).
    \end{equation}
    Therefore, 
    \begin{equation*}
        \lim_{t\to\infty} \Vert u(t)\Vert_{L\p}\to 1.
    \end{equation*}
     Additionally, if $\Vert u(0)\Vert_{L\p}\leq 1$, $E$ is nondecreasing as well because $\Vert u(0)\Vert_{L\p}$ is nondecreasing.
    
    Finally, we show that $E(u)$ is bounded above. From \eqref{e: ulpcon} we have
    $$\Vert u\Vert_{L\p}\leq 1+\Vert u_0\Vert_{L\p}.$$
    So we use H\"older's inequality and Young's convolution inequality to bound $E$,
     \begin{align}
        2pE(u) &= \int u^p K\ast u^p dx\notag\\
            &\leq \Vert u\Vert^p_{L\p} \Vert K*u^p\Vert_{L\p}\notag\\
            &\leq \Vert K\Vert_{L^\frac{p+1}{2}} \Vert u\Vert^{2p}_{L\p}\notag\\
            &\leq(1+\Vert u_0\Vert_{L\p})^{2p}\Vert K\Vert_{L^\frac{p+1}{2}}<\infty\label{e: E uBound}.
    \end{align}
This implies $F$ is bounded. Therefore $F>0$ converges monotonically to a positive constant, and thus $E=\alpha(t) F>0$ also converges to a positive constant.
\end{proof}
 
\chapter{Well-posedness, compactness and nontrivial wave profiles}
In this chapter, we explore properties of the concentration-dispersion equation~\eqref{e: solwav} on the torus $\mathds{T}_L=\R/L=[-\frac{L}{2},\frac{L}{2}]$ of length $L\in (0,\infty]$. 

Recall that the equation~\eqref{e: solwav} is motivated by solitary waves and the solution $u$ evolves in a way to balance the concentration effect and the dispersion effect. We use Picard iteration to show that the solution is well-posed in $\mathcal{V}:=H^1(\mathds{T}_L)$. Simultaneously, we show that if the initial data $u_0$ is ``bell-shaped", then the solution stays bell-shaped along the evolution.

We also show that the solution is bounded above, and below away from $0$. Directly, by Arzela-Ascoli’s theorem, the solution is precompact in the space of continuous functions.
Actually, we can say much more. Recall that dispersive effect on the semi-linear heat equations provides the smoothing effect, which makes the solution smoother than it initially was, so we have compactness (Section~\ref{s: slheat} of Chapter~\ref{ch: lv cont}). 
Even if the convolution operator cannot make the solution smoother than the initial data, but we can show that the solution is a combination of the scalar multiple of the initial data and the smoother part from the dispersion effect by $K\in W^{1,1}$. So the solution is precompact in $H^1(\mathds{T}_L)$.

Note that we aim to approximate nontrivial solitary wave solutions. We show the constant equilibrium is unstable and, moreover, use the gradient structure to prove that nontrivial equilibria exist.
The convergence result, however, cannot be proven using the \L ojasiewicz framework. We will discuss the issue at the end of this chapter.

\section{Well-posedness of solutions and consistency of bell-shape}
We use the Banach fixed point theorem to show well-posedness of the concentration-dispersion equation. Later in the section, we continue the argument to prove that the solution with bell-shaped initial value stays bell-shaped.
At a first glance at \eqref{e: solwav}, it is not immediate to believe this to be true, since $u'$ is a subtraction of two bell-shaped function which, in general, might not be bell-shaped. However, using simple ODE trick we show that the bell-shape of $u$ persists

\begin{theorem}[Well-posedness]\label{t: wpbs}\mnote{W: $L\leq\infty$}
    For any initial $u_0\in \mathcal{V}$, there exists unique solution $u(t,x)\in C^1(\R_+;\mathcal{V})$ of~\eqref{e: solwav}. Additionally, if $u_0(\cdot)>0$ then $u(t,\cdot)>0$ for all $t\ge0$.
\end{theorem}
\begin{proof}
To use Picard iteration, we modify \eqref{e: solwav} as follows,
        \begin{align*}
            u'+c(t)u &=K\ast u^p \\
            u(t)&=e^{-\int_0^t c(s)ds}u(0) +\int_0^t e^{-\int_s^t c(r)dr}K\ast u^p  ds.
        \end{align*}
    For fixed $u_0\in\mathcal{V}$ we define $I : C\big([0,T], \overline{B_+(R)}\big) \to C\big([0,T], \overline{B_+(R)}\big)$, as follows
    \[
    I(u)(x,t) :=e^{-\int_0^t c(s)ds}u_0(x) +\int_0^t e^{-\int_s^t c(r)dr}K\ast u^p ds,
    \]
      where $B_+(R)=\{u\in\mathcal{V}: \Vert u\Vert_\mathcal{V}<R\text{ and } u(x)>0, \forall x\in \mathds{T}_L\}.$
    Claim that if $\Vert u_0(x)\Vert_{H^1} \leq \frac{R}{2}$ for some $R$ then there exists $T>0$ such that,
    \[
    I : C\big([0,T], \overline{B_+(R)}\big) \to C\big([0,T], \overline{B_+(R)}\big)
    \]
    with the induced uniform norm 
    $$\Vert u\Vert_\infty := \sup_{t\in[0,T]}\Vert u(x,t)\Vert_{H^1}.$$
    Let $u(x,t) \in C([0,T], \overline{B_+(R)})$,
    \begin{align*}
        \Vert I(u)\Vert_\infty &\leq \Vert e^{-\int_o^T c(t)dt} u_0\Vert_\infty + \int_0^T \Vert K\ast u^p \Vert_\infty dt\\
        &\leq \Vert u_0\Vert_{H^1} + T\Vert K\ast u^p \Vert_\infty.
    \end{align*}
     If $1<p<2$,  we use Young's convolution inequality to analyze $K\ast u^p$
    \begin{align*}
        \Vert K\ast u^p\Vert_{H^1} &\leq \Vert K\ast u^p\Vert_{L^2} +\Vert K_x\ast u^p\Vert_{L^2}\\
        &\leq (\Vert K\Vert_{L^q}+\Vert K_x\Vert_{L^q}) \Vert u\Vert_{L^2}^p\\
        &= \Vert K\Vert_{W^{q,1}} R^p,
    \end{align*}
    where $q$ satisfies, $\frac{p}{2}+\frac{1}{q}=1+\frac{1}{2}$, i.e. $q=2/(3-p)$.
    
    If $p\geq 2$, $\Vert K\ast u^p \Vert_{H^1}=\Vert K*u^2u^{p-2}\Vert_{L^2} +\Vert K_x*u^2u^{p-2}\Vert_{L^2}$ so we can use Morrey's embedding theorem and Young's convolution inequality to get the bound.
    \begin{align*}
        \Vert K\ast u^p \Vert_{H^1}&=\Vert K*(u^2u^{p-2})\Vert_{L^2} +p\Vert K\ast (u_xuu^{p-2})\Vert_{L^2}\\
        &\leq R^{p-2}(\Vert K\Vert_{L^2}\Vert u\Vert _2^2 +p\Vert K\Vert_{L^2}\Vert u\Vert_{L^2} \Vert u_x\Vert_{L^2})\\
        &\leq pC\Vert K\Vert_{L^2} R^p
    \end{align*}
    
    Therefore,
    \begin{align*}
        \Vert I(u)\Vert_\infty &\leq \frac{R}{2} + C(K,p)R^p T\\
        &\leq R
    \end{align*}
     by letting $T\leq \frac{1}{2CR\pp}$.
    
    Also we can check that $I$ is Lipschitz on $C\big([0,T], \overline{B(0,R)}\big)$, since $x\mapsto e^{-x}$ and $y\mapsto y^p$ are Lipschitz when $x\ge 0$ and when $y$ is bounded, respectively. We use Picard iteration, or Banach fixed point theorem to conclude the well-posedness for short-time. To extend this to long-time, we build time independent $H^1$ bound (Proposition~\ref{p: uhn}) later in this chapter, and use the Picard iteration again.
\end{proof}

\paragraph{\textbf{Bell-shape consistency}}
In this section, in addition to prove well-posedness of the equation~\eqref{e: solwav}, we show that the solution with bell-shaped initial value stays bell-shaped.
We consider the equation~\eqref{e: solwav} on compact domain $\mathds{T}_L$ with the same kernel $K$ for $L\in(0,\infty]$. 
We start with making an observation on the shape of the periodization of $K$ and $K*u$.
    \begin{definition}
         A real-valued function $u$ on $\mathds{T}_L$ is \textit{bell-shaped} if $u$ is even and monotone decreasing on $(0,L/2)$.
    \end{definition} 
    
    \begin{definition}
We say $K_L$ is the $L$-periodization of a function $K$ on $\R$ when
       \begin{align*}
           K_L(x) &= \sum_{k=-\infty}^\infty K(x-kL).
        \end{align*}
    Also the convolution of $K$ with a $L$-periodic function $u:\R\to\R$ is as follows,   
    \begin{align*}
           (K\ast u)(x) &:= \int_{-\infty}^\infty u(y)K(x-y)dy\\
           &=\int_{x_0}^{x_0+L}u(y)K_L(x-y)dy,
       \end{align*}
for any $x_0\in\R.$
\end{definition}

    \begin{lemma}
        Let $K$ be a bell-shaped function. If $K$ is convex in $[0,\infty)$ then $K_L$ is bell-shaped on $\mathds{T}_L$.
    \end{lemma}
    \begin{proof}
            Enough to show that $K_L$ is monotone non-increasing in $[0, \frac{L}{2}]$.
            \begin{align*}
                K_L(x) =&\sum_{n=-\infty}^\infty K(x-nL)\\
                &=\sum_{n=0}^\infty K(x+nL) + \sum_{n=1}^\infty K(x-nL)\\
                &=\sum_{n=0}^\infty K(x+nL) + K((n+1)L-x)
            \end{align*}
            
            Let $h>0$, $x\in[0,\frac{L}{2})$.
            
            \begin{align*}
                K_L(x+h) -K_L(x) =\sum_{n=0}^\infty&[K(x+h+nL+K((n+1)L-x-h]\\
                &\quad-[K(x+hL)+K((n+1)L-x)]
            \end{align*}
    Claim that $\forall n\in \N \cup\{0\}$, 
    $$K(x+h+nL) +K((n+1)L-x-h)-K(x+hL)+K((n+1)L-x)\leq 0.$$
    In the interval $[x+nL, (n+1)L-x]$,
    let $s=\frac{h}{L-2x}$ and use convexity.
    \begin{align*}
        K(x+h+nL)+K((n+1)L-x-h) &\leq (1-s)K(x+nL) +sK((n+1)L-x)\\ &+sK(x+nL)+(1-s)K((n+1)L-x)\\
        &\leq K(x+nL)+K((n+1)L-x)
    \end{align*}
    as claimed.
    \end{proof}
    \begin{remark}
    {Note as long as $K$ does not resemble ``stairs", i.e. $K_x$ does not vanish except at $0$,} we can generalize the previous result as follows. If $K$ is a bell-shaped function which is strictly concave inside a compact interval and convex outside, then $K_L$ is bell-shaped for big enough $L$. For example suppose that $K\approx o(x^{-m})$ for $m>1$ then
    $$K_L(x)=K(x)+\sum_{n=1}^\infty \frac{1}{(nL+x)^m}+\frac{1}{(nL-x)^m}\approx K(x)+\frac{1}{L^m}\sum_{n=1}^\infty \frac{1}{(n+x/L)^m}+\frac{1}{(n-x/L)^m},$$
    so that for big enough $L$ and near $0$, $K_L$ resembles the shape of $K$.
    \end{remark}
    
    \begin{lemma}~\label{l: bell shape}
        For such $K$, where $K_L$ is bell-shaped, if $u$ is a $L$-periodic bell-shaped function, then $K\ast u$ is $L$-periodic bell-shaped.
    \end{lemma}
    
    \begin{proof}
            \begin{align*}
                K\ast u(x) =\int_\R K(y)u(y-x)dx
                =\int_{-\frac{L}{2}}^{\frac{L}{2}}K_L(y)u(y-x)dy.
            \end{align*}
    On the inverval of half period such that $u(x)$ is monotone, $u_x(x)$ exists for almost every $x\in[0,\frac{L}{2}]$. Note that $u_x$ is odd and $u_x(x)\leq 0$ when $x\geq 0$. Claim that $\forall x\in[0,\frac{L}{2}] \quad (K\ast u)_x(x) \leq 0$. We start with partitioning $\mathds{T}_L$ with the intervals of length $x$,$x$,$(\frac{L}{2}-x)$ and $(\frac{L}{2}-x)$.
    \begin{align*}
        (K\ast u)_x(x) &= \int_{-\frac{L}{2}}^{\frac{L}{2}} K_L(y)\frac{d}{dx}u(y-x)dy\\
        &= \int_{-\frac{L}{2}}^{\frac{L}{2}} -K_L(y)u_x(y-x)dy\\
        &=(\int_{-\frac{L}{2}}^{-\frac{L}{2}+x}+\int_{-\frac{L}{2}+x}^{\frac{L}{2}+2x}) +(\int_{-\frac{L}{2}+2x}^{x}+\int_{x}^{\frac{L}{2}}) (-K_L(y)u_x(y-x))dy\\
        &= I +II
    \end{align*}
    Now on each part, we reflect to another side.
    \begin{align*}
        I &= \int_{-\frac{L}{2}}^{-\frac{L}{2}+x}+\int_{-\frac{L}{2}+x}^{\frac{L}{2}+2x} (-K_L(y)u_x(y-x))dy\\
        &= \int_{-\frac{L}{2}}^{-\frac{L}{2}+x}(-K_L(y)u_x(y-x))-K_L(-L+2x-y)u_x(-L+x-y)dy\\
        &= \int_{-\frac{L}{2}}^{-\frac{L}{2}+x}(-K_L(y)u_x(y-x))-K_L(y')u_x(x-y)dy\\
        &= \int_{-\frac{L}{2}}^{-\frac{L}{2}+x}-(K_L(y)-(K_L(y'))(u_x(y-x))dy <0
    \end{align*}
    Same for II.
    where $\frac{y'+y}{2}=-\frac{L}{2} +x$.
    \end{proof}
    \begin{remark}
        With similar observations, we can prove the same for $L=\infty$, i.e. if $u$ and $K$ are bell-shaped on $\R$ then $K*u$ is bell-shaped.
    \end{remark}
\begin{theorem}[Bell-shape consistency]\label{t: ubs}
   If $u_0\in \mathcal{V}$ is bell-shaped, $u(t)$ is bell-shaped for all $t\ge0.$
\end{theorem}
\begin{proof}
    We continue from the proof of Theorem~\ref{t: wpbs}.
    Let 
    $$\mathcal{K}:=\{u\in H^1(\mathds{T}_L) : \text{u is bell-shaped}\}$$ 
    and suppose $u_0\in \mathcal{K}.$ Due to Lemma~\ref{l: bell shape} if $u\in\mathcal{K}$ then $I(u)\in\mathcal{K}$.
    Again, by Picard iteration, $u(t)$ stays bell-shaped for $0\le t<T$ if $u_0\in\mathcal{K}$. Furthermore by the similar argument from the proof of Theorem~\ref{t: wpbs} we can let $T\to\infty.$\qedhere
\end{proof}
\section{Compactness of the solutions}
\paragraph{\textbf{Asymptotic bound on $u$ and $u_x$}}
We investigate asymptotic pointwise bounds on the solution. It turns out the solution stays bounded not only above but also below, as long as $K$ stays away from 0. We can expect that the lower bound will diminish as $L\to\infty$.
\begin{prop}\label{p: u upperB}
    Assume that $K(x)\in W^{1,1}\cap W^{1,p+1}(\mathds{T}_L)$. Let $u\in \mathcal{V}$ be a solution of \eqref{e: solwav}. Then there exists $M(K,u_0)>0$ such that $\sup_{t\ge0}\Vert u(x,t)\Vert_\infty <M$. Additionally, there exists $m>0$, for any $x$ such that $\frac{\partial}{\partial x}u_{0}$ exists, $\lim\sup_{t\to\infty}\Vert u_x(x,t)\Vert_\infty <m$.
\end{prop}
\mnote{W: $L\leq\infty$, need to use $L^p$ bound and concentration compactness for $L=\infty$}
\begin{proof}
        We want to find a pointwise threshold where $u$ and $u_x$ stop growing. Start from \eqref{e: solwav},
        \begin{align}
            u' &= K\ast u^p  - c(t)u \\
            &= \int K(x-y) u(y)^p dy - c(t) u\\
            &\leq (\int K(x-y)\p dy)^{1/p
            +1}(\int u\p dy)^{p/(p+1)} - \underline{c} u,
        \end{align}
        where $\underline{c}=\inf_{t\ge 0}c(t)$. Note that $\underline{c}>0$ because $u>0$ and $F(u)=\frac{1}{2p} e^{-\frac{2p}{p+1}\int u\p dx}c(t)>0$ is increasing according to the Proposition~\ref{p: conv func}.
        This means
        $$u'(x,t) <0 \quad \text{if } u(x,t) > \frac{1}{\underline{c}}\Vert K\Vert_{L\p} \Vert u\Vert^p_{L\p}.$$
         By Morrey's embedding theorem, $u\in L^\infty$. 
         Due to monotonic convergence of $\Vert u\Vert_{L\p}\to 1$ as $t\to \infty$, we can conclude that
        $$\max_{t\ge 0}\Vert u(x,t)\Vert_\infty \leq M=\Vert u_0\Vert_\infty+\frac{1}{\underline{c}}(1+ \Vert u_0\Vert^p_{L\p})\Vert K\Vert_{L\p}.$$
        
        Similarly, for any point where $u_x$ exists and is positive,
        \begin{align*}
            u_x' &= (K\ast u^p )_x - c(t)u _x\\
            &= \int K_x(x-y) u(y)^p dy - c(t) u_x\\
            &\leq \Vert K_x\Vert_{L\p} \Vert u\Vert^p_{L\p} -\underline{c}u_x,
        \end{align*}
       and a similar argument holds for any point where $u_x$ exists and is negative. 
\end{proof}
\begin{remark}\label{r: M bound}
    In the case when $\Vert u(0)\Vert_{L\p}\leq 1$, $c(t)$ is increasing. So we can obtain the explicit upper bound,
    \begin{equation*}
        \Vert u(t)\Vert_\infty \leq M=\max\{\Vert u(x,0)\Vert_\infty, \Vert K\Vert_{L\p}/c(0)\}.
    \end{equation*}
\end{remark}
In the compact spatial domain, we can use the upper bound to calculate the lower bound. This will be used to control the Riemannian metric. Naturally, the lower bound vanishes as $L\to\infty.$
\begin{corollary}\label{c: low u}
    For fixed $L<\infty$, let $u\in\mathcal{V}$ be a solution of \eqref{e: solwav}. Assume that $K_L(x)\in W^{1,1}\cap W^{1,p+1}(\mathds{T}_L)$ and there is $k>0$ such that $K_L(x)\ge k,$ $\forall x\in\mathds{T}_L$. Then there exist $m(K,L,u_0)>0$ such that $\min_{x}u(x,t) \ge m$.
\end{corollary}
\mnote{$L<\infty$, we need this bound to get Cauchy in $L^2$ plus should I distinguish $K_L$}
\begin{proof}
   From Proposition~\ref{p: u upperB}, there exists $M(u_0,K)>0$ such that $u(x,t)<M$ for all $t\geq 0$. Now for any $\epsilon>0$ we can choose large enough $T>0$ so that 
   $$1-\epsilon\le\int u\p dx \le M\int u^p dx.$$ 
   Note that  $\overline{c}:=\sup_{t\ge 0}c(t)$ is well defined since it converges. Then,
    \begin{align*}
        u'&= K\ast u^p -c(t) u\\
        &\ge k\int u^pdx -\overline{c} u\\
        &\ge \frac{(1-\epsilon)k}{M}-\overline{c} u >0,
    \end{align*}
    when $u < (1-\epsilon)k/M\overline{c}$. So the result follows.
\end{proof}
\begin{remark}If $\Vert u(0)\Vert_{L\p}\leq 1$, we can obtain a more explicit lower bound using Remark~\ref{r: M bound} and \eqref{e: E uBound}.
For any $\epsilon$ there exists $T>0$ such that for all $t>T$, we have
\begin{equation*}
    \min_xu(t,x)\ge m= \frac{(1-\epsilon)k}{M\Vert K \Vert_{L^\frac{p+1}{2}}}.
\end{equation*}
\end{remark}

\paragraph{\textbf{$L^2$-compactness of the solutions}}
Since the Rellich–Kondrachov theorem is only valid on compact domain, we need more careful analysis to achieve compactness on the real line. We will take advantage of the bell-shaped solutions to show that the solution is tight in $L^p$-spaces for any $p>1$.
\begin{prop}[uniform $H^1$-norm] \label{p: uhn}
    Let $u\in H^1(\mathds{T}_L)$ be a (maximal) solution of \eqref{e: solwav}. Then there exists time independent $C>0$ such that $\Vert u(t,\cdot)\Vert_{H^1}<C.$
\end{prop}
\begin{proof}
We used the fixed point formula~\eqref{e: u fp},
\begin{align}
     u(x,t)&=e^{-\int_0^tc(\tau)d\tau}u_0+\int_0^t e^{-\int_s^t c(\tau)d\tau}K*u^pds.
\end{align}
Since we have uniform $L\p$ and $L^\infty$ norm (Proposition~\ref{p: conv func}, Proposition~\ref{p: u upperB}), we have uniform $L^2$ bound on the integrand as follows,
\begin{align*}
     \Vert K*u^p\Vert_{H^1} \le \Vert K\Vert_{W^{1,1}} \Vert u^p\Vert_{L^2}
    \leq\Vert K\Vert_{W^{1,1}}\Vert u\Vert_{L\p}^{\frac{p+1}{2}} \Vert u\Vert_{L^\infty}^\frac{p-1}{2} <\infty.
\end{align*}
Therefore,
\begin{align*}
    \Vert u\Vert_{H^1}&\leq \Vert u_0\Vert_{H^1} +
    \Vert\int_0^t e^{-\int_s^t c(\tau)d\tau}K*u^pds\Vert_{H^1}\\
    &\le \Vert u_0\Vert_{H^1} + \int_0^t e^{-\underline{c}(t-s)}\Vert K*u^p\Vert_{H^1} ds\\
    &=\Vert u_0\Vert_{H^1} +\underline{c}^{-1}(1-e^{-\underline{c}t})\Vert K*u^p\Vert_{H^1}<\infty,
\end{align*}
where $\underline{c}=\inf_{t>0}c(t)>0$.
\end{proof}
When $L<\infty$, we can use the Rellich-Kondrachov theorem to directly infer that the solution is precompact in $L^2$. Indeed, we can show the solution is precompact in $H^1$. We will discuss in the next subsection with detail.
On the other hands, when $L=\infty$, we need further restriction on $u$ to derive the compactness result. The following lemma will be useful to show the tightness of the bell-shaped solution.
\begin{lemma}[Uniform $L^1$-norm]
Suppose $u_0\in H^1 \cap L^1$ then there exists $C>0$ such that
\begin{equation*}
    \Vert u(t,\cdot)\Vert_{L^1}<C.
\end{equation*}
\begin{proof}
The argument can be much simpler when $p\ge 2$. We claim that $u(t,x)$ is uniformly bounded in $L^p$ using the uniform $L^2$ and $L^\infty$ bound,
\begin{equation*}
    \int u^p dx\leq \Vert u\Vert_\infty^{p-2} \int u^2 dx<\infty.
\end{equation*}
Now we use Gr\"onwall type estimate on $L^1$ norm of $u$ as follows,
\begin{align*}
    \frac{d}{dt}\int udx=\int K*u^pdx - c(t)\int udx\le \Vert K\Vert_{L^1}\Vert u(t,x)\Vert_{L^p}^p - \underline{c} \int udx.
\end{align*}

The general proof starts from~\eqref{e: u fp},
\begin{align*}
    u(x,t)&=e^{-\int_0^tc(\tau)d\tau}u_0+\int_0^t e^{-\int_s^t c(\tau)d\tau}K*u^pds.
\end{align*}
Let $n\in\N$ be such that $p^n\ge 2$ so that we have uniform control in $L^{p^{n}}$, then
\begin{align*}
    \Vert K*u^p\Vert_{L^{p^{n-1}}}&\le \Vert K\Vert_{L^1}\Vert u^p\Vert_{L^{p^{n-1}}}\\
    &\le \Vert K\Vert_{L^1}\Vert u\Vert_{L^{p^{n}}}^{p}.
\end{align*}
Therefore,
\begin{align*}
    \Vert u\Vert_{L^{p^{n-1}}} \le \Vert e^{-\int_0^tc(\tau)d\tau}u_0\Vert_{L^{p^{n-1}}} +\int_0^t e^{-\int_s^t c(\tau)d\tau}\Vert K*u^p\Vert_{L^{p^{n-1}}} ds\\
    \le e^{-\underline{c}t} \Vert u_0\Vert_{L^{p^{n-1}}} +\underline{c}^{-1}(1-e^{-\underline{c}t})\Vert K*u^p\Vert_{L^{p^{n-1}}},
\end{align*}
where $\underline{c}:=\inf_{t\ge 0} c(t) >0$, constructed by Proposition~\ref{p: conv func}. So we have uniform control in $L^{p^{n-1}}$.
We recur the argument until we reach $n=0$ to get $L^1$ control.
\end{proof}
\end{lemma}
Using the previous lemma, we will show that the solution is precompact in $L^2$. Prior to that, we show some interesting consequences of the lemma. Directly, we can show that the solution cannot converge uniformly to $0$.
\begin{corollary}\label{c: nz}
Suppose $u_0\in H^1 \cap L^1$  then
    $\lim\inf_{t\to\infty} \Vert u(t,\cdot)\Vert_\infty >0$.
\end{corollary}
\begin{proof}
        For the sake of contradiction, let's assume that $\lim\inf_{t\to\infty} \Vert u(t,\cdot)\Vert_\infty =0$.
        
        \begin{align*}
          c(t)= \int u^pK*u^p dx \leq  \Vert u(t,\cdot)\Vert_\infty^p \int K*u^p dx\leq  \Vert u(t,\cdot)\Vert_\infty^p \Vert K\Vert_{L^1}\Vert u(t,\cdot)\Vert_{L^p}^p.
        \end{align*}
        Since we achieve the uniform control on $L^p$ norm and $c(t)$ converges to a positive constant this leads to the contradiction.
\end{proof}
The next corollary will be useful to show $L^2$-compactness.
\begin{corollary}[$L^q$ tightness]\label{c: lqt}
    Let $q>1$ and $u(t,x)$ be a solution of \eqref{e: solwav} such that $u_0\in H^1\cap L^1$ is bell shaped. Then for any $\epsilon>0$ there exists $R>0$ such that 
    \begin{equation*}
        \int_{|x|>R} u^q dx <\epsilon.
    \end{equation*}
\end{corollary}
\begin{proof}
      Note that $u(t,x)$ is bell-shaped. Using the previous lemma and the Chebyshev's inequality, for $x>0$
        \begin{align*}
            2x u(t,x) \leq \Vert u(t,x)\Vert_{L^1}<C.
        \end{align*}
        Therefore,
    \begin{equation*}
        u(t,x)\le C/2x
    \end{equation*}
    
    \begin{equation*}
        \int_{|x|>R} u^q dx < (\frac{C}{2})^q\int_{|x|>R} 1/|x|^q dx =\Tilde{C} R^{-q+1}<\epsilon.
    \end{equation*}
\end{proof}

Now we show that any bell-shaped solutions are precompact in $L^2(\R)$.
\begin{theorem}[Compactness on $\R$]\label{t: comr}
Let $u(t,x)$ be a solution of~\eqref{e: solwav}. If $u_0\in H^1(\R)$ is bell-shaped then $\{u(t,x)\}_{t\ge0}$ is precompact in $L^2(\R)$.
\end{theorem}
\begin{proof}
    Since $u(t,\cdot)$ is bell-shaped for all $t\ge0$ by Theorem~\ref{t: ubs}, we can use Helly's selection theorem to have a (further) subsequence $\{u(t_k)\}_{k\in\N}$ and the pointwise limit $\ut$. Recall that we have uniform bound on $u$ and $u_x$ by Proposition~\ref{p: u upperB}. This means, by Arzela-Ascoli theorem, on any compact interval, the (subsequential) convergence is uniform and $\ut$ is continuous on $\R$.
    
    Now we claim that $\Vert u(t_k)-\ut\Vert_{L^2}\to 0$ as $k\to\infty$. For given $\epsilon>0$, we use Corollary~\ref{c: lqt} to find $R>0$ to have small $L^2$-tail bound and we choose $k\in \N$ large enough to have small $L^2$-bound, i.e.,
    \begin{equation*}
        \Vert u(t_k)-\ut\Vert_{L^2}^2 \leq  \int_{|x|\le R} (u(t_k)-\ut)^2 dx +\int_{|x|>R} (u(t_k)-\ut)^2 dx<\epsilon.\qedhere
    \end{equation*}
\end{proof}

\paragraph{\textbf{$H^1$-compactness of the solutions ($L<\infty$)}}
In the nice case, when $L<\infty$ and $u_x$ is uniformly bounded, we can use Arzela-Ascoli's theorem to directly infer the compactness of solutions.
\mnote{bell shaped: precompact in uniform norm}
\begin{corollary}\label{c: bell compact}
    Assume $L<\infty$ and let $u\in\mathcal{V}$ be a solution of \eqref{e: solwav} such that $u_x\in L^\infty$. Then $\{u(x,t)\}_{t\ge 0}$ is precompact in $C_0(\mathds{T}_L)$. 
\end{corollary}
Note that we can relax the assumption of uniform bound on $u_x$, using usual diagonalization trick. In fact, we can choose another approach to prove compactness of the solution, even with $L^2$-initial data.

Recall that in semi-linear heat equation, the Laplacian operator provides smoothing action resulting the solution trajectory to be precompact. Even though not as powerful as the Laplacian operator, convolution with a nice function is renowned for regularization. From the proof of Theorem~\ref{t: wpbs} we have
\begin{align}
     u(x,t)&=e^{-\int_0^tc(\tau)d\tau}u_0+\int_0^t e^{-\int_s^t c(\tau)d\tau}K*u^pds.\label{e: u fp}
\end{align}
Although we cannot claim that $u(x,t)$ is smoother than the initial data $u_0$, we can see that $u(x,t)$ has a smoother part.
Indeed, the fixed point formula~\eqref{e: u fp} suggests that $u$ consists of $u_0$ with a scaling factor and a smoother trajectory created by the convolution with $K$. It is immediate to see that the first part of~\eqref{e: u fp} converges to $0$ as $t\to\infty$, since  $c(t)$ converges to a positive constant as $t\to\infty$. So the first part of~\eqref{e: u fp} is precompact. Therefore once we prove the second part of~\eqref{e: u fp} is precompact, we prove that $\{u(t,x)\}_{t\ge0}$ is precompact.
\begin{theorem}\label{t: ucomp}
    Let $L<\infty$ and $u(t,x)$ be a solution of~\eqref{e: solwav}. If $u_0\in H^1(\mathds{T}_L)$ then $\{u(t,x)\}_{t\ge0}$ is precompact in $H^1(\mathds{T}_L)$.
\end{theorem}
\begin{proof}
from \eqref{e: u fp} and the previous discussion, it is enough to show that the trajectory $t\mapsto\int_0^t e^{-\int_s^t 2pE(u)d\tau}K*u^pds$ is compact in $H^1(\mathds{T}_L)$.
From Proposition~\ref{p: conv func} we know that $c(t)$ stays positive and converges, i.e., there exists $\underline{c}:=\inf_{t\ge 0} c(t) >0$. Without specifying the norm, we have
\begin{align*}
    \Vert\int_0^t e^{-\int_s^t c(\tau)d\tau}K*u^pds\Vert &\le \int_0^t e^{-\underline{c}(t-s)}\Vert K*u^p(s,\cdot)\Vert ds\\
    &=\underline{c}^{-1}(1-e^{-\underline{c}t})\sup_{\tau\ge0}\Vert K*u^p(\tau,\cdot)\Vert.
\end{align*}
So once we achieve time independent control on the integrand, we can control the integral term likewise.

Firstly we use Proposition~\ref{p: conv func} and Proposition~\ref{p: u upperB} to achieve uniform $H^1$ bound,
\begin{align*}
     \Vert K*u^p\Vert_{H^1}&=\Vert K*u^p\Vert_{L^2}+\Vert K_x*u^p\Vert_{L^2}\\
    &\le \Vert K\Vert_{W^{1,1}} \Vert u^p\Vert_{L^2}\\
    &\leq \Vert K\Vert_{W^{1,1}} \Vert u\Vert_{L\p}^{\frac{p+1}{2}} \Vert u\Vert_\infty^\frac{p-1}{2} <\infty.
\end{align*}
Since $u_0\in\mathcal{V}$, by \eqref{e: u fp} we have uniform $H^1$ bound on $u$. 

Finally, we claim that $\Vert K*u^p\Vert_{H^2}<C$ for a time independent $C>0$. It is enough to check $\Vert K_x*(u\pp u_x)\Vert_{L^2}.$
\begin{align*}
    \Vert K_x*(u\pp u_x)\Vert_{L^2}&=\Vert \int K_x(x-y)u\pp(y)u_x(y)dy\Vert_{L^2}\\
    &\le  \Vert K\Vert_{W^{1,1}} \Vert \Vert u\Vert_\infty\pp \Vert u\Vert_{H^1}<\infty.
\end{align*}
So the trajectory $t\mapsto\int_0^t e^{-\int_s^t 2pE(u)d\tau}K*u^pds$ is uniformly bounded in $H^2$, and compact in $H$. Therefore by the Rellich–Kondrachov theorem we have the result.
\end{proof}

\section{Nontrivial wave profiles}
The famous LaSalle's invariance principle states that any precompact of gradient flow of $C^1$ energy $E$ approaches to the set of critical points of $E$. Since we achieve the compactness, one can expect that the subsequential limit point to be a critical point. Recall that the uniform lower bound holds only on the compact domain, so LaSalle's principle cannot be directly applied. So we will perform careful analysis to show the limit point on $\R$ satisfies the fixed point, which shows the existence of the solitary wave profile. Additionally, we show the instability of the constant solution, so that we have nontrivial limit point with the choice of the initial data.
\paragraph{\textbf{Instability of constant solutions}}
Note that $\Bar{u}$ is a constant satisfying $L\bar{u}\p=1$, $\Bar{u}$ is a fixed point of \eqref{e: solwav}. Considering that we aim to calculate nontrivial equilibria of \eqref{e: solwav}, we need to investigate the stability of $\bar{u}.$

Since \eqref{e: solwav} is linear in $K$, without loss of generality we assume that $\int K(y)dy=1.$
 In this case, 
\begin{equation*}
    2pE(\Bar{u})=\int_{\mathds{T}_L}\Bar{u}^pK*\Bar{u}^p=L\bar{u}^{2p}=\bar{u}\pp,
\end{equation*}

Recall that
\begin{align*}
  DF(u)(v) = e^{-\frac{2p}{p+1}\int u\p dx}\int u\pp(K*u^p - c(t)u)vdx.
\end{align*}
The Hessian at an equilibrium $\ut$ is the following,
\begin{align*}
  D^2F(\ut)(v,w) &= D(e^{-\frac{2p}{p+1}\int \ut\p dx})(w) \int \ut\pp\cancelto{0}{(K*\ut^p - c(t)u)}vdx\\& + e^{-\frac{2p}{p+1}\int \ut\p dx}D(\int \ut\pp(K*\ut^p - c(t)\ut)vdx)(w)\\
  &=e^{-\frac{2p}{p+1}\int \ut\p dx}\int (p-1)\ut^{p-2}w\cancelto{0}{(K*\ut^p - c(t)\ut)}vdx\\ 
  &+ e^{-\frac{2p}{p+1}\int \ut\p dx} \int \ut\pp\Big(pK*(\ut\pp w) \\& - 2p(\int \ut\pp (K*\ut^p) wdx)\ut  -2pE(\ut)w\Big)vdx\\
  &=\langle v,D^2F(\ut)w\rangle_{L^2}.
\end{align*}
\begin{align}\label{e: fl hess}
    D^2F(\Bar{u})(v)=e^{-\frac{2p}{p+1}}\Bar{u}^{2p-2}(pK*v -\frac{2p}{L}\int_{\mathds{T}_L}vdy -v).
\end{align}
Note that, the constant function $\mathds{1}$ is an eigenvector with negative eigenvalue,
\begin{equation*}
    D^2F(\bar{u})(\mathds{1})=-e^{-\frac{2p}{p+1}}\Bar{u}^{2p-2}(p+1)\mathds{1}.
\end{equation*}
In fact, other possible eigenvectors pointing outside of the invariant manifold $\mathcal{M}$ have the same negative eigenvalue. Let's take a look at the following eigenvalue problem,
\begin{align*}
    K*v -\frac{2}{L}\int_{\mathds{T}_L}vdy&=\lambda v\\
    \int_{\mathds{T}_L} (K*v -\frac{2}{L}\int_{\mathds{T}_L}vdy) dx&=\lambda\int_{\mathds{T}_L} vdx\\
   - \int_{\mathds{T}_L} vdx&=\lambda\int_{\mathds{T}_L} v dx.
\end{align*}
So we can conclude that, either
\begin{equation*}
    \lambda=-1 \text{ or } \int_{\mathds{T}_L} v dx=0.
\end{equation*}
Note that the set of mean zero functions corresponds to the tangent manifold at the constant equilibrium $\bar{u}$, i.e., $T_{\bar{u}}\mathcal{M}=\{v\in\mathcal{V}: \int_{\mathds{T}_L} v dx=0 \}$.
So for such $v\in T_{\bar{u}}\mathcal{M}$,
\begin{align}
    D^2F(\Bar{u})(v)=e^{-\frac{2p}{p+1}}\Bar{u}^{2p-2}(pK*v -v).
\end{align}
We claim that $ D^2F(\Bar{u})$ has positive eigenvalues, which implies the instability of the constant solution $\bar{u}.$
\paragraph{\textbf{Spectrum of convolution operator}}
Note that for $\lambda\in\R$, $e^{i\lambda x}$ is an eigenvector of convolution operators,
\begin{align*}
        K*e^{i\lambda x}&=\int e^{i\lambda(x-y)}K(y)dy\\
        &=e^{i\lambda x}\int e^{-i \lambda y}K(y)dy\\
        &=\hat{K}(\lambda) e^{i\lambda x}.
\end{align*}
Recall that the perturbation $\phi$ we chose has mean zero, i.e. $\int \phi dx=0$. In order to satisfy $\int_{\mathds{T}_L}e^{i\lambda y}dy=0$, or $e^{i\lambda x}\in H^1(\mathds{T}_L)$, $\lambda$ must be resonant with $L$, i.e., $\lambda=\frac{2n\pi}{L}$ for $n\in\N$.

Since $K$ is even, we can check that the corresponding eigenvalue $\hat{K}(\lambda)$ of the eigenvector $e^{i\lambda x}$ is positive,
\begin{align*}
    \hat{K}(\lambda)&=\int (\cos(\lambda y)-i\sin(\lambda y))K(y)dy\\
    &=\int \cos(\lambda y)K(y)dy>0.
\end{align*}

Therefore, the eigenvalues of $D^2F(\Bar{u})$ on $\mathcal{M}$ is $e^{-\frac{2p}{p+1}}\Bar{u}^{2p-2}(p \hat{K}(\frac{2n\pi}{L})-1)$.
Because $\hat{K}(\frac{2n\pi}{L})\to 1$ as $L\to \infty$, for any $p>1$ there exists large $L>0$ such that $\Bar{u}$ is a saddle point.
Moreover, every eigenvalues of $D^2E(\Bar{u})$ becomes positive as $L\to \infty$, which corresponds to the fact that $0$ is a source when $L=\infty$.

\paragraph{\textbf{Existence of Nontrivial Wave Profiles}}
One of the benefits of a gradient system is that by LaSalle's invariance principle for any bounded solution $u(t)$ of gradient flow of $C^1$ energy, we have $\lim_{t\to\infty}\Vert u'(t)\Vert_\mathcal{H} =0$. So by showing the compactness of trajectory, we infer the existence of periodic and solitary wave profiles. Furthermore, we can show that there is a nontrivial wave profile by starting with $u_0=\bar{u}+\delta\cos(\frac{2\pi}{L}x)$ for a small enough $\delta>0$, we can conclude that there is a nontrivial equilibrium of~\eqref{e: solwav}.
\begin{theorem}[Existence of nontrivial wave profiles]\label{t: npw}
For large enough $L\in(0,\infty]$, there exists nontrivial equilibrium $\Tilde{u}\in\mathcal{M}$ of ~\eqref{e: solwav} such that $K\ast \Tilde{u}^p =\Tilde{c}\Tilde{u}$ for some $\Tilde{c}>0$ .
\end{theorem}
\begin{proof}
1. $L<\infty$:
Let $\bar{u}$ be the constant equilibrium such that $L\bar{u}\p=1$ and $\nabla F(\bar{u})=0$.
We continue the argument from the discussion about the instability of $\bar{u}.$
Let $L>0$ be large enough such that $p \hat{K}(\frac{2\pi}{L})-1\ge 0$ and let the perturbation be $\phi(x)=\cos(\frac{2\pi}{L}x)$ so that $\nabla^2  F(\bar{u})(\phi,\phi)>0$.
Using Taylor's theorem for small enough $\delta>0$ we have
        \begin{align*}
            F(\bar{u}+\delta\phi)=F(\bar{u})+ \frac{1}{2}\delta^2\nabla^2  F(\bar{u})(\phi,\phi)+ o(\delta^2)
            > F(\bar{u}).
        \end{align*}
        
Now we initiate the evolution starting with $u_0=\bar{u}+\delta\cos(\frac{2\pi}{L}x)$. Then by the compactness result, Theorem~\ref{t: ucomp}, we can find a limit point $\ut.$

Note that when $L<\infty$, the solution is uniformly bounded above, and below away from $0$ by Proposition~\ref{p: u upperB} and Corollary~\ref{c: low u}. So the norm defined by the metric $\langle\cdot,\cdot\rangle_u$ is equivalent to $L^2$ norm, and we can use LaSalle's invariance principle. Technically, we have
\begin{equation*}
    \Vert \nabla F(u)\Vert_{L^2}\ge C \Vert u'\Vert_{L^2},
\end{equation*}
which implies $F$ is a strict Lyapunov function for $\{u(t)\}_{t\ge0}$ (Definition 8.4.5 of~\cite{haraux2015convergence}).
We use LaSalle's invariance principle (Theorem 8.4.6 of~\cite{haraux2015convergence}) to deduce that 
\begin{equation*}
    \Vert \nabla F(\ut) \Vert_\mathcal{H} =0.
\end{equation*}
Note that since $F$ is increasing along the evolution we have
$F(\ut)>F(\bar{u})$, which implies $\ut\not=\bar{u}.$

2. $L\le\infty$: 
Here we provide more careful analysis for $L=\infty$, since $F$ might not be a strict Lyapunov function, or the inequality
\begin{equation*}
    \Vert \nabla F(u)\Vert_{L^2}\ge C\Vert u'\Vert_{L^2},
\end{equation*}
might not hold. We will proceed with the arguments as follows,
\begin{enumerate}
    \item $t\mapsto u^q(t,\cdot)$ is H\"older continuous in $L^2$ for any $q\ge{\frac{p+1}{2}}$.
    \item $t\mapsto \frac{d}{dt}(u\p)(t,\cdot)$ is uniformly continuous in $L^2$.
    \item When $u_0$ is bell-shaped, $\Vert\frac{d}{dt}(u\p)(t,\cdot)\Vert_{L^2}\to 0$ as $t\to\infty$.
    \item For any $\ut\in\omega(u)$, $\Vert \ut^p(K*\ut^p-\Tilde{c}\ut)\Vert_{L^2}=0$.
    \item $K*\ut^p=\Tilde{c}\ut$.
    \item $\ut$ is nontrivial.
\end{enumerate}

Step (1): Claim that for any $q\ge{\frac{p+1}{2}}$
\begin{equation*}
   \Vert u^q(t)-u^q(s)\Vert_{L^2}\le C|t-s|^{1/2}.
\end{equation*}
Note that by the dissipation of $F$,
\begin{equation*}
    \alpha(t) \frac{d}{dt}F(u(t))=\int_\R u\pp (K*u^p-c(t)u)^2dx=\frac{4}{(p+1)^2}\Vert \frac{d}{dt}u^{\frac{p+1}{2}}(t)\Vert_{L^2}^2
\end{equation*}
Therefore,
\begin{align*}
    \Vert u^{\frac{p+1}{2}}(t)-u^{\frac{p+1}{2}}(s)\Vert_{L^2}&\le \int_s^t \Vert \frac{d}{d\tau}u^{\frac{p+1}{2}}(\tau)\Vert_{L^2}d\tau\\
    &=|t-s|^{1/2}\int_s^t\Vert \frac{d}{dt}u^{\frac{p+1}{2}}(t)\Vert_{L^2}^2 d\tau\\
    &\le\sup_{\tau\ge0}\alpha(\tau) \frac{(p+1)^2}{4} \int_s^t \frac{d}{dt}F(u(\tau))d\tau|t-s|^{1/2}\\
    &\le C|t-s|^{1/2}.
\end{align*}
Note that since $u$ is bounded by Proposition~\ref{p: u upperB}, $u\mapsto u^\gamma$ for any $\gamma>1$ is (locally) Lipschitz, i.e.,
\begin{equation}
    |u^\gamma(t)-u^\gamma(s)|\le C|u(t)-u(s)|,
\end{equation}
where the Lipschitz constant $C>0$ may depend on $\Vert u\Vert_{L^\infty}$.
So for any $q\ge{\frac{p+1}{2}}$ we choose $\gamma=\frac{2q}{p+1}$ so that $u^q=(u^{\frac{p+1}{2}})^\gamma$. Then
\begin{equation*}
    \Vert u^q(t)-u^q(s)\Vert_{L^2}^2\le C\int_\R |u^{\frac{p+1}{2}}(t)-u^{\frac{p+1}{2}}(s)|^2 dx \le C|t-s|.
\end{equation*}

Step (2): Claim that $t\mapsto\frac{d}{dt}(u\p)(t,\cdot)$ is uniformly continuous in $L^2$. Because,
\begin{equation*}
   \frac{1}{p+1} \frac{d}{dt}(u\p)=u^pK*u^p -c(t)u\p,
\end{equation*}
and $u^p(t)$ and $u\p(t)$ are uniformly continuous in $L^2$ and bounded uniformly in $L^\infty$, it is enough to show that $t\mapsto c(t)$ is H\"older continuous.
\begin{align*}
    c(t)-c(s)&= \int (u^p(t)-u^p(s))K*u^p(t) dx+ \int (u^p(t)-u^p(s))K*u^p(s) dx\\
    &\leq 2\Vert u^p(t)-u^p(s)\Vert_{L^2} \Vert K*u^p\Vert_{L^2}\\
     &\leq 2\Vert u^p(t)-u^p(s)\Vert_{L^2} \Vert K\Vert_{L^1}\Vert u\Vert_{L\p}^{\frac{p+1}{2}} \Vert u\Vert_\infty^\frac{p-1}{2}.
\end{align*}
By Proposition~\ref{p: conv func} and the previous step, we have the claim.

Step (3): Let $u_0$ be bell-shaped and claim that $\Vert\frac{d}{dt}(u\p)(t,\cdot)\Vert_{L^2}\to 0$ as $t\to\infty$.  We use the uniformly continuity of $t\mapsto\frac{d}{dt}(u\p)(t,\cdot)$ in $L^2$. For any $\epsilon>0$ and $t>0$, we can find $\delta>0$ such that for any $h\in[0,\delta]$
\begin{equation*}
    \Vert\frac{d}{dt}(u\p)(t,\cdot)\Vert_{L^2} \le  \epsilon+ \Vert\frac{d}{dt}(u\p)(t+h,\cdot)\Vert_{L^2}.
\end{equation*}
Therefore, by averaging
\begin{align*}
     \Vert\frac{d}{dt}(u\p)(t,\cdot)\Vert_{L^2} &\le \epsilon+ \frac{1}{\delta}\int_0^\delta \Vert\frac{d}{dt}(u\p)(t+h,\cdot)\Vert_{L^2} dh\\
     &\le \epsilon+ \delta^{-1/2}(\int_t^{t+\delta} \Vert\frac{d}{dt}(u\p)(s,\cdot)\Vert_{L^2}^2 ds)^{1/2}\\
     &\le \epsilon+ \delta^{-1/2}(\int_t^{\infty} \Vert 2u^{\frac{p+1}{2}}(s,\cdot)\frac{d}{dt}(u^{\frac{p+1}{2}})(s,\cdot)\Vert_{L^2}^2 ds)^{1/2}\\
     &\le \epsilon+ C(\int_t^{\infty} \Vert \frac{d}{dt}(u^{\frac{p+1}{2}})(s,\cdot)\Vert_{L^2}^2 ds)^{1/2}\\
     &\le \epsilon+ C(F(\infty)-F(t))^{1/2}.
\end{align*}
which implies the claim.

Step (4): Claim that $\Vert \ut^p(K*\ut^p-\Tilde{c}\ut)\Vert_{L^2}=0$ where $\Tilde{c}=\int_\R \ut^pK*\ut^p dx$.
Since the $L^2$-compactness, Theorem~\ref{t: comr}, we have a increasing sequence time  $\{t_k\}_{k\in\N}\subset \R_+$ and the limit $\ut\in L^2(\R)$ such that 
\begin{equation*}
    \lim_{k\to\infty}\Vert u(t_k)-\ut\Vert_{L^2}=0.
\end{equation*}
Using uniform $L^\infty$-bound, one can show that $u\mapsto u^p, u\p$ is continuous in $L^2$, and we have the claim by continuity. Note that as we discussed in the proof of Theorem~\ref{t: comr}, $\ut$ is a also continuous subsequential limit.

Step (5): Claim that either $\ut\equiv 0$ or $\ut>0$ so that $K*\ut^p=\Tilde{c}\ut$. Recall that $\ut$ is continuous and bell-shaped. Suppose $\ut$ is not identically $0$ but not strictly positive. Since $K*\ut^p>0$, the minimum of $K*\ut^p$ on the support of $\ut$ is positive and let's call it $m$. Now we can find an open interval in the support of $\ut$ such that $m-2\Tilde{c}\epsilon>0$ and $\epsilon<\ut<2\epsilon$, so that in the interval
\begin{equation*}
    |\ut^p(K*\ut^p-\Tilde{c}\ut)|\ge \epsilon^p(m-2\Tilde{c}\epsilon)>0,
\end{equation*}
which leads to the contradiction to the acquired Step (4).

Step (6): Claim that $\ut>0$. We can conclude using the continuity of $\ut$, Corollary~\ref{c: nz} and the fact that $\ut$ is also a pointwise limit of $u(t_k)$.
\end{proof}

\paragraph{\textbf{Discussion}}
In fact, there is an issue with applying the \L ojasiewicz convergence criterion to ~\eqref{e: solwav}. 
So at present we are unable to conclude convergence of solutions as $t\to\infty.$
Let $H^1_+(\mathds{T}_L)$ denote the positive cone in $H^1(\mathds{T}_L)$, i.e.,
$$H^1_+(\mathds{T}_L):=\{u\in H^1(\mathds{T}_L): u(x)> 0, \;\; \forall x\in H^1(\mathds{T}_L)\}.$$
Note that $H^1(\mathds{T}_L)$ is a Banach algebra. So the Lyapunov functional,
\[
F(u)= \frac{1}{2p}e^{-\frac{2p}{p+1}\int u\p dx}\int u^pK*u^pdx,
\]
is analytic in the neighborhood of $H^1_+(\mathds{T}_L).$ 

The issue, however, arise from failure of the Hessian of $F$ 
$$D^2F(\ut): w\mapsto e^{-\frac{2p}{p+1}}\ut\pp\left({p K*(\ut\pp w)} -{4p^2E(\ut)(\int \ut^pwdx)\ut} -2pE(\ut)w\right)$$
being semi-Fredholm. 

Firstly, on $\R$ there is no chance for $D^2E(\ut)$ to have closed range since $\ut(x)\to0$ as $x\to \pm\infty$. For example, for any $u\in\mathcal{V}$,  $\big\{u\pp\mathds{1}_{\{-L,L\}}\big\}_{L>0}$ is Cauchy in $\mathcal{V}'$ but $\mathds{1}\not\in \mathcal{V}$.
So let's assume that $L<\infty$. Now since $\ut$ is bounded above and below from $0$, the multiplication with $\ut$, $v\mapsto \ut v$ is an isomorphism from $\mathcal{V}$ to itself.
By Theorem~\ref{t: fredcom} one can check that 
is semi-Fredholm from $\mathcal{V}$ to itself, but not from $\mathcal{V}$ to $\mathcal{V}'$.

\mnote{E is not well defined in H}
On the contrary, the key to expand the \L ojasiewicz inequality to infinite dimensional domain is to construct the finite dimensional copy of the energy, which can be approximated by $\Vert \nabla E(u)\Vert_\mathcal{H}$, like in Proposition~\ref{p: gam}. When $\Pi+A: \mathcal{V}\to\mathcal{V}'$ is an isomorphism, we represent $u\sim(\Pi u,\nabla E(u))$ using the isomorphism $\mathcal{N}: \mathcal{V}\to\mathcal{V}'$ such that
\begin{align*}
   \mathcal{N} u=\Pi u+\nabla E(u).
\end{align*}
The issue with \eqref{e: solwav} for applying the infinite dimensional framework in Section~\ref{c: inf dim} of Chapter~\ref{ch: gf} is $R(\Pi+A)$ is too regular to be closed in $\mathcal{V}'$. Moreover, even if we identify $\Pi+A$ as an isomorphism from $\mathcal{V}$ to itself, we will end up having $\Vert \nabla{E}(u)\Vert_\mathcal{V}$, not in control of $\mathcal{H}.$ 

\chapter{Regularized concentration-dispersion dynamics}
In this chapter, we propose and analyze a modification of concentration-dispersion dynamics with small $\epsilon>0$, as follows.
\paragraph{\textbf{Regularized concentration-dispersion equation}}
   \begin{equation}\label{e: rsolwav}
    \left\{
    \begin{aligned}
        \;\;\frac{d}{dt}u(t,x)&=\epsilon (u^p)_{xx} +K\ast u^p  - c(t) u\\
        c(t)&=\int u^p K\ast u^p- \epsilon[(u^p)_x]^2 dx,
    \end{aligned}
     \right.
    \end{equation}
where $K(x)>0$ for all $x\in \mathds{T}_L$ and $K(x)\in W^{1,1}\cap W^{1,p+1}(\mathds{T}_L)$.
We will use the gradient structure of the regularized concentration-dispersion equation~\eqref{e: rsolwav} to prove convergence of solutions to a nontrivial equilibrium. From now on we consider only $L<\infty.$
\paragraph{\textbf{Well-posedness}}
For local well-posendess with positive initial data, we refer Chapter 8 of \cite{lunardi2012analytic}, and set $X=C(\mathds{T}_L)$ and $D=C^2(\mathds{T}_L)$. Considering singular aspect of solutions of porous medium equations, global well-posedness is not trivial. By combining the local well-posedness, regularity (Theorem 8.1.1, Theorem 8.3.4 of~\cite{lunardi2012analytic}) and the $C^2$-compactness result in the section 2, we deduce regularity up to $C^5(\mathds{T}_L)$. Thus for any $u_0\in C^5(\mathds{T}_L)$ initial data, we can choose sufficiently small $\epsilon>0$ to achieve global well-posedness in $C^2(\mathds{T}_L)$ by pointwise maximum principle and maximal solution argument (Proposition 8.2.1 of~\cite{lunardi2012analytic}).

\section{Gradient structure}
Based on those regularity properties, we show the asymptotic bound of the solution $u(t,x)$. 
Prior to show that the solution stay above from $0$ uniformly, we show the solution $u(t,x)$ stays positive. In this way, \eqref{e: rsolwav} is well-defined.
\begin{prop}[Positivity]
    Let $u(t,x)\in C^2(\mathds{T}_L)$ be a solution of~\eqref{e: rsolwav}. If $u(0,x)>0$ for all $x\in \mathds{T}_L$ then $u(t,x)>0$ for all $x\in \mathds{T}_L$ and $t>0.$
\end{prop}
\begin{proof}
Let $\beta(t):=\exp{(\int_0^t c(s)ds)}$ and $v(t,x):=\beta(t)u(t,x)$. We claim that for all $t\ge0$, 
$$v(t,x)\ge \min_x v(0,x).$$ 
Suppose at $t_0>0$, $x_0=\arg\min_x v(t_0,x)$, then
    \begin{align*}
            \frac{dv}{dt}(t_0,x_0)=\beta^{1-p}(t)\left(\epsilon (v^p)_{xx}+ K\ast v^p\right)\ge \beta^{1-p}(t) K\ast v^p >0.
    \end{align*}
In particular, $\{v\in C^2(\mathds{T}_L): v>0\}$ is an invariant set of \eqref{e: rsolwav}. So for all $t>0$,
\begin{equation*}
    u(t,x)=\beta^{-1}(t)v(t,x)>0.\qedhere
\end{equation*}
\end{proof}

\paragraph{\textbf{Important functionals}}
We will show that \eqref{e: rsolwav} has a nice $L^2$-gradient structure. We use the following functionals.
\begin{align}
    E(u)&:=\frac{1}{2p}\int u^pK*u^p- \epsilon[(u^p)_x]^2dx,\\
    F(u)&:= \frac{1}{2p}e^{-\frac{2p}{p+1}\int u\p dx}\int u^pK*u^p- \epsilon[(u^p)_x]^2dx,\\
    &=:E(u)/\alpha(t)
\end{align}
The derivative of $F$ shows that \eqref{e: rsolwav} is a gradient-like system of the energy $F.$
\begin{align}
      \alpha(t)DF(u)(v)&= DE(u)(v)-E(u)\frac{2p}{p+1}D(\int u\p dx)(v)\notag\\
    &=\int u\pp v K*u^p -\epsilon(u\pp v)_x(u^p)_xdx-2pE(u)\int u^pu' dx\notag\\
    &=\int u\pp(\epsilon(u^p)_{xx} +K*u^p - c(t)u)vdx.
\end{align}
We show the convergence properties of the functionals.
\begin{prop}\label{p: conv rfunc}
    Let $u(t,x)\in L\p(\mathds{T}_L)$ be a solution of \eqref{e: solwav} such that $u_0(x)>0$ for all $x\in\mathds{T}_L$. Assume that $E(u_0)>0$, then the following functionals have convergent properties as $t\to\infty:$
    \begin{enumerate}
         \item $F(u)= E(u)/\alpha(t)$ is nondecreasing,
        \item $\Vert u\Vert_{L\p}$ converges monotonically to 1,
        \item $E(u)$ stays positive and converges to a positive constant.
    \end{enumerate}
    Additionally, if $\Vert u_0\Vert_{L\p}\leq 1$ then $E(u)$ is nondecreasing as well.
\end{prop}
\begin{proof}
    Firstly, we prove monotonic convergence of $F$. 
    $F$ is nondecreasing because the time derivative of $F$ is nonnegative, i.e.,
    \begin{align*}
        F'(u(t))e^{\frac{2p}{p+1}\int u\p dx}&=E'(u)-E(u)\frac{2p}{p+1}(\int u\p dx)'\\
        &=\int u\pp (K*u^p+\epsilon(u^p)_{xx} u'dx-2pE(u)\int u^pu' dx\\
        &=\int u\pp(\epsilon(u^p)_{xx}+K*u^p - c(t)u)^2dx \ge0.
    \end{align*}
   Since $E(u_0)>0$ implies $F(u_0)>0$ and
    \begin{equation}\label{e: fe}
        F(u)=e^{\frac{-2p}{p+1}\int u\p dx}E(u)\leq E(u)=2pc(t),
    \end{equation}
   we have $\inf_{t\ge 0}c(t)>0$. So the monotonic convergence of $L\p$ norm can be shown by taking time-differentiation,
    \begin{equation}\label{e: rulpcon}
        \frac{1}{p+1}(\int u\p dx)'= c(t) \big(1-\int u\p dx\big).
    \end{equation}
    We can conclude 
    $$\lim_{t\to\infty}\Vert u(t)\Vert_{L\p} \to 1.$$
    
Now we claim that $E(u)$ is bounded above.
    \begin{align}
        2pE(u) &\le \int u^p K\ast u^p dx\notag\\
            &\leq \Vert u\Vert^p_{L\p} \Vert K*u^p\Vert_{L\p}\notag\\
            &\leq \Vert K\Vert_{L^\frac{p+1}{2}} \Vert u\Vert^{2p}_{L\p}\notag\\
            &\leq(1+\Vert u_0\Vert_{L\p})^{2p}\Vert K\Vert_{L^\frac{p+1}{2}}<\infty.
    \end{align}
    Therefore $E(u(t))$ converges to a positive constant as $t\to\infty$ and $F$ converges monotonically.
\end{proof}

\section{Asymptotic bound on $u$ and $C^2$-compactness}
We open this section with doing the same pointwise estimate as in the original concen\-tration-dispersion dynamics.
\begin{prop}\label{p: ru upperB}
    Let $u\in \mathcal{V}$ be a solution of \eqref{e: solwav}. Then there exists $M(K,u_0)>0$ such that $\sup_{t\ge0}\Vert u(t,\cdot)\Vert_{L^\infty }<M$.
\end{prop}
\begin{proof}
       We use similar argument in the proof of Proposition~\ref{p: u upperB}. Since $(u^p)_{xx}\le0$ at the maximizer,
        \begin{align}
            u' &= \epsilon (u^p)_{xx}+ K\ast u^p  - c(t)u \\
            &\le \int K(x-y) u(y)^p dy - c(t) u\\
            &\leq (\int K(x-y)\p dy)^{1/p
            +1}(\int u\p dy)^{p/(p+1)} - \underline{c} u.
        \end{align}
        where $\underline{c}=\inf_{t\ge 0}c(t)$. Note that $\underline{c}>0$ because $u>0$ and $F(u)=\frac{1}{2p} e^{-\frac{2p}{p+1}\int u\p dx}c(t)>0$ is increasing according to the Proposition~\ref{p: conv rfunc}.
        This means
        $$u'(t,x) <0 \quad \text{if } u(t,x) > \frac{1}{\underline{c}}\Vert K\Vert_{L\p} \Vert u\Vert^p_{L\p}.$$
        Due to monotonic convergence of $\Vert u\Vert_{L\p}\to 1$ as $t\to \infty$, we can conclude that
        $$\max_{t\ge 0}\Vert u(t,x)\Vert_\infty \leq \Vert u_0\Vert_\infty+\frac{1}{\underline{c}}(1+ \Vert u_0\Vert^p_{L\p})\Vert K\Vert_{L\p}.$$
\end{proof}
In the compact spatial domain, we can use the upper bound to calculate the lower bound. This will be used to control the Riemannian metric. Naturally, the lower bound vanishes as $L\to\infty.$
\begin{corollary}\label{c: rlow u}
    For fixed $L<\infty$, let $u$ be a solution of \eqref{e: solwav} with $u_0\in C^2(\mathds{T}_L)$. Assume that there is $k>0$ such that $K(x)\ge k,$ $\forall x\in\mathds{T}_L$. Then there exist $m(K,L,u_0)>0$ such that $\min_{x}u(x,t) \ge m$.
\end{corollary}
\begin{proof}
   From Proposition~\ref{p: ru upperB}, there exists $M(u_0,K)>0$ such that $u(x,t)<M$ for all $t\geq 0$. Now for any $\delta>0$ we can choose large enough $T>0$ so that 
   $$1-\delta\le\int u\p dx \le M\int u^p dx.$$
   Note that  $\overline{c}:=\sup_{t\ge 0}c(t)$ is well-defined since it converges. Then at the minimizer,
    \begin{align*}
        u'&=\epsilon(u^p)_{xx}+ K\ast u^p -c(t) u\\
        &\ge k\int u^pdx -\overline{c} u\\
        &\ge \frac{(1-\delta)k}{M}-\overline{c} u >0,
    \end{align*}
    when $u < (1-\delta)k/M\overline{c}$. So the result follows.
\end{proof}

\paragraph{\textbf{$C^2$-compactness}}
Recall that the original concentration dispersion equations~\eqref{e: solwav}, provide precompact solution even  without smoothing effect. 
For the regularized equations,
we can derive the following: for any regular enough initial data, we can make $\epsilon$ small enough to make the trajectory precompact in $C^2$. Recall that we need $C^2$-compactness to achieve global well-posedness (cf., \cite{lunardi2012analytic}). We assume $L<\infty$ and analyze asymptotic pointwise bound of $u_x$, $u_{xx}$ and $u_{xxx}$. 
\mnote{the soln need to stay $C^4$}
\begin{theorem}\label{t: rcdcom}
    Let $L<\infty$. For any $u_0\in C^5(\mathds{T}_L)$ with $u_0(t,\cdot)>0$ and  $\Vert u_0\Vert_{L\p}\le 1$, there exists $\epsilon_0>0$ such that for any $\epsilon\in(0,\epsilon_0]$ the solution $\{u(t)\}_{t\ge 0}$ of \eqref{e: rsolwav} with $u(0)=u_0$ has uniformly bounded $C^3$-norm. In particular, $\{u(t)\}_{t\ge 0}$ precompact in $C^2(\mathds{T}_L).$
\end{theorem}
\begin{proof}
    We will firstly calculate pointwise time evolution of $u_x$ and $u_{xx}$ and later argue how to deduce the conclusion. 
        
    The derivatives of each term on~\eqref{e: rsolwav} are as follows,
    \begin{align*}
    \partial_x (u^p)&=pu\pp u_x\\
    \partial^2_x (u^p)&=p(p-1)u^{p-2}u_x^2+ pu\pp u_{xx}\\
    \partial^3_x (u^p)&=p(p-1)(p-2)u^{p-3}u_x^3+ 3p(p-1)u^{p-2}u_x u_{xx}+ pu\pp u_{xxx}\\
    \partial^4_x (u^p)&=p(p-1)(p-2)(p-3)u^{p-4}u_x^4+ 6p(p-1)(p-2)u^{p-3}u_x^2u_{xx}\\
    &\quad +3p(p-1)u^{p-2} u_{xx}^2 +4p(p-1)u^{p-2}u_x u_{xxx} + pu\pp u_{xxxx}.\\
     \partial^5_x (u^p)&=p(p-1)(p-2)(p-3)(p-4)u^{p-5}u_x^5+ 10p(p-1)(p-2)u^{p-4}u_x^3u_{xx}\\
     &\quad +15p(p-1)(p-2)u^{p-3}u_xu_{xx}^2 +10p(p-1)(p-2)u^{p-3}u_x^2u_{xxx}\\&\quad +10p(p-1)u^{p-2}u_{xx}u_{xxx}+5p(p-1)u^{p-2}u_{x}u_{xxxx}+pu^{p-1}u_{xxxxx}.
    \end{align*}
    \begin{align*}
    |\partial_x(K*u^p)|&=|K_x*u^p|\le \Vert K\Vert_{W^1}\Vert u^p\Vert_{L^\infty}\\
    |\partial^2_x(K*u^p)|&=|K_x*(u\pp u_x)|\le \Vert K\Vert_{W^1}\Vert u\pp\Vert_{L^\infty}\Vert u_x\Vert_{L^\infty}.
    \end{align*}
    Since $\Vert u_0\Vert_{L\p}\le 1$, $c(t)$ is increasing, so
    \begin{equation*}
        c(t)=\int u^p K\ast u^p- \epsilon[(u^p)_x]^2 dx \ge c(0).
    \end{equation*}
    Let $c(0)=\alpha-\epsilon\beta$ where,
    \begin{equation*}
     \alpha=\int u_0^p K\ast u_0^pdx, \qquad \beta= \int [(u_0^p)_x]^2 dx .
    \end{equation*}

We use Proposition~\ref{p: ru upperB} and Proposition \ref{c: rlow u} to control the powers of $u$.

1. $u_x$ control: Suppose for a fixed $t>0$, there is a point $x_0\in \mathds{T}_L$, $u_x(t,x_0)$ is positive and reached maximum so that $u_{xx}=(u_x)_x=0$ and $u_{xxx}=(u_x)_{xx}\le0$. Then from~\eqref{e: rsolwav} there exist constants $A,B\in\R$ such that
\begin{equation*}
    \partial_t u_x\le \epsilon A u_x^3+B -(\alpha-\epsilon\beta)u_x.
\end{equation*}
Note that as $\epsilon$ gets smaller, pointwise invariant interval for $u_x$ gets larger.
Thus we can find small enough $\epsilon_0>0$ so that $\limsup_t u_x(t,\cdot)$ is bounded. We can use the similar argument when $u_x$ is negative and reached minimum. Therefore, for $\epsilon\in(0,\epsilon_0]$ we can say $u_x(t)\le \max\{\Vert u_x(0,\cdot)\Vert_{L^\infty}, 1+B/\alpha\}$.

2. $u_{xx}$ control: Suppose $u_{xx}$ is positive and reached maximum so that $u_{xxx}=(u_{xx})_x=0$ and $u_{xxxx}=(u_{xx})_{xx}\le0$. Then from~\eqref{e: rsolwav} there exist constants, possibly depend on $\Vert u_x\Vert_{L^\infty}$, $A',B',D\in\R$ such that
\begin{equation*}
   \partial_t u_{xx}\le \epsilon A' u_{xx}^2 -(\alpha-\epsilon B')u_{xx} +D.
\end{equation*}
Since we already achieved $\epsilon$ and $t$ independent control on $\Vert u_x\Vert_{L^\infty}$ we continue the same argument. As $\epsilon$ gets smaller, pointwise invariant interval for $u_{xx}$ gets larger.
Thus we can find, possibly smaller, $\epsilon_0>0$ so that $\limsup_t u_{xx}(t,\cdot)$ is bounded. We can use the similar argument when $u_{xx}$ is negative and reached minimum. In particular, since the domain is compact, we have uniform $H^2$-norm and by Rellich–Kondrachov theorem.

3. $u_{xxx}$ control: Suppose $u_{xxx}$ is positive and reached maximum so that $u_{xxxx}=(u_{xxx})_x=0$ and $u_{xxxxx}=(u_{xxx})_{xx}\le0$. Then from~\eqref{e: rsolwav} there exist constants, depend on $\Vert u_x\Vert_{L^\infty}$ and $\Vert u_{xx}\Vert_{L^\infty}$ , $A'',B''\in\R$ such that 
\begin{equation*}
   \partial_t u_{xxx}\le A''-(\alpha-\epsilon B'')u_{xxx}.
\end{equation*}
We continue with the similar argument as before to make conclusion.
\end{proof}

\section{Convergence analysis}
Let $L<\infty$ and $\mathcal{V}=H^1(\mathds{T}_L)$, $\mathcal{H}=L^2(\mathds{T}_L)$ so that
\begin{equation*}
    \mathcal{V}\subset\mathcal{H}\subset\mathcal{V}'.
\end{equation*}
We follow the framework presented in Section~\ref{c: inf dim} of Chapter~\ref{ch: gf}.
\paragraph{\textbf{Analyticity}}
We prove convergence in compact spatial domain where $L<\infty$. Let $H^1_+(\mathds{T}_L)$ denote the positive cone in $H^1(\mathds{T}_L)$, i.e.,
$$H^1_+(\mathds{T}_L):=\{u\in H^1(\mathds{T}_L): u(x)> 0, \;\; \forall x\in H^1(\mathds{T}_L)\}.$$
Note that $H^1(\mathds{T}_L)$ is a Banach algebra. So the Lyapunov functional,
\begin{align*}
    F(u)&= \frac{1}{2p}e^{-\frac{2p}{p+1}\int u\p dx}\int u^pK*u^p-\epsilon[(u^p)_x]^2 dx,\\
    &=\frac{1}{2p}e^{-\frac{2p}{p+1}\int u\p dx}\int u^pK*u^p-\epsilon p^2u^{2p-2}u_x^2 dx.
\end{align*}
is analytic in the neighborhood of $H^1_+(\mathds{T}_L).$ In order to show that $F$ satisfies the \L ojaisewicz inequality, let's analyze the derivatives of $F.$
\paragraph{\textbf{Hessian of F}}
Recall that
\begin{align*}
  DF(u)(v) = e^{-\frac{2p}{p+1}\int u\p dx}\int u\pp\Big[\epsilon (u^p)_{xx}+ K*u^p - c(t)u\Big]vdx.
\end{align*}
The Hessian at an equilibrium $\ut$ is the following,
\begin{align*}
  D^2F(\ut)(v,w) &= D(e^{-\frac{2p}{p+1}\int \ut\p dx})(w) \int \ut\pp\cancelto{0}{\Big[\epsilon(\ut^p)_{xx}+K*\ut^p - c(t)u\Big]}vdx\\& + e^{-\frac{2p}{p+1}\int \ut\p dx}D(\int \ut\pp\Big[\epsilon(\ut^p)_{xx}+K*\ut^p - c(t)\ut\Big]vdx)(w)\\
  &=e^{-\frac{2p}{p+1}\int \ut\p dx}\int (p-1)\ut^{p-2}w\cancelto{0}{\Big[\epsilon(\ut^p)_{xx}+K*\ut^p - c(t)\ut\Big]}vdx\\ 
  &\quad+ e^{-\frac{2p}{p+1}\int \ut\p dx} \int \ut\pp\bigg[\epsilon p(\ut\pp w)_{xx}+pK*(\ut\pp w) \\
  &\quad - 2p\Big\{\int \ut\pp (\epsilon(\ut^p)_{xx}+K*\ut^p) wdx\Big\}\ut  -2pE(\ut)w\bigg]vdx\\
  &=\langle v,D^2F(\ut)w\rangle_{L^2}.
\end{align*}
Note that $\Tilde{c}=\lim_{t\to\infty}c(t)=2pE(\ut)$ so that,
\begin{equation*}
    \epsilon(\ut^p)_{xx}+K*\ut^p=\Tilde{c}\ut.
\end{equation*}
We can identify $D^2F(\ut):\mathcal{V}\longrightarrow\mathcal{V}'$ as follows,
\begin{equation*}
    D^2F(\ut): w\mapsto e^{-\frac{2p}{p+1}}\ut\pp\left(\epsilon p(\ut\pp w)_{xx}+{p K*(\ut\pp w)} -{2p\Tilde{c}(\int \ut^pwdx)\ut} -\Tilde{c}w\right).
\end{equation*}

Since $\ut\in C^\infty$ is bounded above, and below away from $0$, the multiplication with $\ut$, $v\mapsto \ut v$ is an isomorphism from $\mathcal{V}$ to $\mathcal{V}$, and from $\mathcal{V}'$ to $\mathcal{V}'$. As we discussed in Proposition~\ref{p: asfn}, $\Delta:\mathcal{V}\to\mathcal{V}'$ with periodic boundary condition is semi-Fredholm. Therefore,
\begin{equation*}
    w\mapsto \ut\pp(\ut\pp w)_{xx}
\end{equation*}
is semi-Fredholm from $\mathcal{V}$ to $\mathcal{V}'$.
So by Theorem~\ref{t: fredcom} $D^2F(\ut)\in\mathcal{L}(\mathcal{V};\mathcal{V}')$ is semi-Fredholm.

\paragraph{\textbf{Convergence result}}
We use Theorem~\ref{t: vloja} and Theorem~\ref{t: conv} to deduce the convergence result as follows.
\begin{theorem}[Convergence result]
Let $u(t)\in \mathcal{V}$ a solution of \eqref{e: solwav}. If $\{u(t)\}_{t\ge0}$ is precompact in $\mathcal{V}$, e.g., as in Theorem~\ref{t: rcdcom}, then there exists $\ut\in\mathcal{E}:=\{u\in \mathcal{V}: \nabla F(u) = 0\}$ such that
    $$\lim_{t\rightarrow\infty}\Vert u(t) - \ut\Vert_\mathcal{V} = 0$$
    Moreover, let $\theta$ be any \L ojasiewicz exponent of $F$ at $\ut$. Then we have
    \begin{equation*}
        \Vert u(t) -\ut\Vert_\mathcal{H} = 
        \left\{\begin{array}{ll}
             O(e^{-\delta t}) & \text{if}\;\; \theta=\frac{1}{2}, \;\;\;\;\text{for some}\;\; \delta>0 \\
             O(t^{-\theta/(1-2\theta)})& \text{if}\;\; 0<\theta<\frac{1}{2} 
        \end{array}\right.
    \end{equation*}
\end{theorem}
\begin{proof}
       We use the uniform bound on $u$ (Proposition~\ref{p: ru upperB} and Corollary~\ref{c: rlow u}) to check the angle condition 
\begin{align*}
   \langle \nabla{F(u)}, u'\rangle &=\alpha(t)^{-1}\int u\pp\Big[\epsilon (u^p)_{xx}+ K*u^p - c(t)u\Big]^2dx\\
    &\ge \sigma\Vert \nabla F(u)\Vert_\mathcal{H}\Vert u'\Vert_\mathcal{H},
\end{align*}
and the rate condition
\begin{equation*}
    \Vert u'\Vert_\mathcal{H}\ge \gamma \Vert \nabla F(u)\Vert_\mathcal{H}.
\end{equation*}

Suppose $\ut\in \omega(u)$ and let  $\bar{F}=\lim_{t\to\infty} F(u(t))=F(\ut)$ and $\theta$ be the \L ojasiewicz exponent at $\ut$. Since the argument is the same with the proof of Theorem~\ref{t: conv}, we conclude the proof by showing the following calculation,
\begin{align*}
   - \frac{d}{dt}(\bar{F}-F(u(t)))^\theta&=(\bar{F}-F(u(t)))^{1-\theta}\langle \nabla{F(u)}, u'\rangle\\
   &\geq \sigma(\bar{F}-F(u(t)))^{1-\theta}\Vert \nabla F(u)\Vert_\mathcal{H}\Vert u'\Vert_\mathcal{H}\\
   &\geq \sigma C\Vert u'\Vert_\mathcal{H}.\qedhere
\end{align*}
\end{proof}

\paragraph{\textbf{Instability of constant solutions}}
Note that $\Bar{u}$ is a constant satisfying $L\bar{u}\p=1$, $\Bar{u}$ is a fixed point of \eqref{e: rsolwav} as well as \eqref{e: solwav}. So as we did in Theorem~\ref{t: npw}, we can find a perturbation $\phi$ at $\bar{u}$ such that 
\begin{equation*}
    F(\bar{u}+\phi)>F(\bar{u}).
\end{equation*}
Indeed, at $\bar{u}$
\begin{align}
    D^2F(\Bar{u})(v)=e^{-\frac{2p}{p+1}}\Bar{u}^{2p-2}(\epsilon pv_{xx}+pK*v -\frac{2p}{L}\int_{\mathds{T}_L}vdy -v).
\end{align}
Note that, the constant function $\mathds{1}$ is an eigenvector with negative eigenvalue,
\begin{equation*}
    D^2F(\bar{u})(\mathds{1})=-e^{-\frac{2p}{p+1}}\Bar{u}^{2p-2}(p+1)\mathds{1}.
\end{equation*}
It is easy to see that the invariant manifold $\mathcal{M}$ attracts the solutions.

For the perturbations in the tangent manifold at $\bar{u}$, i.e., $T_{\bar{u}} \mathcal{M}=\{\phi\in\mathcal{V}:\int_{\mathds{T}_L} \phi dx=0\}$,
\begin{align}
    D^2F(\Bar{u})(\phi)=e^{-\frac{2p}{p+1}}\Bar{u}^{2p-2}(\epsilon p\phi_{xx}+pK*\phi -\phi).
\end{align}
Let $\phi_n(x)=\cos(\lambda_n x)$ where $\lambda_n=\frac{2n\pi}{L}$ then
\begin{equation*}
     D^2F(\Bar{u})(\phi_n)=e^{-\frac{2p}{p+1}}\Bar{u}^{2p-2}(-\epsilon p \lambda_n^2+ p\widehat{K}(\lambda_n)-1)\phi_n.
\end{equation*}
So for small enough $\epsilon$, $D^2F(\bar{u})$ have positive eigenvalues, which shows the instability of the constant solution $\bar{u}.$

In conclusion, we have the following convergence result to a nontrivial equilibrium.
\begin{theorem}
    For large enough $L>0$ there exist small enough $\epsilon>0$ and $\delta>0$ such that the solution $u(t)$ of \eqref{e: rsolwav} with the initial $u_0=\bar{u}+\delta\cos(\frac{2\pi}{L}x)$, with $L\p$-normalization if needed, converges to a nontrivial equilibrium.
\end{theorem}


\nocite{*}  
\bibliographystyle{siam} 
\bibliography{Thesis}

\begin{thebibliography}{10}

\bibitem{absil2005convergence}
{\sc P.-A. Absil, R.~Mahony, and B.~Andrews}, {\em Convergence of the iterates
  of descent methods for analytic cost functions}, SIAM Journal on
  Optimization, 16 (2005), pp.~531--547.

\bibitem{akin1982recurrence}
{\sc E.~Akin and J.~Hofbauer}, {\em Recurrence of the unfit}, Mathematical
  Biosciences, 61 (1982), pp.~51--62.

\bibitem{allen1979microscopic}
{\sc S.~M. Allen and J.~W. Cahn}, {\em A microscopic theory for antiphase
  boundary motion and its application to antiphase domain coarsening}, Acta
  metallurgica, 27 (1979), pp.~1085--1095.

\bibitem{artacho2019boosted}
{\sc F.~J.~A. Artacho, R.~Campoy, and P.~T. Vuong}, {\em The boosted dc
  algorithm for linearly constrained dc programming}, arXiv preprint
  arXiv:1908.01138,  (2019).

\bibitem{artacho2018accelerating}
{\sc F.~J.~A. Artacho, R.~M. Fleming, and P.~T. Vuong}, {\em Accelerating the
  dc algorithm for smooth functions}, Mathematical Programming, 169 (2018),
  pp.~95--118.

\bibitem{attouch2004regularized}
{\sc H.~Attouch and M.~Teboulle}, {\em Regularized lotka-volterra dynamical
  system as continuous proximal-like method in optimization}, Journal of
  optimization theory and applications, 121 (2004), pp.~541--570.

\bibitem{barakat2020convergence}
{\sc A.~Barakat and P.~Bianchi}, {\em Convergence rates of a momentum algorithm
  with bounded adaptive step size for nonconvex optimization}, in Asian
  Conference on Machine Learning, PMLR, 2020, pp.~225--240.

\bibitem{barta2012every}
{\sc T.~B{\'a}rta, R.~Chill, and E.~Fa{\v{s}}angov{\'a}}, {\em Every ordinary
  differential equation with a strict lyapunov function is a gradient system},
  Monatshefte f{\"u}r Mathematik, 166 (2012), pp.~57--72.

\bibitem{BLANCHET20181650}
{\sc A.~Blanchet and J.~Bolte}, {\em A family of functional inequalities:
  Łojasiewicz inequalities and displacement convex functions}, Journal of
  Functional Analysis, 275 (2018), pp.~1650 -- 1673.

\bibitem{bolte2007lojasiewicz}
{\sc J.~Bolte, A.~Daniilidis, and A.~Lewis}, {\em The {\l}ojasiewicz inequality
  for nonsmooth subanalytic functions with applications to subgradient
  dynamical systems}, SIAM Journal on Optimization, 17 (2007), pp.~1205--1223.

\bibitem{10.2307/52420}
{\sc M.~Broom, C.~Cannings, and G.~T. Vickers}, {\em On the number of local
  maxima of a constrained quadratic form}, Proceedings: Mathematical and
  Physical Sciences, 443 (1993), pp.~573--584.

\bibitem{cahn1961spinodal}
{\sc J.~W. Cahn}, {\em On spinodal decomposition}, Acta metallurgica, 9 (1961),
  pp.~795--801.

\bibitem{chapman1992macroscopic}
{\sc S.~J. Chapman, S.~D. Howison, and J.~R. Ockendon}, {\em Macroscopic models
  for superconductivity}, Siam Review, 34 (1992), pp.~529--560.

\bibitem{christlieb2014high}
{\sc A.~Christlieb, J.~Jones, K.~Promislow, B.~Wetton, and M.~Willoughby}, {\em
  High accuracy solutions to energy gradient flows from material science
  models}, Journal of Computational Physics, 257 (2014), pp.~193--215.

\bibitem{colding2014lojasiewicz}
{\sc T.~H. Colding and W.~P. Minicozzi~II}, {\em Lojasiewicz inequalities and
  applications}, arXiv preprint arXiv:1402.5087,  (2014).

\bibitem{colding2017arnold}
\leavevmode\vrule height 2pt depth -1.6pt width 23pt, {\em Arnold-thom gradient
  conjecture for the arrival time}, arXiv preprint arXiv:1712.05381,  (2017).

\bibitem{diakonikolas2021generalized}
{\sc J.~Diakonikolas and M.~I. Jordan}, {\em Generalized momentum-based
  methods: A hamiltonian perspective}, SIAM Journal on Optimization, 31 (2021),
  pp.~915--944.

\bibitem{dinh2014recent}
{\sc T.~P. Dinh and H.~A. Le~Thi}, {\em Recent advances in dc programming and
  dca}, Transactions on computational intelligence XIII,  (2014), pp.~1--37.

\bibitem{elliott1993global}
{\sc C.~M. Elliott and A.~Stuart}, {\em The global dynamics of discrete
  semilinear parabolic equations}, SIAM journal on numerical analysis, 30
  (1993), pp.~1622--1663.

\bibitem{eyre1998unconditionally}
{\sc D.~J. Eyre}, {\em An unconditionally stable one-step scheme for gradient
  systems}, Unpublished article, 6 (1998).

\bibitem{friesecke1999solitary}
{\sc G.~Friesecke and R.~L. Pego}, {\em Solitary waves on {FPU} lattices: I.
  qualitative properties, renormalization and continuum limit}, Nonlinearity,
  12 (1999), p.~1601.

\bibitem{friesecke2002solitary}
\leavevmode\vrule height 2pt depth -1.6pt width 23pt, {\em Solitary waves on
  {FPU} lattices: {II.} linear implies nonlinear stability}, Nonlinearity, 15
  (2002), p.~1343.

\bibitem{friesecke2003solitary}
\leavevmode\vrule height 2pt depth -1.6pt width 23pt, {\em Solitary waves on
  {Fermi--Pasta--Ulam} lattices: {III.} {Howland-type} {Floquet} theory},
  Nonlinearity, 17 (2003), p.~207.

\bibitem{friesecke2004solitary}
\leavevmode\vrule height 2pt depth -1.6pt width 23pt, {\em Solitary waves on
  {Fermi--Pasta--Ulam} lattices: {IV.} proof of stability at low energy},
  Nonlinearity, 17 (2004), p.~229.

\bibitem{friesecke1994existence}
{\sc G.~Friesecke and J.~A. Wattis}, {\em Existence theorem for solitary waves
  on lattices}, Communications in mathematical physics, 161 (1994),
  pp.~391--418.

\bibitem{glasner2016improving}
{\sc K.~Glasner and S.~Orizaga}, {\em Improving the accuracy of convexity
  splitting methods for gradient flow equations}, Journal of Computational
  Physics, 315 (2016), pp.~52--64.

\bibitem{haraux2015convergence}
{\sc A.~Haraux and M.~A. Jendoubi}, {\em The convergence problem for
  dissipative autonomous systems: classical methods and recent advances},
  Springer, 2015.

\bibitem{henry2006geometric}
{\sc D.~Henry}, {\em Geometric theory of semilinear parabolic equations},
  vol.~840, Springer, 2006.

\bibitem{Herrmann2010UnimodalWA}
{\sc M.~Herrmann}, {\em Unimodal wavetrains and solitons in convex
  {F}ermi–{P}asta–{U}lam chains}, Proceedings of The Royal Society A:
  Mathematical, Physical and Engineering Sciences, 140 (2010), pp.~753--785.

\bibitem{hofbauer2003evolutionary}
{\sc J.~Hofbauer and K.~Sigmund}, {\em Evolutionary game dynamics}, Bulletin of
  the American mathematical society, 40 (2003), pp.~479--519.

\bibitem{hoffman2008simple}
{\sc A.~Hoffman and C.~E. Wayne}, {\em A simple proof of the stability of
  solitary waves in the {Fermi-Pasta-Ulam model} near the {KdV} limit}, arXiv
  preprint arXiv:0811.2406,  (2008).

\bibitem{hovsepian2011supervised}
{\sc K.~Hovsepian, P.~Anselmo, and S.~Mazumdar}, {\em Supervised inductive
  learning with lotka--volterra derived models}, Knowledge and information
  systems, 26 (2011), pp.~195--223.

\bibitem{jabin2017non}
{\sc P.-E. Jabin and H.~Liu}, {\em On a non-local selection--mutation model
  with a gradient flow structure}, Nonlinearity, 30 (2017), p.~4220.

\bibitem{jordan1998variational}
{\sc R.~Jordan, D.~Kinderlehrer, and F.~Otto}, {\em {The variational
  formulation of the Fokker--Planck equation}}, SIAM journal on mathematical
  analysis, 29 (1998), pp.~1--17.

\bibitem{kurdyka1994wf}
{\sc K.~Kurdyka and A.~Parusinski}, {\em wf-stratification of subanalytic
  functions and the lojasiewicz inequality}, Comptes rendus de l'Acad{\'e}mie
  des sciences. S{\'e}rie 1, Math{\'e}matique, 318 (1994), pp.~129--133.

\bibitem{lakoba2007generalized}
{\sc T.~I. Lakoba and J.~Yang}, {\em A generalized petviashvili iteration
  method for scalar and vector hamiltonian equations with arbitrary form of
  nonlinearity}, Journal of Computational Physics, 226 (2007), pp.~1668--1692.

\bibitem{law1965ensembles}
{\sc S.~law Lojasiewicz}, {\em Ensembles semi-analytiques}, IHES notes,
  (1965).

\bibitem{le2019convergence}
{\sc U.~Le and D.~E. Pelinovsky}, {\em Convergence of {Petviashvili's} method
  near periodic waves in the fractional {Korteweg--de Vries} equation}, SIAM
  Journal on Mathematical Analysis, 51 (2019), pp.~2850--2883.

\bibitem{le2018dc}
{\sc H.~A. Le~Thi and T.~P. Dinh}, {\em Dc programming and dca: thirty years of
  developments}, Mathematical Programming, 169 (2018), pp.~5--68.

\bibitem{le2018convergence}
{\sc H.~A. Le~Thi, T.~P. Dinh, et~al.}, {\em Convergence analysis of
  difference-of-convex algorithm with subanalytic data}, Journal of
  Optimization Theory and Applications, 179 (2018), pp.~103--126.

\bibitem{le2009convergence}
{\sc H.~A. Le~Thi, V.~Huynh, and T.~Pham~Dinh}, {\em Convergence analysis of dc
  algorithm for dc programming with subanalytic data}, Ann. Oper. Res.
  Technical Report, LMI, INSA-Rouen,  (2009).

\bibitem{li2018calculus}
{\sc G.~Li and T.~K. Pong}, {\em Calculus of the exponent of
  kurdyka--{\l}ojasiewicz inequality and its applications to linear convergence
  of first-order methods}, Foundations of computational mathematics, 18 (2018),
  pp.~1199--1232.

\bibitem{liu2015entropy}
{\sc H.~Liu, W.~Cai, and N.~Su}, {\em Entropy satisfying schemes for computing
  selection dynamics in competitive interactions}, SIAM Journal on Numerical
  Analysis, 53 (2015), pp.~1393--1417.

\bibitem{lojasiewicz1963propriete}
{\sc S.~Lojasiewicz}, {\em Une propri{\'e}t{\'e} topologique des sous-ensembles
  analytiques r{\'e}els}, Les {\'e}quations aux d{\'e}riv{\'e}es partielles,
  117 (1963), pp.~87--89.

\bibitem{lojasiewicz1999gradient}
{\sc S.~Lojasiewicz and M.~Zurro}, {\em On the gradient inequality}, Bulletin
  of the Polish Academy of Sciences-Mathematics, 47 (1999), pp.~143--146.

\bibitem{lunardi2012analytic}
{\sc A.~Lunardi}, {\em Analytic semigroups and optimal regularity in parabolic
  problems}, Springer Science \& Business Media, 2012.

\bibitem{mizumachi2008asymptotic}
{\sc T.~Mizumachi and R.~L. Pego}, {\em Asymptotic stability of {T}oda lattice
  solitons}, Nonlinearity, 21 (2008), p.~2099.

\bibitem{ochs2018local}
{\sc P.~Ochs}, {\em Local convergence of the heavy-ball method and ipiano for
  non-convex optimization}, Journal of Optimization Theory and Applications,
  177 (2018), pp.~153--180.

\bibitem{ochs2014ipiano}
{\sc P.~Ochs, Y.~Chen, T.~Brox, and T.~Pock}, {\em ipiano: Inertial proximal
  algorithm for nonconvex optimization}, SIAM Journal on Imaging Sciences, 7
  (2014), pp.~1388--1419.

\bibitem{palis2012geometric}
{\sc J.~J. Palis and W.~De~Melo}, {\em Geometric theory of dynamical systems:
  an introduction}, Springer Science \& Business Media, 2012.

\bibitem{Pego2018ExistenceOS}
{\sc R.~L. Pego and T.-S. Van}, {\em Existence of solitary waves in one
  dimensional peridynamics}, Journal of Elasticity,  (2018), pp.~1--30.

\bibitem{article}
{\sc D.~Pelinovsky and Y.~Stepanyants}, {\em Convergence of {P}etviashvili's
  iteration method for numerical approximation of stationary solutions of
  nonlinear wave equations}, SIAM J. Numerical Analysis, 42 (2004),
  pp.~1110--1127.

\bibitem{pelinovsky2004convergence}
{\sc D.~E. Pelinovsky and Y.~A. Stepanyants}, {\em Convergence of
  petviashvili's iteration method for numerical approximation of stationary
  solutions of nonlinear wave equations}, SIAM Journal on Numerical Analysis,
  42 (2004), pp.~1110--1127.

\bibitem{petviashvili1976equation}
{\sc V.~I. Petviashvili}, {\em Equation of an extraordinary soliton}, Fizika
  plazmy, 2 (1976), pp.~469--472.

\bibitem{shin2017unconditionally}
{\sc J.~Shin, H.~G. Lee, and J.-Y. Lee}, {\em Unconditionally stable methods
  for gradient flow using convex splitting runge--kutta scheme}, Journal of
  Computational Physics, 347 (2017), pp.~367--381.

\bibitem{simon1983asymptotics}
{\sc L.~Simon}, {\em Asymptotics for a class of non-linear evolution equations,
  with applications to geometric problems}, Annals of Mathematics,  (1983),
  pp.~525--571.

\bibitem{tao1997convex}
{\sc P.~D. Tao and L.~T.~H. An}, {\em Convex analysis approach to dc
  programming: theory, algorithms and applications}, Acta mathematica
  vietnamica, 22 (1997), pp.~289--355.

\bibitem{tao1986algorithms}
{\sc P.~D. Tao et~al.}, {\em Algorithms for solving a class of nonconvex
  optimization problems. methods of subgradients}, in North-Holland Mathematics
  Studies, vol.~129, Elsevier, 1986, pp.~249--271.

\bibitem{wise2009energy}
{\sc S.~M. Wise, C.~Wang, and J.~S. Lowengrub}, {\em An energy-stable and
  convergent finite-difference scheme for the phase field crystal equation},
  SIAM Journal on Numerical Analysis, 47 (2009), pp.~2269--2288.

\bibitem{wu2019general}
{\sc Z.~Wu and M.~Li}, {\em General inertial proximal gradient method for a
  class of nonconvex nonsmooth optimization problems}, Computational
  Optimization and Applications, 73 (2019), pp.~129--158.

\end{thebibliography}

\end{document}